\documentclass[11pt]{amsart}
\usepackage{pdfsync}
\usepackage{graphicx}
\usepackage{color}
\usepackage{amssymb}
\usepackage{enumitem}
\usepackage[all,cmtip]{xy}
\usepackage{hyperref}
\usepackage{comment}
\usepackage{tikz}
\usepackage{tqft}
\newcommand{\R}{\mathbb{R}}
\newcommand{\C}{\mathbb{C}}
\newcommand{\T}{\mathbb{T}}
\newcommand{\Z}{\mathbb{Z}}
\newcommand{\N}{\mathbb{N}}
\newcommand{\K}{\mathbb{K}}
\newcommand{\g}{\Gamma}

\newcommand{\Pp}{\mathbb{P}}

\newcommand{\cF}{\mathcal{F}}
\newcommand{\cP}{\mathcal{P}}
\newcommand{\cS}{\mathcal{S}}
\newcommand{\cD}{\mathcal{D}}
\newcommand{\cH}{\mathcal{H}}
\newcommand{\cJ}{\mathcal{J}}
\newcommand{\cL}{\mathcal{L}}
\newcommand{\cA}{\mathcal{A}}

\newcommand{\restrict[3]}{\left. {#2}\right|_{#3}}

\DeclareMathOperator{\id}{Id}
\DeclareMathOperator{\Log}{Log}

\DeclareMathOperator{\val}{val}

\DeclareMathOperator{\area}{Area}
\DeclareMathOperator{\vol}{Vol}
\DeclareMathOperator{\Sec}{Sec}
\DeclareMathOperator{\inj}{inj}
\DeclareMathOperator{\im}{im}

\DeclareMathOperator{\Skel}{Skel}
\DeclareMathOperator{\supp}{supp}
\DeclareMathOperator{\osc}{osc}
\setlist[enumerate,1]{label={(\alph*)}}
\setlist[enumerate,2]{label={(\roman*)}}

\newtheorem{tm}{Theorem}[section]
\newtheorem*{sa*}{Standing Assumption}

\newtheorem{lm}[tm]{Lemma}
\newtheorem{cy}[tm]{Corollary}

\newtheorem*{tm*}{Theorem}

\theoremstyle{definition}
\newtheorem{df}[tm]{Definition}

\theoremstyle{remark}
\newtheorem{rem}[tm]{Remark}
\newtheorem{ex}[tm]{Example}

\title{Floer theory and reduced cohomology on open manifolds}
\author{Yoel Groman}

\begin{document}

\maketitle
\begin{abstract}
We construct Hamiltonian Floer complexes associated to continuous, and even lower semi-continuous, time dependent exhaustion functions on geometrically bounded symplectic manifolds. We further construct functorial continuation maps associated to monotone homotopies between them, and operations which give rise to a product and unit. The work rests on novel techniques for energy confinement of Floer solutions as well as on methods of non-Archimedean analysis. The definition for general Hamiltonians utilizes the notion of reduced cohomology familiar from Riemannian geometry, and the continuity properties of Floer cohomology. This gives rise in particular to local Floer theory. We discuss various functorial properties as well as some applications to existence of periodic orbits and to displaceability.
\end{abstract}

\tableofcontents

\section{Introduction}
Floer theory is a machine that associates algebraic structures to objects of symplectic geometry. Over the years, it has come to play central role in all aspects of the field, from mirror symmetry to quantitative symplectic topology. In this paper we extend the range of applicability of Floer theory, focusing on Hamiltonian Floer theory, to open symplectic manifolds which are geometrically bounded.

There are a number of reasons why one would be interested in studying Floer theory on open manifolds. First and foremost, many symplectic manifolds which arise naturally are open. Among these we count the cotangent bundle, the magnetic cotangent bundle, affine varieties, coadjoint orbits of non-compact groups, Hitchin moduli spaces, and many more. More specifically related to Floer theory, there are phenomena which only become apparent on open manifolds. For example, on a closed manifold the Hamiltonian Floer cohomology reduces as an Abelian group to the singular cohomology, and is thus often too coarse to see much of the symplectic topology. On open manifolds, new invariants, such as  symplectic cohomology, make their appearance which encode purely symplectic phenomena \cite{CielFloerHofer,Viterbo99, Oancea04,Seidel08}.

There is a vast literature studying these invariants and their structural properties. For example, symplectic cohomology of a Liouville domain has been shown to play a key role in homological mirror symmetry by encoding the Hochschild homology of the Fukaya category \cite{Seidel02, abouzaid2010, ganatra2013, ganatra2019}. In another related direction, there are numerous results relating Floer theory of a symplectic manifold with that of the Floer theory of embedded local models, \cite{Seidel02, CieliebakOancea2018, Varolgunes2018, GPS2019}. But the existing literature focuses mostly on examples which are convex at infinity, a condition which does not cover, e.g., most of the examples mentioned in the previous paragraph. As another example, in the study of the Fukaya category by the method of localization away from a divisor \cite{Seidel02} one does not necessarily wish to restrict attention to ample divisors \cite{AurouxSYZ,Auroux2007,DaemiFukaya2018,Groman2018}. Studying Floer theory in more general settings would contribute to our understanding of mirror symmetry, the geometric Langlands program, and many branches of symplectic topology.

The class of geometrically bounded manifolds contains those that are convex at infinity, but is much larger. Geometric boundededness has appeared as a relevant condition already in Gromov's seminal paper \cite{Gr} and in numerous works since. A couple of early ones are \cite{ALP94,Sikorav94}. Geometrically bounded symplectic manifolds are the most general setting in which holomorphic curve theory is known to work without resorting to the methods of symplectic field theory.

%Geometric boundedness gives rise to $C^0$ estimates for pseudo-holomorphic curves because of the monotonicity inequality (cited as Theorem \ref{tmMonontonicity} below).
The novelty in the present paper is twofold. One is showing how to apply geometric boundedness of the underlying manifold to carry out Floer theoretic constructions beyond $J$-holomorphic curves. Such constructions are central, e.g., to the notion of wrapped Floer theory. The other is showing that the invariants constructed by choosing a  geometrically bounded metric at infinity are independent of the choice. The results here provide a unified and flexible framework which incorporates the various constructions in the literature (e.g.,\cite{CielFloerHofer,Viterbo99,Oancea06,Ritter10}), works in full generality, and has transparent symplectic invariance properties.

It should be emphasized that while we do not mention the Fukaya category elsewhere in this paper, the difficulties posed by  non-compactness are virtually the same for the Hamiltonian version of Floer theory as for its Lagrangian intersection version. Thus, this paper sets the stage for the study of the (wrapped) Fukaya category on open manifolds such as those mentioned above insofar as one can overcome the usual difficulties already present in the closed case.

We shall assume familiarity with the basic machinery of Hamiltonian Floer theory and symplectic cohomology such as can be acquired from the first three lectures in \cite{Salamon1999} together with \cite{Oancea04}. For the discussion of the product structure we shall assume also some familiarity with treatments such as \cite{Abouzaid2013} or \cite{Ritter13}. The latter is not necessary for most of the novel ideas in this paper.
\subsection{The main result}
A symplectic manifold $(M,\omega)$ is said to be \textbf{geometrically bounded} if there is an $\omega$-compatible almost complex structure $J$, a constant $C>1$, and  a complete Riemannian metric $g$ with sectional curvature bounded from above and injectivity radius bounded away from $0$, such that
\[
\frac1{C}g(v,v)\leq \omega(v,Jv)\leq Cg(v,v)
\]
for all tangent vectors $v$. Note that the almost complex structure $J$ is \textit{not} part of the data.
Examples include closed symplectic manifolds, cotangent bundles of arbitrary smooth manifolds, manifolds whose end is modeled after the convex half of the symplectization of a compact contact manifold \cite{Sikorav94}, twisted cotangent bundles \cite{CielGinzKer}, and many more. The class of geometrically bounded symplectic manifolds is closed under products and coverings.

It should be emphasized at the outset that an open symplectic manifold of finite volume, such as the unit ball in $\C^n,$ cannot be endowed with a metric that is at once complete and satisfies the above bounds on sectional curvature and radius of injectivity and thus is \textit{not} geometrically bounded. Floer theory for open finite volume symplectic manifolds, such as Liouville domains, will be discussed, in the context of local Floer theory, when they are embedded in a geometrically bounded symplectic manifold.

Recall that $(M,\omega)$ is said to be \emph{semi-monotone} if there exists a constant $\tau\geq 0$ such that for any $A\in\pi_2(M)$ we have $\omega(A)=\tau c_1(A)$ where $c_1$ is the first Chern class. $(M,\omega)$ is said to be \emph{Calabi-Yau} if $c_1(A)=0$ for every $A\in\pi_2(M)$.  \textit{Henceforth, $(M,\omega)$ is a geometrically bounded symplectic manifold which is either semi-monotone or Calabi-Yau.}  In particular, for any class $A\in\pi_2(M)$ we have
 \[
    c_1(A)<0\qquad\Rightarrow \qquad \omega(A)\leq 0.
 \]
 We hasten to emphasize that this is assumption is made for definiteness only. The methods introduced herein are orthogonal to the usual questions of transversality and can be adapted to any regularization scheme.

Fix a field $R$ and denote by $\Lambda_R$ the universal Novikov field and by $\Lambda_{R,\omega}$ the Novikov field associated with $\omega$ (See \S\ref{SecDefHam}). We shall use the notation $\K$ to denote either $\Lambda_R$ or $\Lambda_{R,\omega}$. Note that $\K$ is a graded field. That is, a commutative even graded field in which every non-zero homogeneous element is invertible. \textit{In the entire text, $\Lambda_R$ coefficients should be assumed by default whenever the coefficient field is not indicated in the notation.}

Denote by $\cJ$ the space of smooth $\R/\Z$ parameterized families of almost complex structures which are compatible with $\omega$. Denote by $\cH_{sm}$ the space of smooth functions on $\R/\Z\times M$ which are proper and bounded from below. Consider the category $\cF$ of Floer data whose objects are elements $(H,J)\in\cH_{sm}\times\cJ$ and there is a single morphism from an object $(H_1,J_1)\to (H_2,J_2)$ whenever the order relation defined by
\[
(H_1,J_1)\leq (H_2,J_2)\quad\Leftrightarrow\quad H_{1,t}(x)\leq H_{2,t}(x),\quad\forall (t,x)\in\R/\Z\times M
\]
is satisfied.

The main contribution of this paper is summarised in the following theorem.
\begin{tm}\label{mainTmA}

There exists a full subcategory $\cF_{d,reg}\subset\cF$, referred to as the \textit{regular dissipative Floer data}, for which the Floer cohomology
\[
(H,J)\mapsto HF^*(H,J;\K),
\]
is well defined as a functor to $\Z$-graded non-Archimedean (semi-)normed $\K$-modules. Namely, there is a functorial norm preserving continuation map
\[
HF^*(H_1,J_1;\K)\to HF^*(H_2,J_2;\K),
\]
whenever $H_1\leq H_2\in \cF_{d,reg}$. The subcategory $\cF_{d,reg}$ satisfies the following
\begin{enumerate}
\item It is invariant under the action of the symplectomorphism group
    \[
    \psi\cdot(H,J)\mapsto (H\circ\psi,\psi^*J).
    \]
\item It is final and cofinal in $\cF$.
\item \label{mainTmA:itC} It contains all pairs $(H,J)$ for which $J$ is geometrically bounded and $H$ has sufficiently small Lipschitz constant and is (nearly) time independent near infinity.
\item The continuation map $HF^*(H_1,J_1;\K)\to HF^*(H_2,J_2;\K)$ is an isomorphism if $H_2-H_1$ is bounded on $M$.
\end{enumerate}
\end{tm}

Theorem \ref{mainTmA} relies on the dissipative method introduced herein for controlling compactness of various Floer moduli spaces. This is done by systematically replacing the more conventional reliance on maximum principles by a combination of the monotonicity inequality for $J$-holomorphic curves and a certain quantitative non degeneracy condition to control the ends. A more detailed discussion of this method is given in \S \ref{secOvrview}. This method should be of independent interest for researchers wishing to apply Floer theory methods in any way to open symplectic manifolds. We emphasize that the Floer data which are typically used in the literature on symplectic cohomology mostly fit into the dissipative framework. For a discussion of the case of Liouville domains see Example \ref{ConvEndDiss} below.

The true power of the dissipative method is revealed when considering the functoriality statement in Theorem \ref{mainTmA} which is one of the main contributions of this paper. To demonstrate this we first remark that a particular consequence of the functoriality is the independence of Floer cohomology of a dissipative $(H,J)$ on the choice of $J$. This is new even for the case $H=0$, i.e., for $J$-holomorphic curves. We use this in Theorem~\ref{TmSympInvGWSH} below to show that in a number of contexts where invariants on open manifolds are defined using a geometrically bounded almost complex structures $J$, the resulting invariants do not depend on the choice of such a $J$. The difficulty in proving this is that given two geometrically bounded compatible almost complex structures which are not metrically equivalent it is not likely they can be homotoped to one another through geometrically bounded almost complex structures. Our solution is to introduce the notion of intermittent boundedness, or, i-boundedness, which requires boundedness only on an appropriate infinite sequence of hypersurfaces. We then show that any two such almost complex structures can be homotoped to one another through intermittently bounded almost complex structures. For the rest of the introduction we thus drop $J$ from the notation and consider dissipativity as a property of a Hamiltonian function.

As an illustration of the use of functoriality we indicate an easy proof of the Kunneth formula in symplectic cohomology of Liouville domains (cf. \cite{Oancea06}). This requires comparing the direct limit of Floer cohomologies over a sequence of linear Hamiltonians on the smoothing of a product of Liouville domains to the direct limit over a sequence of linear split Hamiltonians on the product itself, all with slope going to infinity. Since one can squeeze a sequence of linear Hamiltonians between a sequence of split linear ones, the Kunneth formula follows from Theorem \ref{mainTmA} as soon as one establishes that both linear and split linear Hamiltonians are dissipative. The latter is immediately implied by examples \ref{ExLiouvTame} and \ref{exLinLiou}. The line of argument can be shown to extend to Liouville domains with arbitrary corners.

\subsection{Reduced Floer cohomology for general Hamiltonians}

Our next theorem combines the result of Theorem \ref{mainTmA} with certain continuity properties of Floer cohomology to extend the definition of Floer cohomology to more arbitrary Floer data. Namely, we extend a certain version of Floer cohomology as a functor on the category $(\cH_{d,reg},\leq)$ of regular dissipative Hamiltonians to the category $(\cH_{s.c.},\leq)$ of all generalized lower semi-continuous functions $\R/\Z\times M\to\R\cup{\infty}$ which are proper and bounded from below.  %For this observe first that by the functoriality in Theorem \ref{mainTmA} the Floer homology functor factors through the forgetful functor to the category $\cH_d$ of Hamiltonian $H$ for which there exists a $J$ such that $(H,J)$ is dissipative. We refer to such a Hamiltonians as dissipative. In the following we drop $J$ from the notation and consider Floer homology as a functor of the category of dissipative Hamiltonians.

Before proceeding we introduce the concept of \textit{reduced Floer cohomology} $\overline{HF}^*(H)$ of a non-degenerate dissipative Hamiltonian $H$. The ordinary Floer cohomology $HF^*(H)$ is the homology of a chain complex which is complete with respect to a non-Archimedean norm. Thus the group $HF^*(H;\K)$ is naturally semi-normed. However, in general, the differential needn't have closed image. In such a case $HF^*(H;\K)$ contains non-trivial elements of norm $0$. The reduced Floer cohomology $\overline{HF}^*(H)$ is the quotient of $HF^*(H)$ by the elements of norm $0$. A similar construction is familiar from Riemannian geometry in the context of  $L^2$-cohomology \cite{CGMP1982,Dai2011,Luck2013}. The precise statement of the following result requires some further preparation. We therefore present first an informal statement. Theorem \ref{tmUniExtFl} below is a more precise restatement.
\begin{tm}\label{TmA}
The reduced Floer cohomology functor  $H\mapsto \overline{HF}^*(H;\K)$ extends in a natural way from a functor on the category $\cH_d$ of dissipative Hamiltonians to a functor on the category $\cH_{s.c.}$ of all lower semi-continuous exhaustion functions.  Moreover, $\overline{HF}^*(H;\K)$ arises as the reduced cohomology of a certain non-Archimedean Banach chain complex which is associated to $H$ up to an appropriate notion of quasi-isomorphism, and which specializes to the Floer chain complex for dissipative Hamiltonians.
\end{tm}

Theorem \ref{TmA} employs the method of Floer theory by approximation. This is explained in more detail in section \S\ref{SubsecFlByApp}. Among other things, Theorem \ref{TmA} can be interpreted as saying that at least if one is concerned with reduced cohomology, one needn't worry about the question of whether a given Floer datum is dissipative or not. In a forthcoming note joint with U. Varolgunes we show that Theorem \ref{TmA} actually holds for the unreduced version of Floer cohomology. For a more extensive discussion of this see comment \ref{Commentb} right after the statement of Theorem \ref{tmUniExtFl}.

\subsection{The product structure}
To state the final main theorem we introduce the notion of symplectic cohomology on an arbitrary geometrically bounded symplectic manifold. Let $\cH\subset\cH_{s.c.}$ be a subset consisting of time-independent Hamiltonians such that for any $H_1,H_2\in\cH$ we have that $2\max\{H_1,H_2\}\in\cH$. We call $\cH$ a \textit{monoidal indexing set}. For each monoidal  indexing set $\cH$ we define a group
\[
SH^*(M;\cH):=\varinjlim_{H\in\cH}\overline{HF}^*(M).
\]
The set of monoidal indexing sets is partially ordered by the relation $\cH_1\preceq\cH_2$ which is defined to hold if and only if for any $H\in\cH_1$ there is a constant $C$ and an $H_2\in\cH_2$ such that $H_1\leq H_2+C$. Before proceeding to the statement of the following Theorem we note that there is a decomposition 
\[
SH^*(M;\cH)=\oplus_{\alpha\in[S^1,M]}SH^{*,\alpha}(M;\cH),
\]
where for a free homotopy class $\alpha$, the group $SH^{*,\alpha}(M;\cH)$ arises from the periodic orbits in the homotopy class $\alpha$. In particular, in the following theorem we refer to the subgroups $SH^{*,0}(M;\cH)$ which arise from considering only contractible periodic orbits.
\begin{tm}\label{tmFloerProduct}
\begin{enumerate}
\item For any monoidal indexing set $\cH$, there is a product structure
    \[
    *:SH^{*,0}(M;\cH)\otimes SH^{*,0}(M;\cH)\to SH^{*,0}(M;\cH),
    \]
    which is associative, and supercommutative.
\item The small quantum product on $QH^*(M;\K):=H^*(M;\K)$ is well defined and for any monoidal indexing set $\cH$ there is a natural PSS homomorphism
    \[
    QH^*(M;\K)\to SH^{*,0}(M;\cH),
    \]
    so that the image of $1\in QH^*(M;\K)$ acts as the unit in $SH^{*,0}(M;\cH)$.
    \item
    Given monoidal indexing sets $\cH_1\preceq\cH_2$, the natural continuation maps $SH^{*,0}(M;\cH_1)\to SH^{*,0}(M;\cH_2)$ is a unital algebra homomorphism.
\end{enumerate}
\end{tm}

The proof of Theorem~\ref{tmFloerProduct} is carried out in \S\ref{SecSHuniv}. The restriction to contractible periodic orbits is done for the sake of brevity in the proof. See Remark \ref{remProdContr} below for an extended discussion. Theorem \ref{tmFloerProduct} allows the construction of various flavors of symplectic cohomology as a unital algebra. Essentially the same proof can be used to construct operations associated with any family of nodal Riemann surfaces and parameterized Floer data. Moreover, it is possible to carry out a Lagrangian intersection variant of the results of this paper. Thus, the TQFT structure presented for the case of Liouville domains in \cite{Ritter13} can be transferred in its entirety to the present setting.

\subsection{Organization of the paper}
The paper is organized as follows. Section \S\ref{SecSympCoh} discusses various notions of symplectic cohomology and gives some applications of the main theorems. Section \S\ref{secOvrview} provides an overview of the techniques going into the proof of Theorem~\ref{mainTmA} while Section \S\ref{SubsecFlByApp} is devoted to explaining Theorem \ref{TmA}. Sections \S\ref{Sec1} through \S\ref{Sec5} are devoted to constructing the dissipative Floer data featuring in Theorem~\ref{mainTmA}. The proof of the latter is carried out in \S\ref{Sec6}. In \S\ref{SecHamFloer} we prove Theorems \ref{TmA} (restated as Theorem~\ref{tmUniExtFl}). In \S\ref{subsecFloerProd} we prove Theorem~\ref{tmFloerProduct}. In \S\ref{Sec7} we carry out the proofs of the properties and applications mentioned in  \S\ref{SecSympCoh}.
\subsection{Acknowledgements}
This work subsumes part of the author's PhD thesis carried out in the Hebrew University of Jerusalem.  The author is grateful to the Azrieli foundation for granting him an Azrieli fellowship at that time.

Various stages of the work were carried out at the math departments of Hebrew University of Jerusalem, the ETH of Zurich, and Columbia University. They were supported by the ERC Starting Grant 337560, ISF Grant 1747/13, Swiss National
Science Foundation (grant number 200021$\_$156000) and the Simons Foundation/SFARI ($\#$385571,M.A.).

The author would like to thank his thesis advisor Jake P. Solomon for his guidance throughout his graduate work and for a great number of discussions and suggestions. The author is grateful to Mohammed Abouzaid for suggesting some key ideas and to Paul Biran, Lev Buhovski, Sheel Ganatra, Alex Oancea, Dietmar Salamon, Ran Tessler, Amitai Zernik, Umut Varolgunes and Sara Venkatesh, for useful input and enlightening discussions. In preparing \S\ref{subsecFloerProd} of the current version, the author is greatly indebted to M. Abouzaid and U. Varolgunes for discussions on how to set up Floer theoretic operations for non-commuting Hamiltonians. The discussion in previous versions was erroneous on this count.

Finally, the author would like to thank the anonymous referee for his thorough and detailed review, and numerous helpful suggestions. These have helped catch numerous gaps and have no doubt greatly improved the exposition.

\section{Symplectic cohomology}\label{SecSympCoh}
In the following subsections we use Theorems~\ref{TmA} and~\ref{tmFloerProduct} to construct symplectic cohomology rings and discuss some of their functorial properties and applications. One of the main lessons is that there are two different notions of symplectic cohomology associated with two different topologies one can consider on the colimit of a sequence of Banach chain complexes. The first of these involves completing at the chain level so as to obtain a Banach space. The details at the chain level are described in \S\ref{subsecChainLev}, or, at the cohomology level, in eq. \eqref{eqDfFlRedSec}.  The Banach topology gives rise to local invariants and corresponds under mirror symmetry to locally defined analytic poly-vector fields. Similar constructions have been carried independently in \cite{Venkatesh2017, Varolgunes2018} but the construction has roots back in \cite{CielFloerHofer}. We refer to this as \textit{local symplectic cohomology}.

A second topology one may consider is one in which no completion is applied to the direct limit. As a topological space we consider the direct limit with the \textit{final topology} described in \S\ref{SecSHTopCof}. We refer to this as \textit{global symplectic cohomology}. Global symplectic cohomology is a generalization of the construction in \cite{Viterbo99} (see also \cite[\S3e]{Seidel08} which is more explicit in this regard). It gives rise to global invariants and can be thought to correspond under mirror symmetry to the ring of algebraic poly-vector fields. While this distinction, referred to in \cite{Seidel08} as quantitative vs qualitative symplectic cohomology, has been previously known, its significance appears to have been masked to a large extent in the literature so far due to the emphasis on Liouville domains with trivially valued coefficient fields where various different invariants coincide. In general, different topologies may give rise to completely different vector spaces. For an example of this phenomenon see \cite{Venkatesh2017}.

\subsection{Local symplectic cohomology}
Let $K\subset M$ be a compact set. Let
\[
H_K(x):=\begin{cases} 0, & x\in K,\\ \infty, & x\in M\setminus K.\end{cases}
\]
The \textbf{local symplectic cohomology} at $K$ is defined by
\[
SH^*(M|K;\K):=\overline{HF}^*(H_K;\K).
\]
The following theorem lists the basic properties of $SH^*(M|K;\K)$ which can be almost readily read off Theorems~\ref{TmA} and~\ref{tmFloerProduct}. As before, there is a decomposition
\[
SH^*(M;\cH)=\oplus_{\alpha\in[S^1,M]}SH^{*,\alpha}(M;\cH),
\]
and we denote by $0$ the class of the contractible loops. 
\begin{tm}\label{tmLocFlHoProp}
\begin{enumerate}
\item The map $K\mapsto SH^*(M|K;\K)$ is contravariantly functorial with respect to inclusions.
\item Any symplectomorphism $\psi:M\to M$ induces an isometry
    \[
    \psi_*:SH^*(M|K;\K)\to SH^*(M|\psi(K);\K).
    \]
\item The group $SH^{*,0}(M|K;\K)$ is a unital $\K$-algebra with respect to the operation $*$ induced from the identification $SH^{*,0}(M|K;\K)=SH^{*,0}(M;\{H_K\})$.
\item We have a commutative triangle of $\K$-algebras
\begin{equation}\label{eqLocComTri}
\xymatrix{
H^*(M;\K)\ar[d]\ar[dr]& \\
SH^{*,0}(M|K_2;\K) \ar[r] &SH^{*,0}(M|K_1;\K)
}
\end{equation}
\item \label{tmLocFlHoPropPe} For any $H\in\mathcal{H}$ which is bounded on $K$ we have a continuous map and functorial map
\[
\overline{HF}^*(H;\K)\to SH^*(M|K;\K),
\]
which increases the valuation\footnote{By definition, the valuation is $\val:=\log\|\cdot\|.$} by at most $c=\sup_KH$.
\end{enumerate}
\end{tm}
The proof of Theorem \ref{tmLocFlHoProp} appears at the end of \S\ref{subsecFloerProd}.
\begin{rem}
Suppose $M$ is symplectically aspherical and for a pair of compact sets $K=K_1,K_2$, we have that $H_K$ can be approximated by Hamiltonians whose non-trivial periodic orbits have action positive and bounded away from $0$. Then the commutative triangle \eqref{eqLocComTri} can be refined to a commutative square
\[
\xymatrix{
H^*(K_2;\K)\ar[d]\ar[r]&H^*(K_1;\K)\ar[d]& \\
SH^*(M|K_2;\K) \ar[r] &SH^*(M|K_1;\K).
}
\]
Combined with \eqref{eqLocLiouvDom} below, this reproduces Viterbo's commutative square for Liouville domains\cite{Viterbo99}.
\end{rem}
\begin{rem}
We comment on the name local symplectic cohomology. Assume $M$ is symplectically aspherical and the boundary of $K$ is stable Hamiltonian. Then it can be shown that elements of $SH^*(M|K)$ are represented by linear combinations of constant periodic orbits inside $K$ and the Reeb orbits of $\partial K$. If the boundary is not stable Hamiltonian, these can be represented by, in addition to constant orbits inside, periodic orbits lying arbitrarily close to $\partial K$.  Thus, at least when $M$ is aspherical, $SH^*(M|K)$ can be thought of as symplectic cohomology relative to the complement of $K$. That is, as localized at $K$. When $M$ is not aspherical, this type of locality is far from clear. This question is taken up in forthcoming work. Theorem~\ref{lmNearbyEx}  below can be seen as a particular manifestation of locality in the general case.
\end{rem}

\begin{tm}\label{lmNearbyEx}
Let $H$ be a smooth Hamiltonian such that $H^{-1}(0)=\partial K$.
\begin{enumerate}

\item
Suppose $\alpha$ is a non-trivial free homotopy class of loops. Suppose $SH^{*,\alpha}(M|K;\K)\neq 0$. Then there is a sequence $a_n>0$ converging to $0$ such that $H^{-1}(a_n)$ has a periodic orbit representing $\alpha$.
\item\label{lmNearbyExPartB}
If $SH^{*,0}(M|K;\K)\neq H^*(K;\K)$ then there is a sequence $a_n>0$ converging to $0$ such that $H^{-1}(a_n)$ has a contractible periodic orbit.
\end{enumerate}
\end{tm}
Theorem \ref{lmNearbyEx} is proven in \S \ref{subsecNearby}.
\begin{comment}
\begin{tm}\label{tmCircleAcVan}
Suppose $c_1(M)=0$ and that $M$ carries a circle action generated by a Hamiltonian $H$ which is proper and bounded from below. Then for any sublevel set $U$ of $H$, we have that
\[
SH^*(U;\K)=0.
\]
\end{tm}
Combining theorems \ref{lmNearbyEx}\ref{lmNearbyExPartB} and \ref{tmCircleAcVan} we obtain
\begin{cy}[\textbf{Obstruction to circle action}]
Suppose $c_1(M)=0$. If there exists a Hamiltonian $H$ which has no contractible periodic orbits outside some compact set then $M$ does not carry a Hamiltonian circle action generated by a Hamiltonian which is proper and bounded from below.
\end{cy}
In \cite{a,b} numerous examples are given of magnetic cotangent bundles having no contractible periodic orbits on levels beyond the Mane critical value.
\end{comment}
Some applications of local Floer cohomology to embedding and displaceability problems are given in \S \ref{subsecLiouIntro} below.

We conclude with some comments on the relation of these groups with similar work of others.
\begin{enumerate}[wide, labelwidth=!, labelindent=10pt]
\item When $M$ is symplectically aspherical and $K$ is the closure of an open set $U$, the  groups $SH^*_{[a,b)}(M|K)$ coincide with the corresponding symplectic cohomology groups of $U$ as defined in \cite{CielFloerHofer} using Hamiltonians which are constant at infinity.
\item In \cite{Venkatesh2017} the notion of completed symplectic cohomology is introduced and studied for Liouville cobordisms $\mathcal{W}$ inside monotone symplectic manifolds. The computations in \cite{Venkatesh2017} show that the local symplectic cohomology groups depend non trivially on $K$. The choice of Floer data in \cite{Venkatesh2017} is such that the Floer chain complexes have finite boundary depth (see Remark \ref{remBoundaryDepth}). In particular ordinary and reduced Floer cohomology coincide for these Floer data.  A consequence of Theorem~\ref{tmUniExtFl} is that the invariant of \cite{Venkatesh2017} is the local Floer cohomology as defined here.
\item In \cite{Varolgunes2018} an invariant which is closely related to local symplectic cohomology as defined here is studied and is shown to fit into a local to global spectral sequence when the compact sets involved are invariant sets of an involutive system of Hamiltonians.
\end{enumerate}
\subsection{Global symplectic cohomology}

Consider the set $\cH_{univ}\subset\mathcal{H}_{s.c.}$ of \textit{smooth} time independent exhaustion functions on $M$. Then $\cH_{univ}$ is a monoidal indexing set. By Theorem \ref{tmFloerProduct} we thus obtain for any geometrically bounded symplectic manifold a  $k$-algebra which is a symplectic invariant.
\begin{df}
The \textit{universal symplectic cohomology} is defined by
\[
SH^*_{univ}(M):=SH^*(M;\mathcal{H}_{univ}).
\]
\end{df}

For any choice of $\cH$ the algebra $SH(M;\cH)$ carries a topology, called the \textit{final topology} as a direct limit of topological vector spaces. This topology is not guaranteed a priori to be Hausdorff and its Hausdorff completion is not guaranteed to be metrizable. However, in the few cases where something is known about it, $SH^*_{univ}$ turns out to be a reasonably well behaved object. Example \ref{exSUniv} below should give some sense of what universal symplectic cohomology is like in nice cases.% Better behaved topological vector spaces are obtained when one considers monoidal indexing sets $\mathcal{H}$ with appropriate restrictions on the growth at infinity. %Such monoidal indexing sets play a role in numerous settings considered in the literature \cite{Viterbo99,Ritter10}.

Note that for a compact set $K$ we have $\cH_{univ}\preceq\cH_K$. Thus there is a natural unital map $SH^*_{univ}(M)\to SH^*(M|K)$ for any compact set. One way to utilize it is if one finds a monoidal indexing set $\cH\subset\cH_{univ}$ for which $SH^*(M;\cH)=0$, it then follows that $SH^*(M|K)=0$ for all compact $K\subset M$. Observe that since we do not complete after taking the direct limit, the algebra $SH^*(M;\cH)$ is unsensitive to behavior on compact sets. Indeed, define an equivalence relation $\cH_1\sim \cH_2$ by $\cH_1\preceq\cH_2$ and $\cH_2\preceq 1$, then the associated symplectic cohomologies are canonically isomorphic. On the other hand, if $\cH$ consists of continuous functions, the $\sim$-equivalence class of $\cH$ is unaffected by any alterations on any compact set. Thus, for $\cH\subset\cH_{univ}$ the algebra $SH^*(M;\cH)$ is only sensitive to the growth at infinity. For this reason we refer to this type of symplectic cohomology as global $SH$.

Before applying $SH^*_{univ}(M)$ we discuss some settings where something can be said about it.

Let $(M,\omega)$ be a compact symplectic manifold and let $\psi:M\to M$ be a symplectomorphism. Denote by $\tilde{M}_\psi$ the associated symplectic mapping torus. See \S\ref{subSecMappingTorus} for the definition. Denote by $HF^*(M,\psi)$ the fixed point Floer homology of $\psi$ as introduced in \cite{DS94}. The following theorem allows to distinguish mapping tori by fixed point Floer homology. $\tilde{M}_\psi$ carries a distinguished closed $1$-form $dt$ pulled back from the natural map $\tilde{M}_\psi\to S^1$. The $1$-form $dt$ induces a grading on $SH^*_{univ}$ since continuation maps are homotopies. We denote by $SH^{*,k}_{univ}(\tilde{M}_\psi)$ the $k$th graded piece.

\begin{tm}[Cf.  \cite{Fabert10}]\label{tmMappFloFixed}
There is a map
\[
\oplus_{k\in\Z }HF^*(M,\psi^k)\to {SH}^{*}_{univ}(\tilde{M}_\psi),
\]
which is injective and dense. Moreover, for each $k \in\Z$ there is a natural isomorphism
\[
HF^*(M,\psi^k)=SH^{*,k}_{univ}(M).
\]

%$\widehat{SH}^*_{univ}(\tilde{M}_\psi)\simeq\hat{\oplus}_{k\in\Z }HF^*(M,\psi^k).$ Here the right hand side is the completion with respect to the $\sup$ norm and $\simeq$ denotes isomorphism of the underlying topological vector space.
In particular, let $\psi_i:M\to M$ be a symplectomorphism for $i=0,1$. Suppose there exists a symplectomorphism
\[
 \phi:\tilde{M}_{\psi_1}\to\tilde{M}_{\psi_2}
 \]
 which preserves the class of $dt$. Then $\phi$ induces an isomorphism $HF^*(M,\psi_1)=HF^*(M,\psi_2)$.
\end{tm}
Theorem \ref{tmMappFloFixed} is proven in \S \ref{subSecMappingTorus}.

\begin{tm}[\textbf{Kunneth formula}]\label{tmKunneth}
Let $M_1$ and $M_2$ be geometrically bounded symplectic manifolds. Then there is a natural map
\begin{equation}\label{eqKunnethSH}
SH_{univ}^*(M_1)\otimes SH_{univ}^*(M_2)\to SH_{univ}^*(M_1\times M_2)
\end{equation}
which is injective with dense image. A similar claim holds if one restricts to $SH^{*,0}_{univ}$.
\end{tm}
Theorem \ref{tmKunneth} is proven in \S \ref{subsecKunnethSH}.

The following Theorem refers to the additional grading on $SH_{univ}^{*}(M)$ by free homotopy classes of loops  as dicussed in the paragraph preceding Theorem \ref{tmFloerProduct}.

\begin{tm}[\textbf{Nearby existence}]\label{tmNearbyEx0}
\begin{enumerate}

\item\label{NearbyEx1}
Suppose $SH_{univ}^{*,0}(M)=0$. Then for any Hamiltonian $H:M\to\R$ which is proper and bounded from below, the subset of levels containing a contractible periodic orbit is dense in $H(M)\subset\R$. The claim holds also if we merely assume $SH_{univ}^{*,0}(M)=\overline{\{0\}}$, the closure of $0\in SH^{*,0}_{univ}$ with respect to the final topology on $SH^{*,0}_{univ}$.
\item\label{NearbyEx2} Suppose $\alpha$ is a non-trivial free homotopy class of loops. Suppose $SH^{*,\alpha}(M)\neq\{0\}$. Then there is a compact $K\subset M$ such that for any smooth proper and bounded below $H:M\to\R$ and any $a\in \R$ for which  $H(K)\subset(-\infty,a]$ the set of $x\in[a,\infty)$ for which $H^{-1}(x)$ has a periodic orbit representing $\alpha$ is dense in $[a,\infty)$.
\end{enumerate}
\end{tm}
Theorem \ref{tmNearbyEx0} is proven in \S\ref{subsecNearby}.

\begin{rem}
Examples satisfying the hypotheses of the first part of Theorem \ref{tmNearbyEx0} are complete toric varieties $M$ such that $c_1(M)=0$. This follows from the vanishing criterion of Theorem \ref{tmVanishing}. See Example \ref{exTCy}. There are manifolds in this class of examples which unlike $\C$ contain non-displaceable sets. Examples are the canonical bundles over $\Pp^2$ and over $\Pp^1\times \Pp^1$. By the Kunneth formula, the product of such a manifold with any geometrically bounded symplectic manifold of vanishing Chern class will again satisfy the hypothesis.

An example of an $M$ and $\alpha$ satisfying the hypotheses the second part of Theorem \ref{tmNearbyEx0} is given by the cotangent bundle of the torus and any non-trivial homotopy class $\alpha$. This can be deduced from Theorems \ref{tmMappFloFixed} and \ref{tmKunneth}. From this  we obtain many examples by taking the product with an arbitrary compact manifold or with a geometrically bounded one for which symplectic cohomology does not vanish, and considering homotopy classes pulled from the cotangent factor.
\end{rem}

We can also use the methods of this paper to produce periodic orbits with prescribed action. Namely, for a dissipative Hamiltonian $H$ call a class $a\in \overline{HF}^*(H)$ \textbf{essential} if it maps to a non-zero class in $SH_{univ}^*(M)$. Suppose $M$ is symplectically aspherical. If $H_1\leq H_2$ are dissipative then for any essential class $a$ in $HF^*(H_1)$ there is a periodic orbit of $H_2$  in the same homotopy class with action bounded by $\val(a)$. Indeed, the map $HF^*(H_1)\to SH^*_{univ}(M)$  factors through $HF^*(H_2)$ by the continuation map which is action decreasing.

\begin{ex}
On a Liouville domain, for any function $H$ which is convex at infinity, all non-zero classes in $HF^*(H)$ are essential. This follows from Theorem \ref{tmVitrUniv} below. The same holds for the product of a Liouville domain with a compact a spherical manifold. These claims require working over $R$ instead of over $\Lambda_R$, but this is not problematic in this restricted setting since the action spectrum is bounded below and so the topology is discrete.
\end{ex}

\subsection{Liouville domains and displaceability}\label{subsecLiouIntro}
\begin{comment}
We make a few comments about how various constructions specific to Liouville domains fit into the picture just sketched. It turns out, perhaps surprisingly, that it matters very much whether we work over a trivially valued field $R$ or over a Novikov field which for definiteness we take as $\Lambda_R$.
\end{comment}
Let $M$ be the completion of a Liouville domain $U$. Denote by $SH^*_{Viterbo}(U;\K)$ the symplectic cohomology as defined in \cite{Viterbo99} by taking a direct limit of the Floer cohomology groups $HF^*(H,J)$ over all $(H,J)$ where $H$ is linear at infinity and $J$ is of contact type.  See \S\ref{subSecLiouville} for notation and definitions. Denote by $\mathcal{L}\subset\mathcal{H}$ the set of Hamiltonians which are linear at infinity. Then $\cL$ is a monoidal indexing set. We have
\[
SH^*_{Viterbo}(U;\K)=SH^*(M;\mathcal{L},\K),
\]
and therefore a natural map
\[
f:SH^*_{Viterbo}(U;\K)\to SH^*_{univ}(M;\K).
\]
We prove in Theorem~\ref{VitUnivEmb} below
\begin{tm}\label{tmVitrUniv}
The map $f$ is an isomorphism for $\K=R$ coefficients.
\end{tm}
\begin{cy}
For a Liouville manifold $M$ of finite type, $SH_{Viterbo}^*(M;R)$ is independent of the choice of primitive of $\omega$.
\end{cy}

\begin{comment}
A similar setting is that of negative symplectic fibrations over compact aspherical manifolds considered by \cite{Oancea07}. It is no hard to show that the same holds. A corollary is that these constructions of symplectic cohomology are truly symplectic invariants, independent of the contact structure at infinity and the fibration.
\end{comment}
\begin{rem}
Theorem \ref{tmVitrUniv} can generally \textit{not} be expected to be true over a non-trivially valued field. See Remark \ref{rmTrivNonTriv} for an explanation on this point.
\end{rem}
It is also not hard to show that for any Liouville sub-domain $V\subset M$ we have a natural isomorphism of vector spaces
\begin{equation}\label{eqLocLiouvDom}
SH^*(M|V;R)=SH^*_{Viterbo}(V;R).
\end{equation}
Note however that the left hand side of \eqref{eqLocLiouvDom} is naturally a normed vector space while the right hand side is not. The equation will thus cease to be  true over a non-trivially valued field. The generalization of \eqref{eqLocLiouvDom} for the non-trivially valued case is the following excision principle
\begin{equation}\label{eqLocLiouvDomNT}
SH^*(M|V;\Lambda_R)=SH^*(\hat{V}|V;\Lambda_R)
\end{equation}
whenever $M$ is a Liouville manifold and $V$ is a Liouville subdomain with $\hat{V}$ its completion. This follows by the no escape Lemma near the concave boundary of $M\setminus V$. See \cite{Ritter13}. We now formulate a theorem showing that this independence of the ambient manifold holds under more general conditions for skeleta of Liouville domains. In the following, we denote by $SH^{*,0}(M|V)$ the subgroup consisting of periodic orbits that are contractible in $M$. The proofs of the following Theorems as well as some pertinent definitions are given in \S \ref{subSecLiouville}.
\begin{tm}\label{tmSkellVitFunc}
Let $M$ be symplectically aspherical and let $U$ be a Liouville domain with Liouville field $Z$. Let $i:U\to M$ be an embedding with the property that $i_*:H_1(U;\R)\to H_1(M;\R)$ is injective. Then, denoting  by $\Skel(U,Z)$ the skeleton of $U$ with respect to $Z$, 
\[
SH^{*,0}(M|\Skel(U,Z);\mathbb{K})=SH^{*,0}(U|\Skel(U,Z);\mathbb{K}).
\]
%In particular, if $SH^*_{univ}(M)=0$, there can be no such embedding.
\end{tm}
\begin{rem}
The restriction to contractible periodic orbits in Theorem \ref{tmSkellVitFunc} can be removed by adding the assumption that $M$ is symplectically atoroidal.
\end{rem}
Theorem \ref{tmSkellVitFunc} implies
\begin{tm}\label{tmLiouSubNondisp}
Let $U,Z$ and $M$ be as in Theorem \ref{tmSkellVitFunc} and suppose
\begin{equation}\label{tmNonDispNeq0}
SH^*_{Viterbo}(U)\neq 0.
\end{equation}
Then $\Skel(U,Z)$ is not displaceable.  %In particular, no simply connected Liouville domain satisfying~\eqref{tmNonDispNeq0} can be embedded in the product $\C\times M$ where $M$ is symplectically aspherical and geometrically bounded.
\end{tm}

Taking $U$ the cotangent disc bundle, this is a well known theorem by Gromov. Namely, $\C^n$ contains no simply connected Lagrangians. The particular case $M=\hat{U}$, the completion of $U$, is a theorem by \cite{kang2014}. We remark that Theorem \ref{tmLiouSubNondisp} follows from Theorem \ref{tmSkellVitFunc} by a general vanishing principle for the localized Floer cohomology of a displaceable set. We prove this for $M$ aspherical. In \cite{Varolgunes2018} this is proven without the asphericity assumption. Note however that the asphericity assumption in the last two theorems cannot be removed. Indeed, there are examples of displaceable Lagrangian spheres \cite{Abouzaid2012,Pabiniak2015}. However, there are quantitative counterparts which should hold assuming essentially only geometric boundedness.
\begin{tm}\label{tmEpsLiouEmb}
Let $M$ be monotone or Calabi-Yau and let $U\hookrightarrow M$ be a Liouville domain. Then there is a $\delta>0$ for which $SH^{*,0}_{Viterbo}(U;R)$ embeds into $SH^{*,0}_{[0,\delta)}(M|\Skel(U,Z);\K)$ with valuation $0$ as an $R$-subspace.
\end{tm}

\begin{tm}\label{tmPosDis}
Let $M$ be aspherical and let $U\hookrightarrow M$ be a Liouville domain satisfying $SH^*_{Viterbo}(U)\neq 0$. Then $\Skel(U,Z)$ has positive displacement energy.
\end{tm}

\begin{rem}
It should not be hard to remove the asphericity assumption. Once this is done and taking $U$ to be the cotangent disk bundle we recover a classical theorem by Chekanov \cite{Chekanov1998} stating that Lagrangian submanifolds have positive displacement energy.
\end{rem}

\section{Overview}
\subsection{Diameter control of Floer trajectories}\label{secOvrview}
In the next couple of sections we wish to investigate the conditions under which a Floer datum $F\in\mathcal{F}$ gives rise to Floer homology groups. What it comes down to are conditions under which Gromov compactness holds. To sketch an outline of what is to come, let us first discuss how compactness might fail. Fix the coordinates $(s,t)$ on $\R\times \R/\Z$. Let $u_n: \R\times \R/\Z$ be a sequence of solutions to Floer's equation
\begin{equation}\label{eqFloer}
\partial_su+J(\partial_tu-X_{H})=0
\end{equation}
satisfying for some positive number $E$ and some compact set $K\subset M$,
\[
E(u):= \frac1{2}\int\|\partial_su\|^2\leq E,\qquad u(\R\times \R/\Z)\cap K\neq\emptyset.
\]
In general there are two ways in which such a sequence may diverge. First there might be (after possibly reparametrizing) a fixed value $s$ and a compact set $K'$ such that $u_n(s,\cdot)$ intersects $K'$ but the diameter of $u_n(s,\cdot)$ is not bounded uniformly in $n$. We refer to this as a divergence of type $1$. See the left side of Figure \ref{figTypeDiv}. Second, there might be a sequence $s_n\to\infty$ such that $u_n(s_n,\cdot)$ converges to infinity. This is referred to as type $2$ divergence on the right of Figure \ref{figTypeDiv}.
\begin{figure}
\includegraphics[scale=0.6, trim=0 500 0 100,clip]{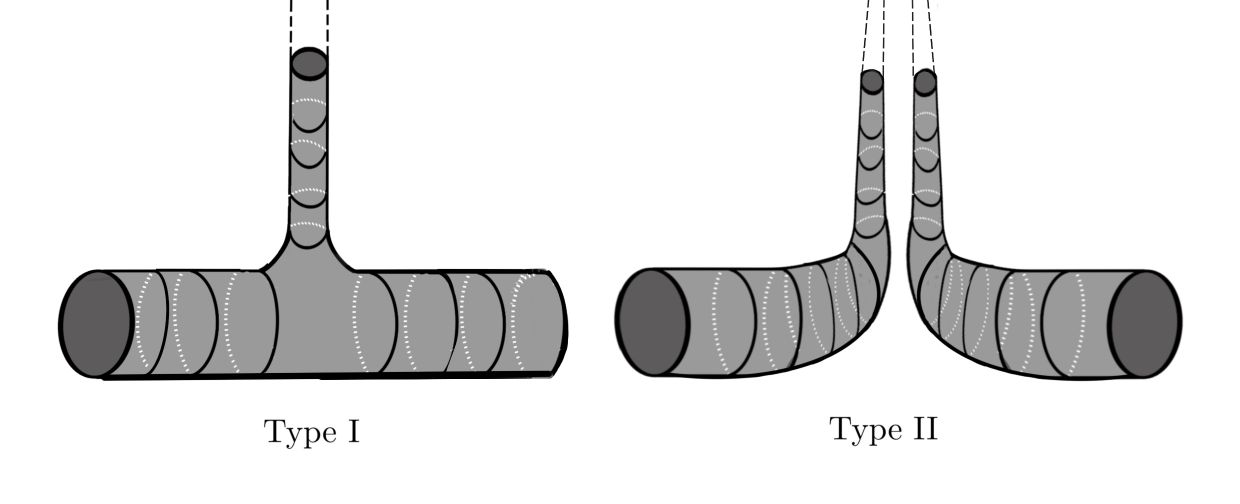}
\caption{Two types of divergence}\label{figTypeDiv}
\end{figure}

In the text below we introduce two conditions, one for ruling out each type of divergence. For the first type of divergence we introduce the condition of intermittent boundedness, or, i-boundedness. It involves bounds on the geometry of an associated metric on the Hamiltonian mapping torus which are required to hold on a sufficiently large subset of $M$. This condition is introduced first for the case where $H=0$ in \S \ref{Sec1} where we show that it provides diameter control for pseudo holomorphic curves. The condition of i-boundedness is framed so as to allow homotopies between any two elements, as well as higher homotopies, for which the diameter estimate continues to hold. This is the content of Theorem \ref{tmWeakTameCont}. Note that it is not reasonable to expect that any two almost complex structures which induce a geometrically bounded metric are connected by a path of the same kind of almost complex structures. Figure \ref{figZigZag} illustrates the kind of homotopy intermittent boundedness allows.

So far, the discussion only pertains to pseudo-holomorphic curves. In \S \ref{Sec3} we discuss a trick which allows us to obtain the same diameter control for $H$ non-zero provided we restrict attention to fixed compact sets of the domain. When $H$ is non-zero, we are considering a geometry which is determined not just by $J$ but also by $H$. Most of \S\ref{Sec3} is devoted to studying the geometry of this metric.

To rule out the divergence of type $2$ we introduce a condition called loopwise dissipativity. It is a variant of the Palais-Smale condition which has played a role in early variational arguments for existence of symplectic capacities \cite[Ch. 12]{MS1}. This condition is not contractible, but this is not a problem since it only needs to be satisfied on the ends. In this it is similar to the non-degeneracy condition which is usually required in Floer theory. Note that unlike the property of i-boundedness, the property of loopwise dissipativity is not readily verifiable on non-exact submanifolds for Hamiltonians that do not have a small Lipschitz constant. In those cases it requires some understanding of the Hamiltonian flow.

Floer data satisfying these conditions are called \textbf{dissipative}. Theorem \ref{tmFloSolDiamEst} states that dissipative Floer data satisfy a priori $C^0$ estimates. A variant which works under a slightly weaker condition on exact symplectic manifolds is given in Theorem \ref{tmPalaiSmaleDiss}.

We discuss three classes of examples of dissipative $(H,J)$.
\begin{enumerate}
\item $H$ is Lipschitz with respect to $g_J$ with sufficiently small Lipschitz constant outside of a compact set. More generally, mainly to allow a cofinal set, we  require the Lipschitz condition only on a sufficiently large subset of $M$. This class of examples is sufficient for all the theoretical constructions of this paper.
\item $M$ is exact and the action functional satisfies the Palais-Smale condition. For the details see \S\ref{SubSec52}. Strictly speaking, as noted in the beginning of \S\ref{SubSec52}, the Palais-Smale condition is slightly weaker than the dissipativity condition. Nevertheless it fits into the general dissipative framework.    
\item The Hamiltonian flow of $H$ is sufficiently close to being invariant with respect to a radial parameter. See \S\ref{SubsecNonex}.
\end{enumerate}
\subsection{Floer theory by approximation}\label{SubsecFlByApp}
This subsection is devoted to clarifying the statement of Theorem \ref{TmA} and its underlying philosophy.
We first discuss the notion of reduced Floer cohomology. The Floer cohomology associated by Theorem \ref{mainTmA} to a dissipative Floer datum $(H,J)$ is the homology of the Floer complex $CF^*(H,J)$ constructed in \S\ref{SecDefHam}. The complexes $CF^*(H,J)$ can be considered as non-Archimedean Banach spaces over $\Lambda_R$ as we explain momentarily. The chain complex $CF^*(H,J)$ is generated by an appropriate Novikov covering of the space of $1$-periodic orbits of $H$. Thus Floer cohomology can be considered as the Morse cohomology of a single valued action functional $\cA_H$. Our conventions are set up so that action decreases along gradient lines. See \S\ref{SecFloerGrom} for precise definitions. Thus our chain complexes carry a decreasing filtration by $\cA_H$. Moreover, the continuation maps of Theorem \ref{TmA} are induced by certain chain maps which preserve this filtration.

$CF^*(H,J)$ is thus normed with norm given by \eqref{NonArchNormDef}. The fact that the differential and continuation maps are action decreasing means they are bounded with respect to this norm, and, in particular, continuous. On an open manifold, $CF^*(H,J)$ will typically not be finitely generated over the Novikov ring. For the differential and continuation maps to be well defined we need to consider the completed complex $\widehat{CF}^*(H,J).$  Moreover, the differential can generally not be expected to have a closed image.
\begin{df}
Let $(C^*,d)$ be a normed complex. The \textbf{reduced cohomology} of $C^*$ is
\[
\overline{H}^*(C^*,d):=\widehat{\ker {d^{*}}}/\overline{\im d^{*-1}},
\]
with the hat denoting completion with respect to the norm, and the overline denoting the closure inside the completion. For a dissipative $H$, we denote the reduced Floer cohomology by $\overline{HF}^*(H)$\footnote{We suppress $J$ in the notation since the homology is independent of $J$ as a consequence of part~\ref{mainTmA:itC}  of Theorem~\ref{mainTmA}.}.
\end{df}

\begin{rem}\label{remBoundaryDepth}
When the Floer complex is finitely generated over a field, the differential has closed image. So, in that case, reduced Floer cohomology coincides with ordinary Floer cohomology. The same is true whenever the Floer complex has finite boundary depth, meaning that the differential has a bounded right inverse \cite{Usher2011}. For Liouville domains, the Floer differential for a strictly convex Hamiltonian has a closed image if one is working over $R$, but not necessarily when working over $\Lambda_R$.
\end{rem}

Denote by $\cH^{\N,\geq}_{d,reg}$ the set of sequences $\{H_i\}$ of regular dissipative Hamiltonians satisfying $H_i(x)\leq H_{i+1}(x)$ for all $i$ and for all $x\in\R/\Z\times M$. The set $\cH^{\N,\geq}_{d,reg}$
carries a natural order relation. Namely, $\{H^1_i\}\leq \{H^2_i\}$ is defined to hold if and only if for any $i$ there is a $j$ such that $H^1_i\leq H^2_j$. General nonsense about filtered complexes leads to a certain extension of the functor $\overline{HF}^*$ to the category $\cH^{\N,\geq}_{d,reg}$ as follows. For a dissipative Floer datum $(H,J)$ and an interval $[a,b)\subset\R$ we can consider the \emph{action truncated Floer cohomology} $HF^*_{[a,b)}(H)$. See \eqref{eqFiltCF} for its definition. Given intervals $[a,b), [a',b')$ such that $a'\leq a$ and $b'\leq b$ there is a natural map  $HF^*_{[a,b)}(H)\to HF^*_{[a',b')}(H) $ which behaves functorially with respect to continuation maps. We then define 
\begin{equation}\label{eqDfFlRedSec0}
\overline{HF}^*(\{H_i\}):=\varprojlim_a\varinjlim_{b,i}HF^*_{[a,b)}(H_i).
\end{equation}
The motivation behind this definition will be clarified in Theorem \ref{tmUniExtFl} and the comments following it. 

On the other hand, by Dini's Theorem from basic calculus, there exists a functor, that is, an ordered map,
\[
\sup:(\cH^{\N,\geq}_{d,reg},\leq)\to (\cH_{s.c},\leq),
\]
which takes $\{H_i\}$ to the function $x\mapsto \sup_iH_i(x)$.

\begin{tm}\label{tmUniExtFl}
The map $\sup$ is surjective. Moreover, if
$\sup(\{H_i^1\})=\sup(\{H_i^2\})$, there is a natural isomorphism
\begin{equation}\label{eqEquivRedIso}
\overline{HF}^*(\{H_i^1\})=\overline{HF}^*(\{H_i^2\}).
\end{equation}
In particular, there is an induced Floer cohomology functor $\overline{HF}^*$ from the category $\cH_{s.c}$ to the category of $\Z$-graded non-Archimedean Banach spaces over $\K$. This definition of $\overline{HF}^*$  coincides on the subcategory $\cH_{d,reg}\subset\cH_{s.c}$ with the previous definition which is implied by Theorem \ref{mainTmA}.
\end{tm}

Theorem~\ref{tmUniExtFl} is proved towards the end of subsection \S\ref{SubsecLscFC} right before Lemma \ref{lmOverHFindep}.

Let us unpack the meaning of  Theorem~\ref{tmUniExtFl}.
\begin{enumerate}[wide, labelwidth=!, labelindent=10pt]

\item \label{comments7} Theorem~\ref{tmUniExtFl} allows one to talk about reduced Floer cohomology of a smooth proper exhaustion Hamiltonian $H$ \emph{without first establishing that $H$ is dissipative}. The further extension to lower semi-continuous functions is of interest since the characteristic function of an open set is lower semi-continuous. This is used in the discussion of local Floer cohomology of compact sets (defined as Floer cohomology of the characteristic function of the complement).
\item  The heart of the proof of the isomorphism \eqref{eqEquivRedIso} is Theorem \ref{lmConstApp} which can be interpreted as saying the \emph{truncated} Floer cohomology is continuous with respect to \textit{convergence on compact sets}. %\footnote{We hasten to emphasize that what is novel to this paper, beyond semantics, is not that the proper domain of definition of the truncated Floer cohomology functors is $\cH_{s.c}$. This has long been known to experts. What is new is the observation of continuity with respect to convergence on compact sets which implies that the truncated Floer cohomology of a lower semi-continuous functions is independent of how it is approximated near infinity. As has been commented already, this point is of interest even for smooth Hamiltonians.}. 
This continuity is a consequence of the quantitative nature of our main $C^0$ estimate Theorem \ref{tmFloSolDiamEst}. Namely, Floer trajectories connecting regions that are far apart must have high energy. Thus, for fixed action truncation, regions that are sufficiently far apart don't interact Floer theoretically.

 This continuity statement is not true for the \emph{reduced} Floer cohomology $\overline{HF}^*(H)$. Indeed, it is easy to construct examples of a monotone sequence $H_i$ of regular dissipative Hamiltonians converging  on compact sets to a regular dissipative Hamiltonian $H$ for which
    \[
    \lim_i\overline{HF}^*(H_i)\neq\overline{HF}^*(H)=\overline{HF}(\{H_i\}).
    \]
The discrepancy between the leftmost side and rightmost side in the last equation arises because of the interchange of direct and inverse limits. 
    For example, on a Liouville domain, $H$ can be taken to be a quadratic Hamiltonian while the sequence $H_i$  can be taken to consist be of Hamiltonians whose slope near infinity is constant and less than the smallest period of a Reeb orbit. This can be done so  that for each compact set the sequence still converges uniformly to $H$. Then it can be shown that for each $i$ we have $\overline{HF}^*(H_i)=H^*(M)$. This is so since, up to isomorphism, the Floer cohomology $\overline{HF}^*(H_i)=H^*(M)$ depends only on the slope at infinity.  Thus the left hand side is isomorphic to singular cohomology whereas the right hand side not generally isomorphic to singular cohomology. The reason for the discrepancy is that the Hamiltonians $H_i$ will have many periodic orbits whose action is arbitrarily close to $-\infty$. These cancel the contribution to $\overline{HF}^*(H_i)$ coming from the high action non-trivial periodic orbits. However, when truncating below at any fixed value as in the procedure described by \eqref{eqDfFlRedSec0}, the contribution of the high action periodic orbits  remains uncanceled.  
\item Continuity of truncated Floer cohomology with respect to \textit{uniform} convergence, and hence an extension of the definition of truncated Floer homology to $C^0$ Hamiltonians, has to the authors' knowledge first been observed in \cite{Viterbo99}. 
\item \label{Commentb}In the text a much stronger statement than the isomorphism of \eqref{eqEquivRedIso} is proven. Namely, it is shown that to each of $j=1,2$ one can associate a complete filtered chain complex $\widehat{CF}^*(H^j_i)$ after making some additional choices such that $\overline{HF}^*(\{H^j_i\})$ is the reduced cohomology of $\widehat{CF}^*(H^j_i)$. It is then shown that these complexes are filtered quasi-isomorphic. See \ref{dfFiltQuasIs} for the definition. Filtered quasi-isomorphism is an equivalence relation which implies isomorphism of the reduced Floer cohomology. In a forthcoming note joint with U. Varolgunes we show that filtered quasi-isomorphism in fact implies quasi-isomorphism in the usual sense. Thus Theorem \ref{tmUniExtFl} can be strengthened to a statement concerning unreduced Floer cohomology. This will be a great advantage as it will allow the application of tools from homological algebra. The notion of reduced cohomology is still central however to our construction of the product in symplectic cohomology as it is purely cohomological. Chain level constructions involving colimits are generally extremely  involved as one needs to keep track of higher homotopical data. The construction at the chain level is carried out in  Theorem~\ref{tmOverHFChainLv} by taking an appropriate kind of chain level limit which takes the Banach topology of the complexes $CF^*(H_i^j)$ into account. This builds on a construction from \cite{abouzaidSeidel2010} and has also been utilized in \cite{Venkatesh2017, Varolgunes2018}.

\item Theorem \ref{tmUniExtFl} relies on the possibility to approximate any element in $\cH_{s.c.}$ from below by a sequence of functions which have small Lipschitz constant and are thus dissipative by Theorem \ref{mainTmA}. Proper functions which are not bounded below would require considering, in addition, inverse limits. We do not pursue this here.
\item Remark \ref{comments7}  allows one to adapt Floer theoretic constructions to the geometry of the specific setting one is interested in without having to worry about complicated compactness questions. For an example of this, see the  derivation of the Kunneth formula in Hamiltonian Floer cohomology in \S\ref{SubsecKunnethHam}. Two cautionary remarks are in order however:
\begin{enumerate}[wide, labelwidth=!, labelindent=10pt]
\item For there to be a relation between the reduced Floer cohomology and periodic orbits of the Hamiltonian we are investigating, we must at least rule out divergence of the second type described in \S\ref{secOvrview} below. Namely, we need to establish loopwise dissipativity, or some related property. In the geometrically interesting settings the author is aware of this is straightforward, but it would be interesting to have a better understanding of this property.
\item It is theoretically possible for there to exist Floer data $(H,J)$ which are not dissipative, but for which, due to some accident, all the Floer moduli spaces are compact and, moreover, give rise to reduced Floer homologies differing from $\overline{HF}^*(H,J)$ as stipulated by Theorem~\ref{tmUniExtFl}. This cannot happen for $(H,J)$ which satisfy the following robustness property enjoyed by dissipative Floer data. \textit{\label{prop}The set of Floer solutions intersecting a given compact set $K$ and having energy at most $E$ does not change if the Floer datum is changed outside of a sufficiently large ball around $K$.} Note that the usually employed maximum principles are global in nature and so do not imply this property.
    \end{enumerate}
\end{enumerate}

The following sections \S\ref{Sec1} through \S\ref{Sec5} are devoted to the construction of dissipative Floer data. They are organized as follows. Sections \S\ref{Sec1} and \S\ref{Sec3} are concerned with ruling out type $1$ divergence. In \S\ref{Sec1} we introduce the notion of i-boundedness, establish its  contractibility and derive various versions of diameter estimate it implies. In \S\ref{Sec3} we introduce the Floer equation and the Gromov metric. We introduce the notions of i-bounded and geometrically dissipative almost Floer data. Finally, we study the geometry of the Gromov metric for translation invariant Floer data. Section \S\ref{Sec5} is concerned with ruling out type $2$ divergence. In it we introduce and study the property of loopwise dissipativity, and establish a diameter estimate as well as some effective criteria.

\section{i-bounded almost complex structures}\label{Sec1}
For a Riemannian metric $g$ on a manifold $M$ and a point $p\in M$ we denote by $\inj_g(p)$ the radius of injectivity and by $\Sec_g(p)$ the maximal sectional curvature at $p$. We drop $g$ from the notation when it is clear from the context.

\begin{df}\label{dfIntBounded}
Let $(M,g)$ be a complete Riemannian manifold. For $a>0$, the metric $g$ is said to be \textbf{$a$-bounded} at a point $p\in M$ if $\inj(x)\geq\frac1a$ and $|\Sec(x)|\leq a^2$ for all $x\in B_{1/a}(p)$.

We say that $g$ is \textbf{intermittently bounded}, abbreviated \textbf{i-bounded}, if there is an exhaustion $K_1\subset K_2\subset \dots$ of $M$ by precompact sets and a sequence $\{a_i\}_{i\geq 1}$ of positive numbers such that the following holds.
\begin{enumerate}
    \item $d( K_i,\partial K_{i+1})> \frac1{a_i}+\frac1{a_{i+1}}.$
    \item $g$ is $a_i$-bounded on $\partial K_i$.
    \item
        \begin{equation}\label{Eqtame}
            \sum_{i=1}^{\infty}\frac1{{a_i}^2}=\infty.
        \end{equation}
\end{enumerate}
The data $\{K_i,a_i\}_{i\geq 1}$ is called taming data for $(M,g)$.

More generally we allow a slight weakening in the definition and say that a Riemannian metric $g$ is i-bounded if there exists a metric $g'$ that is i-bounded as above with taming data $(K_i,a_i)$ and a sequence of constants $C_i$ such that
\begin{enumerate}
    \item
        \begin{equation}\label{Eqtame}
            \sum_{i=1}^{\infty}\frac1{{(C_ia_i)}^2}=\infty.
        \end{equation}
    \item
        $g$ is $C_i$-quasi-isometric to $g'$ on $B\left(\partial K_i,\frac1{a_i}\right)$. Namely,
        \[
            \frac1{C_i}\|X\|_g \leq \|X\|_{g'}\leq C_i\|X\|_g
        \]
        on $B\left(\partial K_i,\frac1{a_i}\right)$.
\end{enumerate}
In this case we will refer to the sequence $(K_i,a_i,C_i)$ as the taming data of $g$.

For a symplectic manifold $(M,\omega)$, an $\omega$-compatible almost complex structure $J$ is called \textbf{i-bounded} if $g_J$ is i-bounded. The symplectic form $\omega$ is said to be \textbf{i-bounded} if it admits an i-bounded almost complex structure. For an i-bounded $(M,\omega)$ denote by $\mathcal{J}_{i.b.}(M,\omega)$ the space of i-bounded almost complex structures.

A $k$-parameter family $(g_t)_{t\in[0,1]^k}$ of i-bounded Riemannian metrics on $M$ is said to be \textbf{uniformly i-bounded}, or \textbf{u.i.b.}, if there is an $\epsilon>0$ such that for each $t_0\in [0,1]^k$ the taming data $\{K_i,a_i,C_i\}$ can be chosen fixed on the $\epsilon$ neighborhood of $t_0$.
 \begin{comment}
 and for each $i$, the restriction of $g$ to $B_{1/a}(\partial K_i)$ is quasi-isometric to a fixed metric on the same neighborhood.
\end{comment}
A family $\{J_t\}$ of almost complex structures is called \textbf{u.i.b.} if the corresponding family $\{g_{J_t}\}$ of Riemannian metrics is uniformly i-bounded.
\end{df}
\begin{ex}
If $J$ is geometrically bounded, meaning that $g_J$ is $a$-bounded everywhere for some $a$, it is i-bounded. In this case, we can take the taming data to be $\{K_i=B_{3i/a}(p),a_i=a\}$ for some arbitrary point $p\in M$.
\end{ex}
\begin{ex}\label{exIboundTam}
Suppose now $f:M\to\R$ is the distance from some point $p\in M$ and that at each point $x\in M$ the metric $g_J$ is $f(p)$-bounded, then $g_J$ is still $i$-bounded. For this case consider the sequence of real numbers $b_i$ obtained from the set $\cup_{n=1}^{\infty}\{n+k/n|0\leq k<n\}\subset\R$ with its standard order.  Then the sequence $(K_i=f^{-1}(0,b_{3i}), a_i=\lceil b_{3i}\rceil)$ constitutes taming data for $g_J$. Indeed, by assumption, the metric is $a_i$ bounded on $K_i$ and the series $\sum\frac1{a_i^2}$ is readily seen to diverge.

\end{ex}
\begin{comment}

Before proceeding, we motivate the various parts of this definition. By the monotonicity inequality of \cite{Sikorav94}, stated below as Theorem~\ref{tmMonontonicity}, a $J$-holomorphic curve passing through an $a$-bounded point loses an amount of energy proportional to $\frac1{a^2}$. Thus, if $J$ is i-bounded then  for any $i\in \N$ and $E\in\R$ there is a number $N=N(i,E)$ such that a closed $J$ holomorphic curve intersecting $K_i$ and having energy at most $E$ is contained in $K_{i+N}.$
\end{comment}
\begin{rem}\label{rmGeoBCont}
The condition of uniform i-boundedness is framed so that it simultaneously guarantees the conclusions of Theorems~\ref{tmWeakTameCont} and ~\ref{tmDiamEst} below. Namely, on the one hand, the condition is contractible in the sense that any two homotopies satisfying the condition are connected by a homotopy satisfying the same condition. On the other hand, it still allows a priori control of the diameters of $J$-holomorphic curves. If we were to require boundedness everywhere, not just near $\partial K_i$, it appears unlikely to get a contractible condition as required in invariance proofs\footnote{As evidence for this consider that one can show using the result of \cite{Nab96} that the space of complete Riemannian metrics inducing a given volume form and having bounded geometry is disconnected. In fact, it has infinitely many connected components.}.
\end{rem}
 \begin{rem}
 Theorem \ref{tmWeakTameCont} below will remain true if we impose more stringent requirements on the numbers $a_i$, say, that they be bounded by a given constant. The reason we allow the numbers $a_i$ to diverge (subject to~\eqref{Eqtame}) is that in the context of Floer theory some times there naturally arise almost complex structures with associated metrics that do not have uniformly bounded sectional curvature. Example are the Sasaki metric on the cotangent bundle and the induced metric on the mapping torus of a quadratic Hamiltonian on the completion of a Liouville domain.
\end{rem}
\begin{comment}
The first main theorem we prove is
\begin{tm}\label{tmGdCont}
The set $\mathcal{J}_{i.b.}(\omega)$ is contractible.
\end{tm}

Theorem~\ref{tmGdCont} is given a more precise formulation below as Theorem~\ref{tmWeakTameCont} and is proven there.
\end{comment}
\begin{rem}
Note that  if $J$ is i-bounded and $J'$ is such that $\|J-J'\|_{g_J}$ is bounded, then $J'$ is i-bounded.
\end{rem}

\begin{tm}\label{tmWeakTameCont}
The space $\mathcal{J}_{i.b.}(M,\omega)$ is connected. Moreover, any two elements can be connected by a u.i.b. family. Similarly, any two u.i.b. $k$-parameter families can be connected by a u.i.b. $k+1$-parameter family.
\end{tm}
\begin{rem}
The idea of the proof is very similar to the that of Proposition 11.22 in~\cite{CielEli}.
\end{rem}
\begin{proof}
Let $J_0,J_1\in\mathcal{J}_{i.b.}$. Given taming  data $\{K^i_n,a^i_n,C^i_n\}_{n\geq 1}$ for $J_i,i=0,1$. For the rest of the proof we assume $C^i_n=1$, the adjustment to the general case being trivial. Let $(c^i_n,d^i_n)_{n\geq 1}$ be sequences of positive integers constructed inductively such that the following holds.
\begin{enumerate}
\item\label{itDisja}
\[
\overline{K^0}_{d^0_n+c^0_n}\subset K^1_{d^1_n},\qquad \overline{K^1}_{d^1_n+c^1_n}\subset K^0_{d^0_{n+1}},\qquad\forall n.
\]
\item\label{itDisjb}
\[
\sum_{k=d^i_n+1}^{d^i_{n+c_n-1}}\left(\frac1{a^i_k}\right)^2\geq\frac1{n},\qquad i=0,1.
\]
\end{enumerate}
Write $V^i_n:=\overline{K^i}_{d^i_n+c^i_n}\setminus K^i_{d^i_n}$ for $i=0,1$. The sets $V^i_n$ are all disjoint by~\ref{itDisja}.  Let $\{J_s\})_{s\in[0,1]}$ be a smooth homotopy connecting $J_0$ and $J_1$ which is fixed and equal to $J_1$ on the subsets $V^0_n$ for all $s\in[0,2/3]$ and to $J_1$ on the subsets $V^1_n$ for all $s\in[1/3,1].$ We refer to such a homotopy as a zig-zag homotopy. See Figure \ref{figZigZag}. Let
\[
A^i:=\cup_n[d^i_n+1,d^i_n+c^i_n-1]\cap\N
\]
for $i=0,1$.  By~\ref{itDisja} and~\ref{itDisjb}, the data
\[
\left\{K^i_{n^i_k},a^i_{n^i_k}\right\}_{n^i_k\in A^i},\qquad i=0,1,
\]
constitute taming data for $J_s$ on the intervals $[0,2/3]$ and $[1/3,1]$ respectively. Moreover, for each $s\in[0,1]$ the metric $g_{J_s}$ is complete. Indeed, the  distance of  $\partial K^i_{n_k}$ from any fixed point goes to $\infty$ for $i=0$ and $s\in[0,2/3],$ and for $i=1$ and $s\in [1/3,1]$. We have thus connected $J_0$ and $J_1$ in a uniformly i-bounded way.
\begin{figure}[h]
\begin{tikzpicture}[scale=1.5]
\foreach \x in {0,0.5,1,1.5}{
\shade[left color=white,right color=white](1,1+2*\x) rectangle (3,1.5+2*\x);
\shade[left color=white,right color=black] (3,1+2*\x) rectangle (7,1.5+2*\x);
\shade[left color=white,right color=black] (1,1.5+2*\x) rectangle (4,2+2*\x);
\shade[left color=black,right color=black](4,1.5+2*\x) rectangle (7,2+2*\x);
}
\draw[thick,->] (0.9,0.9) -- (7.1,0.9);
\node at (4.2,0.5){$J$};
\draw[thick,->] (0.9,0.9) -- (0.9,5.2);
\node at (0.5,3.2){$M$};
\end{tikzpicture}
\caption{A Zig-Zag homotopy from $J_0$ (light) to $J_1$ (dark). }
\label{figZigZag}
\end{figure}
We now generalize to the $k$-parameter case. Let $\{J_{i,t}\}_{t\in[0,1]^k}$ for $i=0,1$ be two smooth $k$-parameter families.  Let $C$ be an open cover of the cube $[0,1]^k$ by open cubes of side length $\epsilon$ for $\epsilon$ so small that the taming data for both families can be chosen fixed on each such cube. For each $c\in C$ and $i\in \{0,1\}$ we construct precompact open subsets $\{V^{i,c}_n\}$ in such a way that
\begin{enumerate}
\item
$V^{i_1,c_1}_{n_1}$ is disjoint from $V^{i_2,c_2}_{n_2}$ whenever $(i_1,c_1,n_1)\neq (i_2,c_2,n_2)$
\item
There is taming data supported in $\cup_{n=1}^{\infty}V^{i,c}_n$ for $\{J_{i,t}\}_{t\in c}$ where we say that the taming data $\{K_i,a_i\}_{i\geq 1}$ is supported in an open set $V\subset M$ if $V$ contains all the balls $B_{1/a_i}(\partial K_i)$.
\end{enumerate}
Such sets can be constructed inductively along the same lines as in the $0$-parameter case. We can then take any smooth homotopy
\[
\{J_{s,t}\}_{(s,t)\in[0,1]\times[0,1]^k}
\]
which is fixed on all the subsets $\{V^{0,c}_n\}$ for $s\in[0,2/3]$ and on all the subsets $\{V^{1,c}_n\}$ for $s\in [1/3,1]$.
\end{proof}
For a $J$ holomorphic curve $u:S\to M$ denote by $E(u;S)$ the energy
\[
\int_Su^*\omega
\]
of $u$ on $S$. We drop $S$ from the notation when it is clear from the context.

The following theorem is taken from \cite{Sikorav94}.
\begin{tm}\label{tmMonontonicity}[\textbf{Monotonicity}]
Let $g_J$ be $a$-bounded\footnote{As the proof shows, we only need an estimate from above on the sectional curvature. The  stronger requirement is needed later in \S\ref{subsubinj}.} at $p\in M$. Let $\Sigma$ be a compact Riemann surface with boundary and let $u:\Sigma\to M$ be $J$-holomorphic such that $p$ is in the image of $u$ and such that
\[
u(\partial \Sigma)\cap B_{1/a}(p)=\emptyset.
\]
Then there is a universal constant $c$ such that
\[
E\left(u;u^{-1}(B_{1/a}(p))\right)\geq\frac c{a^2}.
\]
If $g_J$ is quasi-isometric to an $a$-bounded metric with quasi-isometry constant $A$, the same inequality holds but with $c$ replaced by $\frac{c}{A^2}$.
\end{tm}

\begin{proof}
This just a reformulation of the monotonicity inequality in \cite{Sikorav94}. See Proposition 4.3.1(ii) and the comment right after Definition 4.1.1 there. For completeness, we add a statement and proof of that comment in Lemma~\ref{lmIsopCurv} as we didn't find a proof of it in the literature.
\end{proof}

\begin{lm}\label{lmIsopCurv}
Let $g$ be Riemannian metric which is $a$-bounded at $p\in M$. Then any loop $\gamma:S^1\to B_{1/(2a)}(p)$  bounds a disk of area less than $\frac12\ell^2(\gamma).$
\end{lm}
Our proof is taken from \cite{MS2}, the only addition being the precise dependence on the curvature.
\begin{proof}
Write $\gamma(0)=q$.  Let $\tilde{\gamma}:S^1 \to T_qM$ be the unique path such that $\exp_q\tilde{\gamma}(\theta)=\gamma(\theta)$.  Consider the disk $u(t,\theta)=\exp_qt\tilde{\gamma}(\theta)$. Using the  triangle inequality one shows that  $u$ maps into the ball $B_{1/a}(p)$. Since the geodesics emanating from $q$ minimize distance within $B_{1/a}(q)$, we have
\begin{equation}\label{eqlmIsopCurv1}
\|\partial_tu\|=\left\|\tilde{\gamma}(\theta)\right\|= d(q,\gamma(\theta))\leq\frac12\ell(\gamma).
\end{equation}

We need to estimate the Jacobi field $J(t,\theta):=\partial_\theta u(t,\theta)$. More precisely, we need to estimate the component $J^{\perp}$ which is perpendicular to $\tilde{\gamma}(\theta)$. For this we use the generalized Rauch estimate \cite[1.8.2]{Karcher} according to which the function
\[
f(t)=\frac{\|J^{\perp}(t,\theta)\|}{\sin t}
\]
is nondecreasing on the interval $(0,\pi)$ \footnote{Observe the coordinate $t$ is related to the coordinate $r$ of  \cite[1.8.2]{Karcher} by $r={\|\tilde{\gamma}\|}t$ and that $\|\tilde{\gamma}\|\leq a$ where $a^2$ plays the role of $\Delta$ in \cite[1.8.2]{Karcher}.}. Observe that $\gamma'(\theta)= J(1,\theta)$. So,
\[
\|J^{\perp}(t,\theta)\|\leq\frac{\sin t}{\sin 1}\|\gamma'(\theta)\|,\quad t\leq 1.
\]

Applying the last estimate and eq. \eqref{eqlmIsopCurv1} we get,
\begin{align}
\area(u)&=\int_0^1\int_0^{2\pi}\|\partial_\theta u\|\|\partial_tu\|\sin(\angle(\partial_\theta u,\partial_t u))d\theta dt\notag\\
&=\int_0^1\int_0^{2\pi}\|\partial_\theta u^{\perp}\|\|\partial_tu\|\sin(\angle(\partial_\theta u^{\perp},\partial_t u))d\theta dt\notag\\
&\leq\int_0^1\int_0^{2\pi}\|\gamma'(\theta)\| \frac12\ell(\gamma)d\theta dt\notag\\
&=\frac12\ell^2(\gamma).\notag
\end{align}

\end{proof}

The following theorem is fundamental for all that follows. It gives a priori control over the diameter of a $J$ holomorphic curve $u: \Sigma\to M$ with free boundary in terms of its energy.
\begin{tm}\label{tmDiamEst}
Let $J\in \mathcal{J}_{i.b.}$.
\begin{enumerate}
\item\label{tmDiamEstc1} For any compact set $K\subset M$ and $E\in\R_+$ there exists an $R>0$ such that for any connected compact Riemann surface $\Sigma$ with boundary and any $J$-holomorphic map
\[
u:(\Sigma,\partial\Sigma)\to (M,K)
\]
satisfying $E(u;\Sigma)\leq E,$ we have $u(\Sigma)\subset B_R(K)$. Moreover, if the geometry is uniformly bounded, $R$ can be taken independent of $K$. In fact, it can be taken proportional to $E$.
\item\label{tmDiamEstc2}
  Let $\Sigma$ be a connected compact Riemann surface with boundary. For any compact set $K\subset M$, any compact subset $S$ of the interior of $\Sigma$ and any $E\in\R_+,$ there exists an $R$ such that for any $J$-holomorphic map
\[
u:\Sigma\to M
\]
satisfying $E(u;\Sigma)\leq E$ and $u(S)\cap K\neq\emptyset$ we have $u(S)\subset B_R(K)$.

 \end{enumerate}
In both cases, besides the dependence on $E$ and on $S$, $R$ depends only on taming data of $J$ inside $B_R(K)$. That is, given $J'$ which has the same taming data as $J$ on $B_R(K)$, the claim will hold with the same $R$ for $J'$-holomorphic curves with energy at most $E$.
\end{tm}

\begin{rem}
The reader should be careful to note that in case~\ref{tmDiamEstc2} where there is no control over the image of the boundary, to control the diameter of $u(S)$  we need control of the energy in the larger surface $\Sigma$.
\end{rem}
\begin{rem}\label{remDiamEstDep}
A remark is in order concerning the dependence of $R$ on the geometry in case~\ref{tmDiamEstc2}. In addition to the dependence on the taming data and on $E$, $R$ depends on an estimate from below the distance $d( S,\partial\Sigma)$ and from above on the area and curvature of $\Sigma$ all with respect to an arbitrarily chosen conformal metric.
\end{rem}

\begin{proof}
Let $\{K_i,a_i,C_i\}$ be taming data for $J$. The argument will be given for the case $C_i=1,\forall i\in\N$. Let $N\in\Z$ be such that $K\subset K_N$. Let $i_0>0$ and $x_{i_0}\in\Sigma$ be such that $u(x_{i_0})\in\partial K_{i_0+N}$. If no such $i_0$ and $x_{i_0}$ exist we take $R=d(K,K_{N+1})$ and we are done. Otherwise, there is a sequence $x_i\in\Sigma$ such that $u(x_i)\in\partial K_{N+i}$ for $0<i\leq i_0$.
In case~\ref{tmDiamEstc1} we argue as follows. For each $1\leq i\leq i_0$, we have $B_{1/a_{N+i}}(u(x_i))\cap u(\partial \Sigma)=\emptyset$. Also,
\[
d(u(x_i),u(x_j))>\frac1{a_{N+i}}+\frac1{a_{N+j}},
\]
whenever $i\neq j$. By Theorem \ref{tmMonontonicity} we obtain
\[
E(u;\Sigma)\geq\sum_{i=1}^{i_0}E\left(u;u^{-1}(B_{1/a_{N+i}}(u(x_i))\right)\geq \sum_{i=1}^{i_0}\frac{c}{a^2_{i+N}}.
\]
By \eqref{Eqtame} this implies an a priori upper bound on $i_0$. Let $i_0$ be the largest possible such. The claim then holds with $R=d(K,K_{N+i_0+1})$.

In case~\ref{tmDiamEstc2} we argue as follows. Pick an area form $\omega_{\Sigma}$ on $\Sigma$ which together with $j_\Sigma$ determines a metric whose sectional curvature is bounded in absolute value by $1$. Let $A=\int_{\Sigma}\omega_{\Sigma}$ and let $\epsilon:=d(S,\partial\Sigma)$. Let
\[
\tilde{u}:=id\times u:\Sigma\to \Sigma\times M
\]
be the graph of $u$, and let $\tilde{J}$ be the product almost complex structure on $\Sigma\times M$. Then $\tilde{u}$ is $J$-holomorphic and $E(\tilde{u})= E(u)+A$. For any $x\in\Sigma$, any $p\in M$ and any $a\geq 1$ such that $(M,g_J)$ is $a$-bounded at $p$ we have that $(\Sigma\times M,g_{\tilde{J}})$ is $a$-bounded at $(x,p)$. Moreover, defining $x_i$ as before for points $x_i\in S$, the ball of radius $\min\left\{\frac1{a_{i+N}},\epsilon\right\}$ around $\tilde{u}(x_i)=(x_i,u(x_i))$ does not meet $\tilde{u}(\partial\Sigma)$. Thus, arguing as before, we have
\[
E(u;\Sigma)+A=E(\tilde{u};\Sigma)\geq\sum_{i=1}^{i_0}c\min\left\{\frac1{a^2_{i+N}},\epsilon^2\right\}.
\]
The claim follows as before.

\end{proof}
The final ingredient we shall need is the following elementary observation whose proof we leave for the reader.
\begin{tm}\label{tmPullbackTame}
The pullback of a u.i.b. family by a uniformly continuous map is u.i.b.
\end{tm}
Theorems \ref{tmWeakTameCont} and~\ref{tmDiamEst} have consequences for symplectic invariants on open manifolds which we state as the following theorem.
\begin{tm}\label{TmSympInvGWSH}
The following invariants whose definition requires fixing geometrically bounded almost complex structure $J$ are independent of the choice of such $J$.
\begin{enumerate}
\item
The Gromov-Witten theory on geometrically bounded manifolds studied in \cite{Lu06}.
\item
Symplectic homology of relatively compact open sets studied in \cite{CielGinzKer}.
\item
Rabinowitz Floer homology of tame stable Hamiltonian hypersurfaces in geometrically bounded manifolds \cite{cieliebak2010}\footnote{ See Remark 3.3 in \cite{CielGinzKer} and likewise the beginning of \S 4.5 in \cite{cieliebak2010} where this question is raised.}.
\end{enumerate}
\end{tm}
\begin{proof}
By Theorem~\ref{tmWeakTameCont} we can connect any two such almost complex structures $J_0$ and $J_1$ via a uniformly i-bounded path $J_s$ of compatible almost complex structures. We outline how to prove invariance in each case separately.
\begin{enumerate}

\item As a particular consequence of Theorem \ref{tmDiamEst}, the $J_s$-holomorphic curves representing a given homology class and intersecting non-trivially a compactly supported cohomology class are all contained in a fixed compact set $K$. Thus the moduli space of such spheres with $s$ varying from $0$ to $1$ generically gives rise to a cobordism between the moduli spaces associated with $J_0$ and $J_1$.
    \item The symplectic homology is defined by considering compactly supported Hamiltonians. In that setting, geometric boundedness gives rise to $C^0$ estimates as follows. Suppose $u$ is  a Floer trajectory connecting periodic orbits inside some open set $U$ where some Hamiltonian $H$ is supported. Then the intersection of the image of $u$ with $M\setminus U$ is $J$-holomorphic. Moreover, the symplectic energy is bounded a priori in terms of the action difference across $u$. It is clear by Theorem \ref{tmDiamEst} that i-boundedness is sufficient to obtain the same type of $C^0$ estimate.  We show that the continuation map  associated with the $1$-parameter family $J_s$ fixing $H$ also satisfies a $C^0$ estimate. Note that we cannot directly appeal to Theorem \ref{tmDiamEst}, since we are now considering a domain dependent $J$. To overcome this difficulty we apply the Gromov trick. Namely, we consider the graph
        \[
        \tilde{u}:\R\times\R/\Z\to\R\times\R/\Z\times M
        \]
        for a continuation map $u$. Let $\tilde{J}:=j\times J$ where $j$ is the standard complex structure on the cylinder. Then $\tilde{u}$ is $\tilde{J}$-holomorphic outside of $U$.  Consider the area form on the cylinder obtained by identifying it with the twice punctured sphere. Then the associated metric $g_{\tilde{J}}$ is i-bounded. Moreover the energy of the part of $\tilde{u}$ mapping outside of $U$ is still bounded a priori in terms of the periodic orbits connected by $u$\footnote{Later, when considering a non-zero Hamiltonian we will generally have i-boundedness only if we consider the cylindrical metric, which has infinite area. For this reason we will need to complement i-boundedness with the additional requirement of loopwise dissipativity.}. Appealing to Theorem \ref{tmDiamEst} the path $J_s$ gives rise to a chain homotopy between the Floer homologies of any fixed Hamiltonian with respect to the two choices of $J$. In the same way, given $H\geq K$, and a homotopy $H_s$ of Hamiltonians, the concatenations $(K,J_s)\#(H_s,J_0)$ and $(H_s,J_1)\#(H,J_s)$ can be interpolated by a homotopy $(H_{s,\tau},J_{s,\tau})$ such that $H_{s,\tau}$ is compactly supported and $J_{s,\tau}$ is uniformly i-bounded. A $C^0$ estimate for the homotopy is immediate from Theorem \ref{tmDiamEst}. Thus the continuation maps in the directed system defining symplectic homology also coincide generically  for different choices of $J$. It follows that the two invariants coincide.
        \item  Rabinowitz Floer homology for a stable Hamiltonian hypersurface $\Sigma$ in a geometrically bounded manifold is considered with Hamiltonian vector fields that are supported on some compact set $K$ containing $\Sigma$. A gradient trajectory in Rabinowitz Floer homology consists of a pair $(v,\eta)$ where $v:\R\times\R/\Z\to M$ and $\eta:\R\to\R$ satisfying a certain equation. The part of $V$ which maps out of $K$ is $J$-holomorphic with a priori bounded energy for a gradient connecting critical points. Compactness of the space of gradient trajectories consists in first establishing a $C^0$ estimate on $V$ and once $V$ is confined to a compact region, deriving an estimate on $\eta$ and appealing to Gromov compactness. As in the previous part, i-boundedness is sufficient for the $C^0$ estimates on $V$. This holds as well for $s$ dependent $J$. The argument for invariance now follows as before.
\end{enumerate}

\end{proof}
\begin{rem}\label{rmGW}
The question of what kind of deformation of the symplectic structure preserves which of these invariants appears to be more subtle and is not studied here. In \cite{GromanMerry2018}, the question is taken up for particular type of deformation on Liouville domains.
\end{rem}
\begin{rem}
It is not known to the author whether the class of i-bounded symplectic manifolds is strictly larger than the class of geometrically bounded symplectic manifolds. It appears likely that it might be easier to characterize the class of i-bounded symplectic manifolds in terms of the topology of $\omega$. It is easy to see that a punctured Riemann surface cannot be assigned an i-bounded compatible metric even though there is a compatible complete metric of bounded curvature. This motivates the following question. Suppose $M$ satisfies that for any disconnecting compact hypersurfaces $\Sigma$, a component of $M\setminus \Sigma$ which has finite volume is precompact in $M$. Are there any obstructions to finding a compatible i-bounded $J$?

It is also an interesting question whether finiteness of the total volume is an obstruction to weak boundedness, as it is to boundedness. In dimension $2$ the answer is positive as remarked above, but in higher dimension this is not clear to the author. If the answer is negative, it is possible that there are contact manifolds whose symplectization admit i-bounded almost complex structures allowing to define Floer theoretic invariants on them without recourse to symplectic field theory. This remark is due to A. Oancea.
\end{rem} 

\section{Floer solutions and the Gromov metric}\label{Sec3}
\subsection{Floer's equation}\label{SecFloerGrom}
Let $(M,\omega)$ be a symplectic manifold. Denote by $\mathcal{L}M$ the free loop space $C^{\infty}(\R/\Z,M)$. For a smooth function $H\in C^{\infty}(\R/\Z\times M)$ and for any $t\in \R/\Z$ denote by $X_{H_t}$ its Hamiltonian vector field. This is the unique vector field satisfying $dH_t(\cdot)=\omega(X_{H_t},\cdot).$
For each component $\mathcal{L}M_a$ of $\mathcal{L}M$ pick a base loop $\gamma_a$ and define a (multi-valued) functional $\mathcal{A}_H:\mathcal{L}M\to \R$ by
\[
\mathcal{A}_{H}(\gamma):=-\int\omega-\int_0^{2\pi}H(\gamma(t))dt,
\]
where the integral of $\omega$ is taken over a path in loop-space from $\gamma_a$ to $\gamma$. Later, in \S\ref{SecDefHam} we will consider $\cA_H$ as a single valued functional on an appropriate cover of the loop-space.

Denote by $\mathcal{P}(H)\subset\mathcal{L}M$ the set of $1$-periodic orbits of $X_H$. %Denoting by
%\[
%\pi:\widetilde{\mathcal{L}M}\to\mathcal{L}M,
%\]
%the covering map, let
%\[
%\widetilde{\mathcal{P}(H)}=\pi^{-1}(\mathcal{P}(H)).
%\]
This is the same as the critical point set of $\mathcal{A}_H$.
\begin{comment}
We define an index
\[
i_{RS}:\widetilde{\mathcal{P}(H)}\to\mathbb{Z},
\]
as follows. For each homotopy class $a\in \pi_0(\mathcal{L}M)$ fix  a trivialization of $\gamma_a^*TM$. Then if $\tilde{\gamma}=(\gamma,A), $ trivialize $\gamma^*TM$ along $A$ by extending the existing trivialization from $\gamma_a$. With respect to this trivialization, the linearization $t\mapsto D\psi_{t,\gamma(t)}$ of the flow along $\gamma$ is a path of symplectic matrices to which is associated its Robbin-Salamon index \cite{RobbinSalamon93}. We take $i_{RS}(\tilde{\gamma})$ to be the Robbin-Salamon index in this trivialization.  Note that $i_{RS}$ is independent of choices up to an integer shift $n_a$ for each $a\in \pi_0(\mathcal{L}M).$
\end{comment}
Given an $\R/\Z$ parameterized family of almost complex structures $J_t$ on $M$, the gradient of $\mathcal{A}_{H}(\gamma)$ at $\gamma$ is the vector field
\[
\nabla\mathcal{A}_{H}(\gamma)(t):=J_t(\dot{\gamma}(t)-X_{H_t}(\gamma(t)))
\]
along $\gamma$. Note that the gradient field is independent of the choice of base paths and is single valued.  A gradient trajectory is a path in (a covering of) loop space whose tangent vector at each point is the negative gradient at that point. Explicitly a gradient trajectory is a map
\[
u:\R\times \R/\Z\to M,
\]
satisfying Floer's equation
\begin{equation}\label{eqFloTranInv}
\partial_su+J_t(\partial_tu-X_{H_t}\circ u)=0.
\end{equation}
We refer to such solutions as \textbf{Floer trajectories}. A Floer trajectory is \textbf{nontrivial} if there is a point such that $\partial_tu\neq X_H$.

More generally, let $\Sigma$ be a finite type Riemann surface with cylindrical ends. This means that $\Sigma$ is obtained from a compact Riemann surface $\overline{\Sigma}$ by removing a finite number of punctures. Moreover, near each puncture we fix a conformal coordinate system $(s,t):(a,b)\times \R/\Z\to\Sigma$ such that either $(a,b)=(-\infty,0)$ or $(a,b)=(0,\infty)$. In the first case we call the puncture negative and in the second, positive. Let $\mathfrak{H}\in\Omega^1(\Sigma,C^{\infty}(M))$ be a $1$-form with values in smooth Hamiltonians such that near each puncture there is an $H\in C^{\infty}(\R/\Z\times M)$ for which $\mathfrak{H}=Hdt$ in the cylindrical coordinates. We denote by $X_\mathfrak{H}$ the corresponding $1$-form with values in Hamiltonian vector fields. Let $J\in C^{\infty}(\Sigma,\mathcal{J}(\omega))$ and suppose $J$ is independent of the coordinate $s$ on the cylindrical ends. The datum $(\mathfrak{H},J)$ is called a \textbf{domain dependent Floer datum}.

Let $u:\Sigma\to M$ be smooth. For a $1$-form $\gamma$ on $\Sigma$ with values in $u^*TM$ write
\[
\gamma^{0,1}:=\frac12\left(\gamma+J\circ\gamma\circ j_\Sigma\right).
\]
A \textbf{Floer solution} on $\Sigma$ is a map $u:\Sigma\to M$ satisfying Floer's equation
\begin{equation} \label{eqFloGen}
(du-X_\mathfrak{H}(u))^{0,1}=0.
\end{equation}
Note that Equation~\eqref{eqFloTranInv} is equivalent to a special case of Equation \eqref{eqFloGen}. We refer to $J$ and $\mathfrak{H}$ as the Floer data of $u$. The \textbf{geometric energy} of $u$ on a subset $S\subset\Sigma$ is defined as
\begin{equation}
E_{\mathfrak{H},J}(u;S):=\frac1{2}\int_S\|du-X_\mathfrak{H}\|_J^2dvol_\Sigma.
\end{equation}
We omit any one of $\mathfrak{H}$, $J$ or $S$ from the notation when they are clear from the context. We define the topological energy $E_{top}(u)$ of a Floer solution $u$ as follows. Consider $\mathfrak{H}$ as a $1$-form on $\Sigma\times M$ and let $\tilde{u}:\Sigma\to\Sigma\times M$ be the product map $\tilde{u}=\id\times u$. Then
\begin{equation}
E_{top}(u):=\int u^*\omega+d\tilde{u}^*\mathfrak{H}.
\end{equation}

Floer's equation reduces to the nonlinear Cauchy Riemann equation when $\mathfrak{H}(v)=\gamma\otimes Const$ for $\gamma$ a $1$-form on $\Sigma$. In this case the two definitions of the energy coincide. Namely, we have the identity
\begin{equation}\label{eqEnId}
\frac1{2}\int_S\|du-X_\mathfrak{H}\|^2=\frac1{2}\int_S\|du\|^2=\int_Su^*\omega.
\end{equation}
\subsection{The Gromov metric}\label{SecGrom}
Let $u:\Sigma\to M$ be a Floer solution for the Floer data $F=(\mathfrak{H},J)$. Define an almost complex structure $J_F$ on $\Sigma\times M$ by
\[
J_F(z,x):=\left(\begin{matrix} j_{\Sigma}(z) & 0  \\ X_\mathfrak{H}(z,x)\circ j_\Sigma(z)-J(z,x)\circ X_\mathfrak{H}(z,x) & J(x) \end{matrix}\right).
\]
Let
\[
\tilde{u}=(id,u):\Sigma\to\Sigma\times M.
\]
Then $\tilde{u}$ is $J_F$-holomorphic. This construction is known as Gromov's trick. See, e.g., \cite[Ch. 8.1]{MS2}.
\textit{Henceforth, given a Riemann surface $\Sigma$ with cylindrical ends, we shall assume that it comes equipped with an area form which is compatible with the complex structure and coincides with the standard one $ds\wedge dt$ on the ends.}
Note that $J_F$ is not generally tamed by the product symplectic structure $\omega_{\tilde{M}}=\pi_1^*\omega_{\Sigma}+\pi_2^*\omega_M$. However, we have the following lemma.
\begin{lm}\label{lmAssocSymp}
Suppose $\{\mathfrak{H},\mathfrak{H}\}=0$. Namely,  for any $z\in\Sigma$ and any pair $v_1,v_2\in T_z\Sigma$ we have
 \begin{equation}\label{EqPoisson}
 \{\mathfrak{H}(v_1),\mathfrak{H}(v_2)\}=0.
 \end{equation}
Now consider $\mathfrak{H}$ as a $1$-form on $\Sigma\times M$ which is trivial in the directions tangent to $M$. Assume that for each $(z,x)\in\Sigma\times M$ we have
\begin{equation}\label{EqMonHom}
d\mathfrak{H}(z,x)|_{T_z\Sigma}\geq 0.
\end{equation}
That is, it is positive with respect to the orientation determined by the complex structure $j_\Sigma$. Then the $2$-form
\[
\omega_\mathfrak{H}:=\pi_1^*\omega_\Sigma+\pi_2^*\omega_M+d\mathfrak{H},
\]
is a symplectic form on $\Sigma\times M$ which is compatible with $J_F$.
\end{lm}
\begin{proof}
We only show that $\omega_\mathfrak{H}$ is a symplectic form. Closedness is clear, so we only need to show non-degeneracy. In local coordinates on $\Sigma$ write
\[
\mathfrak{H}=Hdt+Gds.
\]
Then
\[
d\mathfrak{H}=dH\wedge dt+dG\wedge ds +(\partial_sH-\partial_tG)ds\wedge dt.
\]
Suppose there is a vector $v=(v_1,v_2)\in T(\Sigma\times M)$ for which $\iota_v\omega_\mathfrak{H}=0$. Then, in particular, the restrictions of $\iota_v\tilde{\omega}$ to the fibers of $\pi_2$ vanish,  giving
\[
-\iota_{v_2}\omega_M=dt(v_1)dH+ds(v_1)dG.
\]
So, $v_2=aX_H+bX_G$ for appropriate constants $a,b\in\R.$ Since $\{H,G\}=0$ it follows that $\iota_{v_2}(dH\wedge dt+dG\wedge ds)=0$. Thus,
\[
\iota_{v_1}(\omega_\Sigma+(\partial_sH-\partial_tG)ds\wedge dt)=0.
\]
Our assumption is that the coefficient of $ds\wedge dt$ in nonnegative. It follows that $v_1=0$ which in turn implies $v_2=0$.
\end{proof}
\begin{rem}
More generally, if we replace the estimate~\eqref{EqMonHom} by
\begin{equation}\label{EqMonHom1}
d\mathfrak{H}(z,x)|_{T_z\Sigma}\geq -ads\wedge dt,
\end{equation}
for some constant $a$, we have that the form
\[
\omega_{\mathfrak{H},a}:=\omega_\mathfrak{H}+ads\wedge dt,
\]
is symplectic.

The Poisson bracket condition \eqref{EqPoisson} may also be weakened to the requirement that for any point $z\in\sigma$ and vectors $v_1,v_2\in T_z\Sigma$ we have
\begin{equation}\label{eqWeakPoisson}
|\{\mathfrak{H}(v_1),\mathfrak{H}(v_2)\}(x)|<a\|v_1\|\|v_2\|,\quad\forall x\in M.
\end{equation}
In that case, the form $\omega_{\mathfrak{H},a}$ will again be a symplectic form.
\end{rem}

\begin{lm}\label{lmTopGeoEnEst}
Let $\Sigma$ be a Riemann surface with cylindrical ends and let $(\mathfrak{H},J)$ be a domain dependent Floer datum on $\Sigma$. For any $(\mathfrak{H},J)$-Floer solution $u:\Sigma\to M$ satisfying \eqref{EqPoisson} and \eqref{EqMonHom} and for any Borel subset $A\subset \Sigma$, we have
\begin{equation}
E(u;A):=\int_A\|du-X_\mathfrak{H}\|^2dvol_\Sigma\leq E_{top}(u;A).
\end{equation}
\end{lm}

\begin{proof}
Write in local coordinates $\mathfrak{H}=Hdt+Gds$. Then using the Floer equation and denoting by $d'$ the exterior derivative in the $M$ direction,
\begin{align}
\|du-X_\mathfrak{H}\|^2ds\wedge dt&=\omega(\partial_tu-X_H,X_G-\partial_su)ds\wedge dt\notag\\
&=u^*\omega+(d'H(\partial_su)+d'G(\partial_tu))ds\wedge dt\notag\\
&=u^*\omega+d\mathfrak{H}-(\partial_sH\circ u-\partial_tG\circ u))ds\wedge dt\notag\\
&\leq u^*\omega+d\mathfrak{H}.\notag
\end{align}
We have used the conformal invariance of energy
\[
\|du-X_\mathfrak{H}\|^2dvol_\Sigma=\|du-X_\mathfrak{H}\|^2ds\wedge dt.
\]
\end{proof}
Henceforth, we shall denote by $g_{J_F}$ the Riemannian metric determined by $\omega_\mathfrak{H}$ and $J_F$ and refer to it as the \textbf{Gromov metric.} When $\mathfrak{H}=Hdt$ we will also use the notation $J_H$ and $g_{J_H}$.

\begin{ex}\label{ExTransInv}
Let $\mathfrak{H}=Hdt$, where $H:M\to\R$ is smooth. Then one finds by a straightforward computation that
\begin{equation}\label{eqExTransInv}
g_{J_H}=\pi_1^*g_j + \pi_2^*g_J+g_{mixed},
\end{equation}
where $\pi_i$ are the natural projections and
\begin{equation}
g_{mixed}=-g_J(X_H,\cdot)dt-dt g_J(X_H,\cdot)+\|X_H\|^2dt^2.
\end{equation}
\end{ex}

In order to define the notion of $i$-boundedness for Floer data we need a relative notion of intermittent boundedness.
\begin{df}

Let $\Sigma$ be a Riemann surface with cylindrical ends. A Riemannian metric on $\Sigma\times M$  is said to be \textbf{intermittently bounded relative to the projection $\pi:\Sigma\times M\to M$} if there is an exhaustion of $\Sigma\times M$ by a sequence of open sets $K_i$ such that for any pre-compact open $U\subset \Sigma$ the sets $\pi^{-1}(U\cap K_i)$ are pre-compact, and such that the rest of Definition \ref{dfIntBounded} holds for these $K_i$. Let $\tilde{\omega}$ be a symplectic form on $\Sigma\times M$. An $\tilde{\omega}$ compatible almost complex structure $J$ on $\Sigma\times M$  is said to be intermittently bounded relative to $\pi$ if the associated metric $g_J$ is. We denote the set of al these by $\mathcal{J}_{i.b.}(\Sigma\times M, \tilde{\omega},\pi)$. For an open set $U\subset\Sigma$ we denote by  $\mathcal{J}_{i.b.}(U\times M, \tilde{\omega})$ the set of $\tilde{\omega}$-compatible almost complex structures on $U\times M$ that are restrictions of a $J\in \mathcal{J}_{i.b.}(\Sigma\times M, \tilde{\omega},\pi)$. 
\end{df}
The following Lemma is an obvious variant of Theorem \ref{tmDiamEst}\ref{tmDiamEstc2} the only difference being the need to restrict to $J$-holomorphic sections. 
\begin{lm}\label{lmDiamEstSec}
Let $U\subset \Sigma$ be an open pre-compact subset. Let $J\in\mathcal{J}_{i.b.}(U\times M, \tilde{\omega},\pi)$. Suppose $\|d\pi\|$ is uniformly bounded from above with respect to some fixed conformal metric on $\Sigma$. For any compact set $K\subset M$, any compact subset $S\subset U$, and any $E\in\R_+,$ there exists an $R$ such that for any $J$-holomorphic section
\[
u:U\to U\times  M
\]
satisfying $E(u;U)\leq E$ and $u(S)\cap K\neq\emptyset$ we have $u(S)\subset B_R(S\times K)$. 
\end{lm} 
\begin{proof}
The assumption on $\|d\pi\|$ guarantees that for any $z\in\Sigma$ we have $B_{\epsilon}(u(z))\subset u(B_{\epsilon}(z)) $. The argument is then word for word that of Theorem \ref{tmDiamEst}\ref{tmDiamEstc2} .
\end{proof}
\begin{rem}
The dependence of $R$ on $U$ and $S$ is spelled out in Remark \ref{remDiamEstDep}.
\end{rem}
\begin{lm}
For the Gromov metric $g_{J_F}$ associated with $F$ any Floer datum satisfying \eqref{EqPoisson} and~\eqref{EqMonHom} we have $\|d\pi\|\leq 1$. 
\end{lm}
\begin{proof}
For any vector $v$ tangent to $\Sigma\times M$ we have $\|v\|=\pi^*\omega_{\Sigma}(v,J_Fv)+\omega_F(v,J_F v)$. The second term is non-negative by Lemma \ref{lmAssocSymp} since  Lemma \ref{lmAssocSymp} holds for \emph{any} choice of $\omega_{\Sigma}$. The first term can be equals  
$\omega_{\Sigma}(\pi_*v,j_{\Sigma}\pi_*v)$ by holomorphicity of $\pi$. 
\end{proof}
\begin{df}\label{dfFloDataWbounded}
Let $\Sigma$ be a Riemann surface with cylindrical ends. A domain dependent Floer datum $(\mathfrak{H},J)$ on $\Sigma$ is called \textbf{i-bounded} if
\begin{enumerate}
\item $\mathfrak{H}$ satisfies \eqref{EqPoisson} and~\eqref{EqMonHom} (or, more generally, inequalities~\eqref{eqWeakPoisson} and~\eqref{EqMonHom1}).
\item There exists a finite open cover $C$ of $\Sigma$ such that for each $U\in C$ we have $J_{\mathfrak{H}}|_{U\times M}\in \mathcal{J}_{i.b.}(U\times M, \omega_\mathfrak{H})$ (or, more generally, $J_{\mathfrak{H}}\in\mathcal{J}_{i.b.}(U\times M, \omega_{\mathfrak{H},a})).$ 
 \end{enumerate}
\end{df}

\begin{df}\label{dfFloDataWboundedFamily}
Let $\mathcal{S}$ be a compact manifold with corners. A smooth family $\Sigma_{\{s\in\mathcal{S}\}}$ of (broken) Riemann surfaces with cylindrical ends together with a smooth choice of domain dependent i-bounded Floer data $(\mathfrak{H}_s,J_s)$  is called \textbf{admissible} if the following holds. Denote by $\pi:\widetilde{\cS}\to \cS$ the tautological bundle. Then we assume there is a smooth choice $\sigma_s$ of area forms on $\Sigma_s$ and a finite cover of $\widetilde{\mathcal{S}}$ by connected opens consisting of elements of two types: $Thick_{\mathcal{S}}$ and $Thin_{\mathcal{S}}$.  The elements of $Thick_{\mathcal{S}}$ are assumed to be subsets of $\widetilde{\cS}$ which are trivializable to the form $U=V\times W$ where $W\subset\mathcal{S}$  and $V$ is a bordered Riemann surface whose area is uniformly bounded on $W$. The fibers of $\pi$ restricted to elements of $Thin_{\mathcal{S}}$ are generically cylinders (of finite, half infinite or infinite length) which may degenerate to nodes at the corners.  Moreover, for the thin elements we require that the Floer data be translation invariant on the  fibers of $\pi$ and that the area forms coincide with $ds\wedge dt$. We say that the family $\mathcal{S}$  is \textbf{uniformly i-bounded } if for any thick element $U=V\times W$ there exist taming  data on $V\times M $ which are constant on $W$ and for any thin element $U$ there exist taming data on $[-1,1]\times \R/\Z\times M$ which are constant on $\pi(U)$.

\end{df}

For the rest of the section we wish to establish criteria for i-boundedness of $J_F$. This is not strictly necessary for the proof of the main theorems in the introduction. Lemma \ref{lmQuasHam} to be stated presently is all we shall need for that purpose. The proof is left to the reader.
\begin{lm}\label{lmQuasHam}
Let $(H_1,J)$ be a Floer datum and let $H_2$ be a time dependent Hamiltonian such that $\|X_{H_2}\|\leq C$ for some constant $C$. Then $g_{H_1+H_2}$ is $C^2$-quasi-isometric to $g_{H_1}$. In particular, when $J$ is i-bounded and $H$ is such that $\|X_H\|$ is bounded we have that $J_{H}$ is i-bounded.
\end{lm}

However, for applications in practice we need effective criteria. For example, we need to show that Floer data that has been hitherto used in the literature fits into the dissipative framework. To do this we need, first of all, a criterion for completeness of the metric $g_{J_F}$. Then we need to discuss how to compute the curvature of $g_{J_F}$ and control its radius of injectivity in terms of the Floer data $J$ and $\mathfrak{H}$. We do this in the case where $\mathfrak{H}=Hdt$ for a time independent Hamiltonian $H$  as in Example~\ref{ExTransInv}. Since intermittent boundedness is preserved under quasi-isometry, this is quite sufficient for applications insofar as Floer trajectories are concerned. The consideration of more general Floer solutions will be reduced to that of Floer trajectories.
\subsection{Completeness}
\begin{df}
Let $J$ be an almost complex structure. We say that an exhaustion function $H:M\to\R$ is \textbf{$J$-proper} if $H$ factors as $H=f\circ h$ for some proper smooth function $h:M\to\R$ satisfying $\|\nabla h\|\leq 1$ with respect to the metric $g_J$ on $M$.
\end{df}
\begin{lm}\label{lmStrPrpEst}
Let $H$ be a smooth time independent $J$-proper Hamiltonian i.e., $H=f\circ h$ with $\|\nabla h\|\leq 1$. For any function $g:[a,b]\to \R$ and any $\gamma:[a,b]\to M$ we have
\begin{equation}
|h(\gamma(b))-h(\gamma(a))|^2\leq (b-a)\int_a^b\|g(t)X_{H}-\gamma'(t)\|^2 dt.
\end{equation}
\end{lm}
\begin{rem}
More generally, if $H$ is time dependent, and factors as $H_t=f\circ h_t$ where $h_t$ is a smooth proper time dependent function satisfying $|\nabla h_t|\leq 1$ we have
\begin{equation}\label{eqTimeDepOsc}
|h_b(\gamma(b))-h_a(\gamma(a))|^2\leq (b-a)\left(\int_a^b\|g(t)X_{H}-\gamma'(t)\|^2 dt+\sup_{t\in[a,b]}\partial_t h_t\circ\gamma(t)\right).
\end{equation}
\end{rem}
\begin{proof}
We have
\begin{align}
|h(\gamma(b))-h(\gamma(a))|^2&=\left|\int_a^b\langle\nabla h,\gamma'(t)\rangle dt\right|^2\notag\\
&=\left|\int_a^b\langle\nabla h,\gamma'(t)-g(t)X_H\rangle dt\right|^2\notag\\
&\leq(b-a){\int_a^b\|g(t)X_{H}-\gamma'(t)\|^2 dt}\notag.
\end{align}
We used  Cauchy-Schwartz, $\|\nabla h\|\leq 1$, and the fact that $X_H\perp\nabla h$.
\end{proof}
\begin{lm}\label{tmGradFacComp}
Suppose $H$ is smooth time independent and $J$-proper. Then the metric $g_{J_H}$ on $\tilde{M}:=\R\times \R/\Z\times M$ is complete.
\end{lm}
\begin{proof}
Let $H=f\circ h$ where $h:M\to\R$ is proper and satisfies $\|\nabla h\|\leq 1$. We show that the pullback $\tilde{h}$ of $h$ to $\tilde{M}$
is Lipschitz with respect to $g_{J_H}$. It suffices to show that for any path $\tilde{\gamma}:[a,b]\to \tilde{M}$ lifting a path $\gamma:[a,b]\to M$ we have
\[
|\tilde{h}(\tilde{\gamma}(b))-\tilde{h}(\tilde{\gamma}(a))|^2 \leq (b-a)\int\|\tilde{\gamma}'\|^2_{g_{J_H}}.
\]
For each $x\in[a,b]$ we can $g_{H_J}$-orthogonally decompose
\[
\tilde{\gamma}'(x)=v(x)+g(x)(X_H+\partial_t)+\partial_s,
\]
where $v(x)\in TM$. We have
\[
\|\tilde{\gamma}'(x)\|^2_{g_{H_J}}\geq\|v(x)\|^2=\|\gamma'(x)-g(x)X_H\|^2_{g_J}.
\]
Since $\tilde{h}$ is independent of $s$, the claim follows by Lemma~\ref{lmStrPrpEst}.

To see that $g_{J_H}$ is complete note first that by translation invariance it suffices to prove completeness of the restriction of $g_{J_H}$ to the mapping torus $s=const$. For this note that the restriction of $\tilde{h}$ to the set $s=const$ is still Lipschitz, and, moreover, it is proper since $H$ is. Thus for any $x$ we have that the ball of radius $R$ around $x$ in $\R/\Z\times M$ is contained in the compact subset
\[
\tilde{h}^{-1}([\tilde{h}(x)-R,\tilde{h}(x)+R]).
\]
Completeness now follows by Hopf-Rinow.
\end{proof}

We conclude with a criterion for $J$-properness. Call a function $f:\R\to[1,\infty)$ tame if the primitive of $\frac1{f}$ is unbounded from above.
\begin{lm}\label{lmCompCrit}
Suppose there is a tame function such that
\[
\|\nabla H\|_{g_J}\leq f(H).
\]
Then $H$ is $J$-proper.
\end{lm}
\begin{proof}
Let $g$ be a primitive of $\frac1{f}$. We have
\[
\|\nabla (g\circ H)\|=g'\circ H\|\nabla H\|=\frac1{f\circ H}\|\nabla H\|\leq 1.
\]
By assumption $h:=g\circ H$ is proper. Moreover, $g$ is monotone (primitive of a positive function). So $H=g^{-1}\circ h$.
\end{proof}

\subsection{Curvature}We introduce some notation and recall some basic formulae in Riemannian geometry. We refer to \cite{pe} for details. Let $(M,g)$ be a Riemannian manifold and let $r:M\to\R$ be a distance function. That is, a function satisfying $\|\nabla r\|=1$. Write $\partial_r:=\nabla r$ and denote by $S$ the tensor $\nabla\partial_r$. Denote by $U_r$ the level sets of $r$. Denote by $R$ the curvature tensor of $M$, by $R^t$ the tangential component of $R$ restricted to $TU_r$ and by $R^r$ the curvature tensor of $U_r$. Also, write $R_{\partial_r}=R(\cdot,\partial_r)\partial_r.$

The following formulae, together with the symmetries of the curvature tensor, show that the full curvature tensor on $M$ is determined by the curvature of the level sets of $r$, by the tensor $S$ and by its first derivative.
\begin{equation}\label{eqRiem1}
-R_{\partial_r}=S^2+\nabla_{\partial_r} S,
\end{equation}
\begin{equation}\label{eqRiem2}
R^t(X,Y)Z=R^r(X,Y)Z-S(X)\wedge S(Y)Z,
\end{equation}
\begin{equation}\label{eqRiem3}
R(X,Y)\partial_r=(\nabla_XS)(Y)-(\nabla_YS)(X).
\end{equation}
The vectors $X$, $Y$ and $Z$ in the above formulae are all tangent to $U_r$. In what follows, given a vector $V\in TM$  we will use the notation $\theta_{g,V}$ for the dual to $V$ with respect to $g$ and will drop $g$ from the notation when there is no ambiguity. We utilize the following formula for the covariant derivative of a vector field $X$
\begin{equation}\label{eqCovDer}
2\theta_{g,\nabla X}=d\theta_{g,X}+\mathcal{L}_Xg.
\end{equation}
This formula presents the decomposition of $\theta_{g,\nabla X}$ into a symmetric and an anti-symmetric bilinear form. For a proof see \cite[p. 26]{pe}.

Let $\mathfrak{H}=Hdt$, where $H:M\to\R$ is smooth. Since $g_{J_H}$ is translation invariant with respect to $s$, we restrict attention to sub-manifolds of $\R\times \R/\Z\times M$ with fixed values of $s$, or, in other words, to $\R/\Z\times M$ with the metric $g_{J_H}$ as computed in Example~\ref{ExTransInv}. The function $t$ (which is locally single valued) is a distance function on $\R/\Z\times M$ with respect to this metric. To see this, note that by \eqref{eqExTransInv} we have that $dt=g_{J_H}(X_H+\frac{\partial}{\partial t},\cdot)$. That is, $\nabla t=X_H+\frac{\partial}{\partial t}$. One verifies that $\|X_H+\partial_t\|_{g_{J_H}}^2=1.$
%Indeed for any $Y=(v_1,v_2)\in TS^1\times M$ we have,
%\[
%g_{J_F}(X_H+\partial_t,Y)= g_J(X_H,v_2)+g_j(\partial_t,v_1)-g_J(X_H,X_H)dt(v_1)-g_J(X_H,v_2)-g+2\|X_H\|^2dt(v_1)
%\]

\begin{tm}\label{tmCurvXH}
 We have $\nabla\nabla t=\frac12(\nabla^{g_J} X_H+\nabla^{g_J} X_H^*)\circ\pi$ where the superscript denotes conjugation with respect to the metric $g_J$ and $\pi:T(\R/\Z\times M)\to TM$ is the $g_{J_{\mathfrak{H}}}$ orthogonal projection.
\end{tm}
\begin{proof}
Write $N=\nabla t$. By equation~\eqref{eqCovDer} we have
\[
2\theta_{\nabla N}=d\theta_{ N}+\mathcal{L}_{ N}g_{J_F}.
\]
Since $\theta_ N=dt$, we have $d\theta_{ N}=0$. We claim that $\mathcal{L}_Ng_{J_F}=\pi^*\mathcal{L}_{X_H}g_J$. To see this denote by $\psi_t$ the time $t$ flow of $X_H$ and let
\[
\phi:(-\epsilon,\epsilon)\times M \subset \R/\Z\times M\to \R/\Z\times M,
\]
be the map $(t,p)\mapsto (t,\psi_{t}(p))$. Then $\phi_*|_{T(\{t_0\}\times M)}=\psi_{t_0,*}$ and $\phi_*\partial_t=\partial_t+X_H=N$. In particular, $\phi^*g_{J_F}|_{\R/\Z\times M}=\pi^*\psi^*_tg_J+dt^2$. Thus,
\begin{align}
\phi^*\mathcal{L}_{ N}g_{J_F}&=\mathcal{L}_{\partial_t}\phi^*g_{J_F\emph{}}\notag\\
&=\partial_t(\pi^*\psi^*_tg_J+dt^2)\notag\\
&=\pi^*\psi_t^*\mathcal{L}_{X_H}g_J\notag\\
&=\phi^*\pi^*\mathcal{L}_{X_H}g_J.
\end{align}

By~\eqref{eqCovDer} we have
\[
\mathcal{L}_{X_H}g_J=[\theta_{\nabla X_H,g_J}],
\]
where $[\alpha(\cdot,\cdot)]$ denotes the symmetrization. Thus,
\[
S=\frac12(\nabla^{g_J} X_H+\nabla^{g_J} X_H^*)\circ\pi.
\]

\end{proof}
We say that a Hamiltonian $H:M\to\R$ is Killing (with respect to some compatible almost complex structure $J$) if the flow of $X_H$ preserves $g_J$.
\begin{cy}\label{cyKilling}
Suppose $H$ is Killing, then $\nabla\nabla t\equiv 0$.
\end{cy}
\subsection{Injectivity radius}\label{subsubinj}
We turn to discussing the control of the radius of injectivity of $g_{J_H}$. In the following lemmas fix a point $x_0=(s_0,t_0,p_0)\in \R\times \R/\Z\times M$.
\begin{lm}\label{lmVolEst1}
 For any $r<\frac12$ We have
\[
\vol_{g_{J_H}}(B_r(x_0))> \frac{r^2}{9}\vol_{g_J}(B_{r/3}(p_0)).
\]
\end{lm}
\begin{proof}
Denote by $\psi_t$ the Hamiltonian flow of $H$. Since $X_H+\partial_t$ is perpendicular with respect to $g_{J_H}$ to hypersurfaces of constant $t$, we have that $B_r(x_0)$ contains the cylinder
\[
C\quad=\bigcup_{t\in [t_0-r/3,t_0+r/3]} [s_0-r/3,s_0+r/3]\times\{t\}
\times\psi_t(B_{r/3}(p_0))\]
Since $\psi_t$ preserves the $g_J$-volume we have
\[
\vol_{g_{J_H}}(C)=\frac{r^2}{9}\vol_{g_J}(B_{r/3}(p_0)).
\]
\end{proof}
\begin{lm}\label{lmCheeger}
Let $(M,g)$ be an $n$ dimensional Riemannian manifold. Let $a>0$ and let $p\in M$ such that
\[
\vol_g\left(B_{\frac1a}(p)\right)\geq v_0\left(\frac1{a}\right)^n
\]
and such that $|Sec_g(x)|\leq a^2$ on $B_{\frac1a}(p)$. Then there is a constant $f=f(v_0,n)$, independent of $a$, such that $inj_g(p)\geq f(v_0,n)$.
\end{lm}
\begin{proof}
For $a=1$, this is an immediate consequence of  \cite[Theorem 4.3]{CGT1982}. The claim follows by scaling.
\end{proof}
\begin{lm}\label{lmVolEst2}
Suppose $(M,g)$ satisfies for some $p\in M$ that $inj_g(p)\geq a$ and $|Sec_g(x)|\leq a^2$ on $B_{\frac1a}(p)$. Then there is a constant $C=C(n)>0$ such that
\[
\vol_g(B_{1/a}(p))\geq C\left(\frac1{a}\right)^n.
\]
\end{lm}
\begin{proof}
By scaling, the claim is equivalent to the claim that there is a constant $C(n)>0$ such that a geodesic ball of radius $1$ with sectional curvature bounded by $1$ has volume at least $C(n)$. By the Jacobi equation, sectional curvature controls the derivatives of the metric in geodesic coordinates \cite[Ch. 5]{DC}.  In particular there is an a priori estimate from below on the determinant of the metric in these coordinates for a small enough ball around the origin. The claim follows.
\end{proof}
\begin{tm}\label{injEst}
There is a constant $i=i(n)$ such that if $g_J$ is $a$-bounded at $p_0\in M$ then $inj_{g_{J_H}}(x)\geq\frac{i(n)}{a}$.
\end{tm}
\begin{proof}
Combining Lemmas~\ref{lmVolEst1} and~\ref{lmVolEst2} we have that there is a constant such that
\[
\vol_{g_{J_H}}(B_{1/a}(x))\geq \frac{1}{3^{n+2}}C(n)\left(\frac1{a}\right)^{n+2}.
\]
The claim follows by Lemma~\ref{lmCheeger}.
\end{proof}
\subsection{Some criteria for boundedness}\label{SubsecCritBounded}
\begin{lm}
Suppose $g_J$ is $a$-bounded at $p\in M$ and $H$ is a time independent Hamiltonian such that
\begin{equation}\label{eqCovEst}
\max\left\{\left\|\nabla X_H(p)\right\|^2,\left\|\nabla^2X_H\right\|,\left\|\nabla_{X_H}\left(\nabla X_H+\nabla X_H^T\right)\right\|\right\}<a^2.
\end{equation}
Then for a constant $c=c(n)$ independent of $a$, we have that $g_{J_H}$ is $ca$-bounded at $p$.
\end{lm}
\begin{proof}
We need to estimate the sectional curvature and radius of injectivity of $g_{J_H}$. Up to multiplication by a constant dependent on $n$, estimating sectional curvature is the same as estimating the coefficients of the curvature tensor in an orthonormal basis. Since $J$ is $a$-bounded, it remains to estimate only coefficients involving the direction $\partial_t+X_H$ at least once. In light of formulae~\eqref{eqRiem1}-\eqref{eqRiem3} we need to estimate $\nabla S$ and $S^2$ where $S=\nabla t$. Theorem~\ref{tmCurvXH} provides us with an estimate on $S^2$ and the tangential restriction of $\nabla S$ in terms of $\nabla X_H$ and $\nabla^2 X_H$. It remains to estimate the right hand side of \eqref{eqRiem1}. For this it is preferable to use the formula
\[
-R_{N}=\mathcal{L}_{ N}S-S^2.
\]
See \cite{pe}. Each summand vanishes on $N$. So it remains to estimate $\mathcal{L}_{ N}S$ applied to a tangential vector. Let $V$ be a tangential vector field which commutes with $N.$ Then
\begin{align}
\mathcal{L}_{X_H+\partial_t}(SV)&=\mathcal{L}_{X_H}(SV)\notag\\
&=\nabla^{g_J}_{X_H}({S}V)-\nabla^{g_J}_{SV}(X_H)\notag\\
&=(\nabla^{g_J}_{X_H}{S})V+S(\nabla^{g_J}_V{X_H})-\nabla^{g_J}_{SV}(X_H)\notag.
\end{align}
This shows that estimate~\eqref{eqCovEst} implies $Sec_{g_{J_H}}(p)\leq c^2a^2$ for an appropriate $c=c(n)$. Theorem~\ref{injEst} provides us with the estimate on $inj_{g_{J_H}}$ in terms of $inj_{g_J}(p)$. The claim follows.
\end{proof}

\begin{ex}\label{ExLiouvTame}
Let $(\Sigma,\alpha)$ be a contact manifold and let
\[
(M=\R_+\times\Sigma,\omega=e^r(d\alpha+dr\wedge\alpha ))
\]
be the convex end of its symplectization. Denote the Reeb vector field on $\Sigma$ by $R$. Fix an $\omega$-compatible translation invariant almost complex structure $J$ satisfying $JR=\partial r$. Then
\begin{equation}\label{eqMetConvSymp}
g_J=e^r(dr^2+g_\Sigma),
\end{equation}
for some metric $g_\Sigma$ on $\Sigma$. Since the metric $g_J$ scales up, the radius of injectivity of $g_J$ is bounded away from $0$, and in fact goes to $\infty$ with $r$. Pick local coordinates on $\Sigma$ and use the function $r$ as the coordinate on the $\R_+$ factor. Then  the Christoffel symbols of the metric \eqref{eqMetConvSymp} are $O(1)$ in these coordinates. Therefore
\[
\langle \nabla_{\partial_i}\partial j,\partial_k\rangle\sim e^r.
\]
Since $\|\partial_i\|^2\sim e^r$, for some constant $C$ we have $\|\nabla \partial_i\|\leq C$. Similarly,
\[
\|\nabla^2_{ij}\partial_k\|^2\sim e^r,
\]
allowing us to deduce that
\[
\|\nabla^2\partial_k\|\sim e^{-r/2}.
\]
Suppose $H$ is a function on the symplectization which is given outside of a compact set by $H=h(e^r)$ then there are some constants $a_i$  such that
\[
X_H=h'(e^r)\sum a_i\partial_i.
\]
First suppose $h'(e^r)$ is constant. Then by the reasoning above, we conclude that \textit{the induced metric $g_{J_H}$ is uniformly bounded for Hamiltonians which are linear at infinity with a bound that is proportional to the slope $h'(e^r)$.}

\end{ex}

\begin{ex}\label{wxAdmWT}
Continuing with the previous example, Assume now $h'(e^r)$ is at most linear in the distance from $\Sigma$. \textit{Caution: this means it is at most linear in $e^{r/2}$.} Then there is a bound on the geometry of $g_{J_H}$ which is also linear in the distance. To see this note that for a point $p$ which is a distance $d$ from $\Sigma$, the metric $g_{J_H}$ is uniformly equivalent on a neighborhood of size $1$ to the metric associated with the slope $h'(e^{r(p)})$. It follows by example \ref{exIboundTam} that the metric associated to $H$ is $i$-bounded. Note that while this allows superlinear Hamiltonians, it does not include quadratic Hamiltonians $h'\sim e^r$.

\end{ex}
\begin{ex}\label{exSasaki}
Consider the case of the cotangent bundle $T^*M$ of a compact manifold $M$, let $g$ be a Riemannian metric on $M$ and let $J$ be the Sasaki almost complex structure on $T^*M$. It is defined as follows: the Levy-Civita connection on $T^*M$ induces a splitting $TT^*M=V\oplus H$ into horizontal and vertical vectors. Moreover, we take  $J:V\simeq H$ to be the natural isomorphism identifying an element of $V$ with an element of $T^*M$, then via $\omega$ with an element of $TM$ and finally with an element of $H$ via horizontal lifting. Identifying $TM=T^*M$, in standard local Darboux coordinates $\{q_i,p_i=dq_i\}$, where $q_i$ are geodesic coordinates centered at a point $q$, $J$ is given in the fiber over $q$ by
\[
J\frac{\partial}{\partial p_i}=\frac{\partial}{\partial q_i}.
\]
Then it is easy to show that the metric $g_J$ is $\|p\|$-bounded at the point $(p,q)$. In particular, $J$ is i-bounded (but not bounded). Consider a Hamiltonian of the form  $H=\sqrt{a|p|^2+V\circ\pi}$, where $\pi:T^*M\to M$ is the standard projection and $V:M\to\R$ is smooth, $J_H$ is i-bounded. Indeed, denoting by $M$ the maximum of $\sqrt{|V|}$ over $M$ we have in local coordinates as above
\[
\|X_H\|=aM\frac1{2\|p\|}\left\|\sum_ip_i\frac{\partial}{\partial q_i}\right\|\leq aM.
\]
So, the claim follows from Lemma~\ref{lmQuasHam}. Note that mechanical Hamiltonians of the form $|p|^2+V\circ\pi$ are not i-bounded with respect to the Sasaki metric.

\end{ex} 
\section{Loopwise dissipativity}\label{Sec5}
\subsection{Diameter control of Floer trajectories}
Suppose $(H,J)$ is i-bounded, let $u$ be a Floer trajectory, and let $\tilde{u}$ be its graph. Suppose that for some precompact $U\subset \R\times \R/\Z$, we have control over $u(\partial U)$. Theorem~\ref{tmDiamEst} above then provides with control over $u(U)$ in terms of
\[
E(\tilde{u};U)=E(u;U)+\area(U).
\]
This indicates that the only source of non-compactness in the moduli space of finite energy Floer trajectories comes from the potential existence of finite energy solutions with one end converging to infinity. This motivates the following definition.

We refer hence to a Floer solution on a possibly finite cylinder $[a,b]\times \R/\Z$ as a \textit{partial Floer trajectory}. For $H:\R/\Z\times M\to\R$ proper and an $\omega$-compatible almost complex structure define a function $\g_{H,J}(r_1,r_2)$ as the infimum over all $E$ for which there is a partial Floer trajectory $u$ of energy $E$  with one end of $u\times t$ contained in $H^{-1}([-r_1,r_1])$ and the other end in $H^{-1}(\R\setminus (-r_2,r_2))$. Note that $\g_{H,J}(r_1,r_2)$ may take the value of infinity.
\begin{df}
We say that $(H,J)$ is \textbf{loopwise dissipative (LD)} if for any fixed $r_1$ we have $\g_{H,J}(r_1,r)\to\infty$ as $r\to\infty$. If this holds for some function $\g:\R_+\times\R_+\to\R_+\cup \{\infty\}$ satisfying $\g_{H,J}\geq \g$ we say that $(H,J)$ is $\g$-LD. We say that $(H,J)$ is \textbf{robustly loopwise dissipative (RLD)} if there is a function $\g:\R\times\R\to\R$ and an open neighborhood of $(H,J)$ in $C^1\times C^0$\footnote{Here and hereafter the topology can be taken to be the uniform topology with respect to $g_J$. However what is truly necessary is that open sets are sufficiently thick to allow perturbations for achieving regularity.} all elements of which are $\Gamma$-LD.
\end{df}

\begin{df}
Denote by $\mathcal{F}_d^{(0)}(M)$ the set of i-bounded Floer data $(H,J)$ which are RLD. Elements of $\mathcal{F}_d^{(0)}(M)$ are referred to as \textbf{dissipative Floer data}.
\end{df}
\begin{comment}
It is sometimes convenient to have the following characterisation of loopwise dissipativity. We omit the proof.
\begin{lm}\label{lmLDCharAlt}
$(H,J)$ is LD if and only if any one of the following conditions holds. \begin{enumerate}

\item There is an exhaustion of $\R/\Z\times M$ by a sequence $\{K_i\}$ of compact sets such for any $E\geq 0$ and for any natural $i$ there is an $i'(i,E)$ such that if $u$ is a partial solution with one end of $t\times u$ contained in $K_i$ and  satisfying $E(u)<E$ then the other end intersects $K_{i+i'}$.
    \item For any compact $K\subset M$ there is an $R=R(K,E)$ such that if $u$ is a partial solution with one end contained in $K$, the other end meets $B_R(K)$.
    \end{enumerate}
\end{lm}
\end{comment}
\begin{comment}
\begin{proof}
For the first characterization, since $H:\R/\Z\times M\to\R$ is assumed to be proper, we may take $K_i:=p_2(H^{-1}([-i,i]))$ where $p_2$ is the projection $\R/\Z\times M\to M$. The forward implication is now obvious.

For the second characterization find the smallest $i$ such that $K\subset K_i$ and let $R$ such that $K_{i+i'(E,K_i)}\subset B_R(K)$.
\end{proof}
\end{comment}

Our next Theorem shows that dissipativity is all we need for diameter control. In the ensuing sections we  show both that on a geometrically bounded manifold there is always a sufficient supply of dissipative Floer data and that this property can be verified directly in various settings.

In the following theorem recall Definitions \ref{dfFloDataWbounded} and \ref{dfFloDataWboundedFamily} of an i-bounded Floer datum and family of Floer data.
\begin{tm}\label{tmFloSolDiamEst}
Let $(\mathcal{S},F_{s\in S}=(\mathfrak{H}_{s},J_{s}))$  be a uniformly i-bounded family of connected (broken) Riemann surfaces decorated with Floer data and equipped with a thick thin decomposition as in Definition \ref{dfFloDataWboundedFamily}. Let $(H_i,J_i)\in\mathcal{F}_d^{(0)}$ be Floer data such that on the $i$th component of $Thin_\cS$, we have that $F_s$  coincides with $(H_i,J_i)$ for all $s\in\mathcal{S}$. Then for any compact $K\subset M$ and any real number $E$, there is an $R=R(E,K)$  such that for any $s\in\mathcal{S}$, any $F_s$-Floer solution $(\Sigma_s,u)$  with $E(u)\leq E$ and intersecting $\partial K$ is contained in $B_R(K)$. Moreover, if $K$ is a level set of $H$ with no degenerate periodic orbits in a neighborhood of $\partial K$, we can take $R(E,K)\to 0$ as $E\to 0$.
\end{tm}
\begin{df}\label{dfDissFam}
We refer to families $(F_{s\in\mathcal{S}}=(\mathfrak{H}_{s},J_{s}))$ satisfying the hypotheses of Theorem~\ref{tmFloSolDiamEst} as \textbf{dissipative families}. When $\mathcal{S}$ consists of a single element, we refer to it as a \textbf{dissipative Floer datum}.
\end{df}
\begin{proof}[Proof of Theorem \ref{tmFloSolDiamEst}]
By assumption, we can decompose $\Sigma_s$  into a thick part consisting of components $A_i$ with  area bounded by a constant $C_{\cS}$, independently of $s\in\cS$, and a thin part consisting of components $B_i$ on which the Floer data are translation invariant and given by $(H_i,J_i)$. Moreover, there is a number $N_{\cS}$ which bounds the number of components independently of $s\in\cS$. For each $A_i$, the graph $\tilde{u}=id\times u :A_i\to A_i\times M$ is $J_F$-holomorphic and satisfies $E(\tilde{u})\leq E+Area(A_i)$. Furthermore we may assume for some $\epsilon>0$ we have $J_F|_{{B}_{\epsilon}(A_i)\times M}\in\mathcal{J}_{i.b.}({B}_{\epsilon}(A_i)\times M, \omega_\mathfrak{H})$.

We construct an $R_0=R_0(E,K)$ such that if for some component $A_i$ we have $u(A_i)\cap K\neq\emptyset$ then $u(A_i)\subset B_{R_0}(K)$ and moreover for all components $B_j$ which  share a boundary with $A_i$, we have $u(B_j)\subset B_{R_0}(K)$. A similar $R_0$ can be constructed starting with a component $B_i$. Since the total number of components is bounded by $N_{\cS}$ this will inductively give rise to an $R$ as in the statement of the present theorem.

To construct $R_0$, suppose $u(A_i)$ intersects some compact set $K$. Then, since $\tilde{u}$ extends to a neighborhood of $A_i$, by Lemma \ref{lmDiamEstSec} we have that $\tilde{u}(A_i)$ is contained in a ball $B_{\tilde{R}}(K\times A_i)$  for some $\tilde{R}=\tilde{R}(K,E+2C_\cS)$ depending additionally on the taming data associated with $F_s$ and thus on $\cS$. From this we deduce the same for $u(A_i)$ with perhaps a different radius $R$. It follows that each of the components $B_j$ whose closure intersects $A_i$ has a boundary component contained in $B_R(K)$. Let $a_j$ such that $B_R(K)\subset H_j^{-1}([-a_j,a_j])$. By loopwise dissipativity there is a $b_j>a_j$ such that $\Gamma_{H_j,J_j}(a_j,b_j)>E$. Writing $B_j=I\times \R/\Z$ for some interval $I$, we have, by definition of $\Gamma_{H_j,J_j},$  that $u(\{s\}\times \R/\Z)$ intersects $H^{-1}(-b_j,b_j)$ for each $s\in I$.  Restricting $u$ to $(s-1,s+1)\times \R/\Z$ and invoking Lemma~\ref{lmDiamEstSec} again, we obtain an $R'$ such that for any $s\in I$, we have $u((s-1,s+1)\times \R/\Z)\subset B_{R'}(H^{-1}(-b_j,b_j))$.  It follows that the same holds for $u(B_j)$. Now take $R_0$ such that $B_{R'}(H^{-1}(-b_j,b_j))\subset B_{R_0}(K)$ for each $j$.

For the last statement of the Theorem we rely on the following property of Floer trajectories which is stated in \cite{Salamon1999}. There are constants $c$ and $\hbar$ such that
\[
\int_{B_{r}(s,t)}\|\partial_su\|^2<\hbar\quad\Rightarrow\quad\|\partial_su\|^2(s,t)<\frac8{\pi r^2}\int_{B_{r}(s,t)}\|\partial_su\|^2+cr^2.
\]
Once we know that a solution is contained in an a-priori compact set, we can take all the constants to be fixed by that compact set. By taking $r=E(u)^{1/4}$ we deduce that for an appropriate constant
\[
\|\partial_tu-X_H\|^2=\|\partial_su\|^2<CE(u)^{1/2},
\]
once $E(u)$ is small enough. It follows that making $E(u)$ arbitrarily small, $u$ will be contained in an arbitrarily small neighborhood of some periodic orbit.
\end{proof}

We conclude this subsection with a counterexample showing that geometric boundedness alone does not guarantee loopwise dissipativity.
\begin{ex}\label{ExDegAtInfty}
Consider $(M,\omega)=(\R\times \R/\Z,ds\wedge dt)$. Let $H$ be a smoothing of the function
\[
(s,t)\mapsto s-\ln (|s|+1),
\]
and let $J$ be multiplication by $i$. Then $\|X_H\|$ is bounded, so $(H,J)$ is i-bounded by Lemma \ref{lmQuasHam}. But it is not LD. Indeed, the map $u:\R_+\times \R/\Z\to M$ defined by $u(s,t)=(\ln (s+1),t)$ is an $(H,J)$-partial Floer trajectory of finite energy and infinite diameter.
\end{ex}

\subsection{Hamiltonians with small Lipschitz constant}
\begin{tm}\label{tmBoundGradDis}
Let $J$ be a geometrically bounded almost complex structure compatible with $\omega$.
There is an $\epsilon>0$ such that for any  Hamiltonian $H:M\to\R$ which is proper and satisfies with respect to $g_J$ that $\|X_{H}\|<\epsilon$  outside of some compact set, the datum $(H,J)$ is dissipative. The claim remains true when $H$ is $C^0$ close to a time independent Hamiltonian.
\end{tm}

The proof of Theorem \ref{tmBoundGradDis} is carried out in the end of this section.

\begin{lm}\label{lmModBd}
Let $u:[a,b]\times \R/\Z\to M$ be a differentiable map. Then we have
\[
(b-a)\geq\frac{\int_{t\in \R/\Z}d^2(u(a,t),u(b,t))}{\int_{[a,b]\times \R/\Z}\|\partial_su\|^2}.
\]
\end{lm}
\begin{proof}
By the Cauchy Schwartz inequality we have
\begin{align}
(b-a)\int_{[a,b]\times \R/\Z}\|\partial_su\|^2&\geq \int_{t\in \R/\Z}\ell^2(u(t\times [a,b]))dt\notag\\
&\geq \int_{t\in \R/\Z}d^2(u(a,t),u(b,t))\notag.
\end{align}
\end{proof}

\begin{lm}\label{lmLoopDisQuant}
Let $H:\R/\Z\times M\to\R$ be a proper smooth function. Suppose $J$ is a compatible almost complex structure. Suppose $H$ factors as $H=f\circ k$ where  $k:\R/\Z\times M\to\R$ is proper and has uniformly bounded gradient with respect to $g_J$ and $f:\R\to\R$ is monotone. Then $(H,J)$ is LD if and only if there exists a sequence $h_i\to\infty$ and a constant $\delta>0$ such that
\begin{equation}
\Gamma_{H,J}(h_{2i},h_{2i+1})>\delta.
\end{equation}
If $h_i,\delta$ can be fixed for an open neighborhood of $(H,J)$, it is RLD.
\end{lm}
\begin{proof}
The forward implication is obvious from the definition. For the other direction we use the following characterization of  loopwise dissipativity.  \textit{Let $K_i:=H^{-1}[-h_{2i},h_{2i}]\subset \R/\Z\times M$. For any $E\geq 0$ and for any natural $i$ there is an $i'(i,E)$ such that if $u$ is a partial solution with one end of $t\times u$ contained in $K_i$ and  satisfying $E(u)<E$ then the other end intersects $K_{i+i'}$.} For ease of exposition we assume for the rest of the proof that $H$ is time independent, the general case being similar. We prove loopwise dissipativity  by induction on the smallest integer $n$ bounding $E(u)/\delta$.

When $n=1$, this is just reformulating the assumption. Suppose we have proven the statement for all solutions $u$ satisfying $E(u)\leq n\delta$. Let $u$ be a solution with one end in $K_i$ and $E(u)\leq (n+1)\delta$. Without loss of generality we assume $u_a\subset K_i$. Here and henceforth $u_a:=u(a,\cdot)$. Let
\[
s_1=\inf \{s\in[a,b]:u_s\subset M\setminus K_{i+1}\}
\]
If this set is empty there is nothing to prove. Otherwise, let
\[
s_2=\inf\left(\left\{s\in[s_1,b]:\|\partial_su\|<1\right\}\cup\{b\}\right).
\]
Finally, take
\[
s_0=\sup\left(\{s\in[a,s_1]:\|\partial_su\|<1\}\cup\{a\}\right).
\]
We have
\[
E(u)\geq\int_{s_0}^{s_2}\|\partial_su\|^2ds>\int_{s_0}^{s_2}ds=s_2-s_0.
\]
So, by Lemma~\ref{lmModBd}, there is a $t\in \R/\Z$ such that
\begin{equation}\label{eqLoopDisQuant}
d(u_{s_0}(t),u_{s_2}(t))<E.
\end{equation}
We find an $i_0(i)$ such that $u_{s_0}\subset K_{i_0}$. Indeed, If $a=s_0$ there is nothing to prove.  Otherwise, we have $\|\nabla\mathcal{A}_H(u_{s_0})\|\leq 1$.  Since $s_0<s_1$ we have that $u_{s_0}$ intersects $K_{i+1}$. Factor  $H$ as $H=f\circ k$ as in the formulation of the present Lemma. Since $f$ is monotone,
\begin{equation}\label{eqHhEst}
\min_tk_t(u_{s_0}(t))<f^{-1}(h_{2(i+1)}),\quad \max_tk_t(u_{s_0}(t))>f^{-1}(-h_{2(i+1)}).
\end{equation}
From Lemma \ref{lmStrPrpEst}  we get an a priori estimate $c$ on the oscillation of $k$ on $u_{s_0}$ for the time independent case. Here $c$ depends only on the bound on $|\nabla k|$. In the time dependent case we appeal to  \eqref{eqTimeDepOsc} for this a priori estimate. Let $i_0$ satisfy
\[
h_{i_0}\geq\left\lceil\max\{f(f^{-1}(h_{2(i+1)})+c),-f(f^{-1}(-h_{2(i+1)})-c)\}\right\rceil
\]
Combined with \eqref{eqHhEst}, this gives the a priori estimate
\[
u_{s_0}\subset K_{i_0}.
\]
By \eqref{eqLoopDisQuant} we get from this an $i_1=i_1(i,E)$ such that $u_{s_2}$ meets $K_{i_1}$. If $s_2=b$ this  concludes the proof. Otherwise, as for $s_1$, we find an $i_2$ such that $u_{s_2}\subset K_{i_2}$. We have $E(u|_{[s_2,b]\times \R/\Z})\leq n\delta$ since $s_2>s_1$ and by the hypothesis of the Lemma $E(u|_{[a,s_1\times \R/\Z]})>\delta$. So, by the inductive hypothesis, there is an $i_3$ depending on $i$ and $n$ such that $u_b$ meets $K_{i_3}$. The first part of the claim now follows. The second part is clear since $i'(i,E)$ is constructed using only the data of $\{K_i\}$ and $\delta$.
\end{proof}
\begin{lm}\label{lmBoundGradDis}
Let $J$ be a geometrically bounded almost complex structure compatible with $\omega$.  There are constants $R,\epsilon,\delta$, depending on the bounds on the geometry of $g_J$, such that the following holds. Let $H:
M\to\R$ be a proper Hamiltonian satisfying for some $h\in\R$
\begin{equation}\label{eqGlobalGadBo}
\|X_{H}\|<\epsilon,\qquad \forall x\in H^{-1}([h,h+R]) .
\end{equation}
Then $\Gamma_{H,J}(h,h+R)>\delta$. This remains true if $H$ is merely assumed to be $C^0$ close to a time independent Hamiltonian. Moreover, the estimate is unaffected if $H$ is arbitrarily time dependent away from $ H^{-1}([h,h+R])$.

%then for any real number $h$ and any solution to Floer's equation on a cylinder $u:[a,b]\times \R/\Z$ with one end in $H^{-1}((-\infty,h))$ and the other end in $H^{-1}((h+R,\infty))$ we have
%\begin{equation}
%E(u)>\delta.
%\end{equation}
%The claim continues to hold for such $u$ if \eqref{eqGlobalGadBo} assumed to hold only on $H^{-1}([h,h+R])$.
\end{lm}
\begin{proof}
We first prove the claim when the left hand of \eqref{eqGlobalGadBo} is taken to hold for all $x\in M$. We consider the strictly time independent case, leaving the adjustments for the slightly more general case to the reader. For some $R>0$ let $u:[a,b]\times \R/\Z\to M$ be a solution to Floer's equation with one end in $H^{-1}(-h,h)$ and the other in $H^{-1}(\R\setminus [-h-R,h+R])$. Write $A=[a,b]\times \R/\Z$. %Let  $A$ be a component of $(H\circ u)^{-1}([h,h+R]\cup[-h-R,-h]).$ %such that
%\[
%\max_tH\circ u(a,t)<x_1<x_2<\min_t H\circ u(b,t).
%\]
Then by positivity of energy
\begin{equation}\label{eqEnGradFlow}
 E(u;A)=\int_A u^*\omega+\int_{\partial A}u^*Hdt\geq \left| \left|\int_A u^*\omega\right|-\left|\int_{\partial A}u^*Hdt\right|\right|.
\end{equation}
%The sign in \eqref{eqEnGradFlow} is determined according to whether it is the positive or the negative end on which the value of $H$ is greater.
We will show that if we take $\epsilon$ small enough, there are constants $\delta_1,\delta_2$ such that
\begin{equation}\label{eqSmallEn}
E(u;A)<\delta_1\Rightarrow \left|\int u^*\omega\right|<\delta_2.
\end{equation}
Since
\[
\left|\int_{\partial A}u^*Hdt\right|\geq R
\]
it will then follow from~\eqref{eqEnGradFlow} that if $R>2\delta_2$ then $E(u;A)>\min\{\delta_1,\delta_2\}$. This will prove the claim.

Let $\delta>0$ be so small that any loop of length $2\delta$ has diameter less than a tenth of the radius of injectivity of $M$ with respect to $g_J$. The isoperimetric inequality of Lemma \ref{lmIsopCurv} guarantees that any loop of length $<2\delta$ is fillable by a disk $v:D\to M$ such that
\[
 Area(v)<2\delta^2.
\]
We take $\epsilon=\delta$.
Given $u$ as above, and denoting by $\ell( u(s,\cdot))$ the length of the loop $t\mapsto u(s,t),$ let
\[
I=\{s\in[a,b]|\ell( u(s,\cdot))>2\delta\}.
\]
For any interval $(c,d)\subset I$ we have the estimate
\begin{equation}\label{eqIAE}
\left|\int_{(c,d)\times \R/\Z}u^*\omega\right|\leq Area\left(u|_{(c,d)\times \R/\Z}\right)\leq 3 E(u;{(c,d)\times \R/\Z}).
\end{equation}
The first of these is Wirtinger's inequality which says that for a compatible metric the symplectic area is dominated by the Riemannian area. Note that the Riemannian area is not sensitive to orientation while the symplectic area is. For the second, note that
\[
Area(u)\leq\int_{(c,d)\times \R/\Z}\|\partial_su\|\|\partial_tu\|\leq \frac12\int_{(c,d)\times \R/\Z}(\|\partial_su\|^2+\|\partial_tu\|^2),
\]
But,
\begin{align}\label{eqIAE2}
4\delta^2\leq\int_{\R/\Z}\|\partial_tu\|^2&\leq \int_{\R/\Z}\|\partial_tu-X_H\|^2+\int_{\R/\Z}\|X_H\|^2\\
&\leq \int_{\R/\Z}\|\partial_su\|^2+\epsilon^2 \notag\\
&=\int_{\R/\Z}\|\partial_su\|^2+\delta^2\notag\\
&\leq \int_{\R/\Z}\|\partial_su\|^2+\frac14\int_{\R/\Z}\|\partial_tu\|^2\notag.
\end{align}
So,
\[
\int_{\R/\Z}\|\partial_tu\|^2\leq2\int_{\R/\Z}\|\partial_su\|^2
\]
which implies the desired inequality.

Suppose now that
\[
\left|\int_{(a,b)\times \R/\Z}u^*\omega\right|>20 E(u;{(a,b)\times \R/\Z}),
\]
and that for some constant $c_2$ to be determined shortly, $E(u)<c_2\epsilon^2$. Denote by $\mu$ the Lebesgue measure on $\R$. Then by these hypotheses and  by equations \eqref{eqIAE} and \eqref{eqIAE2} we have $\mu(I)<\min\{\frac1{4}(a-b),c_2\}$. Let  $a'=\inf[a,b]\setminus I,b'=\sup [a,b]\setminus I$.
We will show that if $c_2$ is assumed small enough then
\begin{equation}\label{eqContrTh}
\left|\int_{(a',b')\times \R/\Z}u^*\omega\right|<4\delta^2.
\end{equation}
We then have
\begin{align}
\left|\int_{(a,b)\times \R/\Z}u^*\omega\right|&<4\delta^2+\left|\int_I u^*\omega\right|\notag\\
&<4\delta^2+\frac13E(u)\notag\\
&<4\delta^2+\frac1{60}\left|\int_{(a,b)\times \R/\Z}u^*\omega\right|.\notag
\end{align}
By picking $\delta_1=c_2\epsilon^2$ and $\delta_2=\min\{5\delta^2,20\delta_1\}$ we get that with these values \eqref{eqSmallEn} holds in any case.

It remains to prove \eqref{eqContrTh}.
Let
$[s_0,s_1]\subset [a,b]$ be any interval such that $s_0,s_1\not\in \{a,b\}$ and $s_1-s_0\leq 2\mu(I)$. Call such an interval \textit{admissible}. Denote by $u_{[s_0,s_1]}$ the restriction $u|_{[s_0,s_1]\times\R/\Z}$. Each component of the boundary of $u_{[s_0,s_1]}$ is contained in a geodesic ball $B_{\delta}(x_i)\subset M$. We claim that if $c_2$ is taken small enough, then
\begin{equation}\label{eqContrTh2}
u_{[s_0,s_1]}\subset B_{2\delta}(x_0)\cup B_{2\delta}(x_1).
\end{equation}
Indeed, otherwise there is a point $(s,t)\in [s_0,s_1]\times \R/\Z$ such that writing $x_2=u(s,t)$ we have $d(x_2,\{x_0,x_1\})>2\delta$. That is, the ball $B_{\delta}(x_2)\subset M$ does not meet the boundary of $u_{[s_0,s_1]}$. As in Lemma \ref{lmQuasHam}, the metric $g_{J_H}$ is quasi-isometric to the product metric of $g_J$ on $M$ with the flat metric on the cylinder, where we can take the quasi-isometry constant to equal $2$ if $\epsilon<2$. Thus we can apply the monotonicity inequality of Theorem \ref{tmMonontonicity} to obtain for an appropriate constant $c'$ which is independent of $\epsilon$,
\[
E(\tilde{u};[s_0,s_1]\times\R/\Z)=s_1-s_0+E(u;[s_0,s_1]\times\R/\Z)\geq c'\delta^2
\]
where $\tilde{u}$ is the graph of $u$.
This implies
\[
E(u)\geq c'\delta_2-c_2.
\]
Thus, if we take $c_2\leq \frac14c'\delta_2$ \eqref{eqContrTh2} follows.

Denote by $u^*$ a filling of $u_{[s_0,s_1]}$ by  discs contained in $B_{\delta}(x_i)$. Then  $u^*\subset B_{5\delta}(x_0)$ and so is contractible. In particular, the integral of $\omega$ over $u_{[s_0,s_1]}$ can be replaced by the integral of $\omega$ over these filling discs. Since $[a',b']$ can be subdivided into admissible intervals, and the integrals over the filling discs cancel in pairs for all but two, \eqref{eqContrTh} follows.

This proves the Theorem for the case  when the left hand side of \eqref{eqGlobalGadBo} is taken to hold for all $x\in M$.

For the more general case we argue as follows. Introduce the notation $K_x:=H^{-1}([-x,x])$. For some $R>0$ let $u:[a,b]\times \R/\Z\to M$ be a solution to Floer's equation with one end in $K_h$ and the other in $M\setminus K_{h+R}$. Let $[a',b']\subset [a,b]$ such that $u([a',b']\times \R/\Z)$ has one end in $K_{h+R/4}$ and the other in $M\setminus K_{h+3R/4}$. In each case assume the relevant boundary of $u([a',b']\times \R/\Z)$ meets the boundary of the region $K_x$. We separate into two cases. If $u([a',b']\times \R/\Z)\subset K_R$, we have the estimate $\|X_H\|<\epsilon$ for $u|_{[a',b']\times \R/\Z}$ and the entire argument goes through with no change. By taking $R$ big enough, the claim  follows since $\Gamma(h,h+R)\geq \Gamma(h+R/4,h+3R/4)$. Otherwise, for some $c\in\{a',b'\}$ we have that the oscillation of $H$ along  $u_c$ is at least $R/4$. Moreover, by the bound on $\nabla H$ inside $K_R$, a similar estimate applies to the diameter of $u_c$ with respect to the metric $g_{J_H}$. By the argument of Theorem \ref{tmDiamEst} this implies a lower bound on the energy $E(u;[a,b]\times \R/\Z)$. We spell out the details since the present case doesn't fit precisely into the stipulations of Theorem \ref{tmDiamEst}.

As above, denote by
\[
\tilde{u}: [a,b]\times \R/\Z\to\R\times \R/\Z\times M,
\]
the graph of $u$. Since $H$ has Lipschitz constant $\epsilon$ on $K$, Lemma \ref{lmCompCrit} implies the metric $g_{J_H}$ is equivalent on $K$ to the product metric with quasi-isometry constant depending only on $\epsilon$. Thus, by Theorem \ref{tmMonontonicity} there are constants $\delta_0,r_0,$ depending only on $J$ and $\epsilon$, such that for any point $x$ in the domain of $u$ for which
\[
(*)\quad A_x:=\tilde{u}^{-1}(B_{r_0}(x,u(x))\subset (a,b)\times \R/\Z,
\]
we have
\[
E(u;A_x)+Area(A_x)\geq \delta_0.
\]
Here we consider the ball $B_{r_0}(x,u(x)\subset\tilde{M}$ with respect to the metric $g_{J_H}$. Call a point for which the hypothesis $(*)$ holds a good point. Since the ends of the $u$ map entirely outside of $K_{h+R}$ it follows that any $x$ for which $u(x)\in K_{h+R-(1+\epsilon)r_0}$ we have that $A_x$ is a good point and moreover $A_x\subset B_{r_0}(x).$. Indeed, for any $x,x'$ in the domain we have
\[
d_{g_{J_H}}((x,u(x)),(x',u(x')))\geq (1-\epsilon)d_{g_J}(u(x),u(x')).
\]
For any $N$, by assuming $R$ is large enough, we can find $N$ good points $x_i\in \{c\}\times \R/\Z$  such that $d_{g_H}(\tilde{u}(x_i),\tilde{u}(x_j)>2r_0$. That is,  $A_{x_i}\cap A_{x_j}=\emptyset$ whenever $i\neq j$. We then have
\[
E(u)+2r_0\geq E(u;\cup A_{x_i})+Area(\cup A_{x_i}) \geq N\delta_0.
\]
By taking $N$ large enough so that $N\delta_0-2r_0>\delta$ for some chosen $\delta>0$ the claim follows.
\end{proof}

\begin{proof}[Proof of Theorem \ref{tmBoundGradDis}]
Lemmas \ref{lmBoundGradDis} and \ref{lmLoopDisQuant} imply that $(H,J)$ is RLD if $\|X_H\|<\epsilon$ for $\epsilon$ small enough.
To establish dissipativity, we need to prove, in addition, i-boundedness. This follows immediately from Lemma~\ref{lmQuasHam}.

\end{proof}
\subsection{Bidirectedness}
\begin{tm}\label{tmDiaFinCofin}
For \textit{any} smooth exhaustion function $H:M\times/\R/\Z\to\R$ and any geometrically bounded $\omega$-compatible almost complex structure $J$ there are exhaustion functions $H_+,H_-$ such that $(H_\pm,J)$ are dissipative Floer data and $H_-\leq H\leq H_+$ pointwise. In other words, the set of Hamiltonians which taken together with $J$ are dissipative Floer data is both final and cofinal in the set of all exhaustion functions.
\end{tm}
\begin{proof}%[Proof of Theorem \ref{tmDiaFinCofin}]
According to \cite{GreeneWu}  there exists an exhaustion function $f:M\to \R$ such that $\|\nabla f\|=\|X_f\|<\epsilon_0$ with respect to the metric $g_J$. Moreover, we may find a constant $R_0$ such that $d(f^{-1}(x),f^{-1}(x+R_0))$ is bounded away from $0$ for $x\in\R$. Indeed, $f$ can be taken to $C^0$ close to a multiple of the distance function to some point. So, $(f,J)$ is dissipative by Theorem~\ref{tmBoundGradDis}.  Let $h:\R\to\R$ be any monotone function such that $h'(x)=1$ on any of the intervals $[2nR,(2n+1)R)$ and arbitrary otherwise. Here $R$ is a constant as in Lemma \ref{lmBoundGradDis}, and without loss of generality $R>R_0$. Then the set of functions of the form $h\circ f$ is cofinal in the set of all exhaustion functions. On the other hand, $(h\circ f,J)$ is dissipative. Indeed, $h\circ f$ is clearly $J$-proper, since $f$ is. The metric $g_{X_{h\circ f}}$ is uniformly bounded on each of the regions $f^{-1}(h,h+R)$ by Lemma \ref{lmQuasHam}. So this metric is i-bounded.
Lemmas \ref{lmBoundGradDis} and \ref{lmLoopDisQuant} imply that $h\circ f$ is RLD. This complete the proof of cofinality.

By Theorem \ref{tmBoundGradDis} to prove finality it suffices to exhibit an exhaustion function $H_-\leq H$ which has sufficiently small gradient. Fix a point $p\in M$ and let $R_i$ be monotone increasing sequence such that $B_{R_i}$ contains $H^{-1}((-\infty,i])$. Denote by $h:M\to\R$ the distance function $h(x)=d(x,p)$. Define $a_i$ inductively by $a_0=0$ and
\[
a_i=\min\{i-1,a_{i-1}+R_i-R_{i-1}\}
\]
for $i\geq 1$. Let $f:\R_+\to\R$ be the piecewise linear function which is smooth at non-integer points and satisfies $f_i=a_i$ for $i\geq 1$. Note that $f$ is monotone increasing, proper and has slope at most $1$ wherever the slope is defined. So, the function $g=f\circ h$ is Lipschitz with Lipschitz constant $1$. Moreover, $g\leq H$ everywhere. The function $g$ can be $C^0$-approximated by a smooth function $k$ with $\|\nabla k\|\leq 2$ \cite{GreeneWu}. Then $k$ is an exhaustion function, so taking $H_-:=k/C$ for $C$ sufficiently large gives a function as required.
\end{proof}
\subsection{Dissipativity on exact symplectic manifolds}\label{SubSec52}
Let $(M,\omega=d\alpha)$ be an exact symplectic manifold. In this subsection we prove Theorem \ref{tmPalaiSmaleDiss} which is variant of Theorem \ref{tmFloSolDiamEst} that works on exact symplectic manifolds under slightly weaker hypotheses.  Fix an $\omega$-compatible almost complex structure and let $H:\R/\Z\times M\to\R$. The pair $(H,J)$ is said to be \textbf{Palais-Smale} if any sequence of loops $\gamma_n$ with $\mathcal{A}_H(\gamma_n)<c<\infty$ and $\|\nabla\mathcal{A}_H(\gamma_n)\|_{L^2}\to 0$ has a subsequence converging to a periodic orbit of $H$. If $J_0,J_1$ are almost complex structures which are quasi-isometric and $H_0,H_1$ are Hamiltonians such that $\|\nabla (H_0-H_1)\|$ converges to $0$ with respect to either then $(H_0,J_0)$ is Palais-Smale if and only if $(H_1,J_1)$ is.

\begin{lm}\label{lmPalSmDis}
Suppose $(H,J)$ is i-bounded and Palais-Smale. Then for any $c,d$  there is a real number $\ell$ and a compact set $K$ with the following significance. For any segment $[a,b]$ of length at least $\ell$ and any solution $u:[a-1,b+1]\times\R/\Z\to M$ to Floer's equation such that $\mathcal{A}_H(u(s,\cdot))\in[c,d]$ for $s\in [a-1,b+1]$, we have $u([a,b]\times\R/\Z)\subset K$.
\end{lm}
\begin{proof}
First, by the Palais-Smale condition, there are an $\epsilon>0$  and a compact $K'\subset M$ such that any loop $\alpha$ with $\|\nabla\mathcal{A}_H(\alpha)\|_{L^2}<\epsilon$ and $\mathcal{A}_H(\alpha)<d$ is contained in $K'$. Indeed the negation of this statement would allow us to produce a sequence of loops $\gamma_n$ satisfying the hypotheses of the Palais-Smale condition which nevertheless has no convergent subsequence. Suppose $b-a>(d-c)/\epsilon$. Then for all but a subset $I\subset[a,b]$ of total measure $(d-c)/\epsilon$ we have that $u(s,\cdot)$ is contained in $K'$. This follows  by the energy estimate
\[
d-c\geq\mathcal{A}_H(u(b+1))-\mathcal{A}_H(u(a-1))=\int_{a-1}^{b+1}\|\nabla\mathcal{A}_H(\alpha)\|_{L^2}.
\]
Indeed, taking $I$ to be the set of $s$ for which $u(s,\cdot)$ is not contained in $K'$, the right hand side of the last equation dominates 
\[
\int_I\|\nabla\mathcal{A}_H(\alpha)\|_{L^2}\geq \epsilon\int_Idt.
\]
It remains to control  $u(s,\cdot)$ for $s\in I$. Each connected component $I'$ of $I$  has at least one boundary point $s$ for which $u(s,\cdot)\subset K$. Moreover, $I'\subset I$ has a-priori bounded measure. Thus  applying part \ref{tmDiamEstc2} of Theorem~\ref{tmDiamEst} to the graph $\tilde{u}|_{I'\times \R/\Z}$, we deduce the image of $I'\cap [a,b]\times \R/\Z$ is contained  in some larger compact set $K$ depending only on $K'$ and $d-c$.
\end{proof}
\begin{comment}
\begin{cy}
If $(H,J)$ is Palais-Smale, it is LD.
\end{cy}
\begin{proof}
Combine the Lemma~\ref{lmPalSmDis} and Lemma~\ref{lmModBd}.
\end{proof}
\end{comment}

Note that the Palais-Smale condition produces by Lemma \ref{lmPalSmDis} an estimate which is slightly weaker than loopwise dissipativity because it depends not only on energy but also on action. Nevertheless, this is sufficient for proving the following variant of  Theorem \ref{tmFloSolDiamEst}.
\begin{tm}\label{tmPalaiSmaleDiss}
Let $(M,\omega=d\alpha)$ be an exact symplectic manifold. Let $(\mathcal{S},F_{s\in S}=(\mathfrak{H}_{s},J_{s}))$  be a uniformly i-bounded family of connected (broken) Riemann surfaces with a thick thin decomposition as in Definition \ref{dfFloDataWboundedFamily}. Let $(H_i,J_i)$ be Palais-Smale Floer data such that on the $i$th component of $Thin_\cS$, we have that $F_s$  coincides with $(H_i,J_i)$ for all $s\in\mathcal{S}$ . Then for any interval $[c,d]$ there is a compact set $K$  such that for any $s\in\mathcal{S}$ and any  solution $(\Sigma,u)$ associated with $F_s$, such that the actions of the periodic orbits on the ends all occur in the interval $[c,d]$, is contained in $K$.
\end{tm}
\begin{comment}
\begin{tm}
If $F$ and $(H_i,J_i)$ are as in Theorem~\ref{tmFloSolDiamEst} with the $(H_i,J_i)$ required to be Palais-Smale instead of the requirement $(H_i,J_i)\in\mathcal{F}_d^{(0)}$  , the conclusion of Theorem~\ref{tmFloSolDiamEst} holds.
\end{tm}
\end{comment}

\begin{proof}[Proof of Theorem \ref{tmPalaiSmaleDiss}]
First observe that without loss of generality we may assume the all the components of $Thin_\cS$ are of the form $I\times \R/\Z$ for $I$ an interval of length at least $\ell$ for $\ell$ as in Lemma \ref{lmPalSmDis}. Namely, with this assumption, the areas of the elements of thick are bounded a-priori in terms of $c,d$. Under such identification it is a consequence of Lemma \ref{lmTopGeoEnEst}  that for any $s\in I$ we have $\cA_{H_i}(u_s)\in[c,d]$. It thus follows from Lemma \ref{lmPalSmDis} that there is a compact set $K$ depending on $c,d$ only such that the images of  the components of $Thin_\cS$ are all contained in $K$. As a consequence, the image of each component $A$ of $Thick_\cS$ meets $K$. Since the energy of $A$ is at most $d-c$, we can  as in the proof of Theorem \ref{tmFloSolDiamEst} apply Lemma \ref{tmDiamEst}\ref{tmDiamEstc2} to the graph of $u|_A$ to obtain an $R=R(d-c)$ such that $u(A)\subset B_R(K)$.
\end{proof}
\begin{ex}\label{exAdmPS}
Let $\alpha$ be a primitive of $\omega$ and let $Z$ be the $\omega$-dual of $\alpha$. For any time independent Hamiltonian $H,$ define the function $f:M\to\R$ by
\begin{equation}\label{eqAcnFor}
f(x)=\omega(Z(x),X_H(x))-H(x).
\end{equation}
Suppose $f$ is proper and bounded below and $J$ is such that for some constant $C$, we have
\[
\|Z(x)\|^2<C{f(x)},
\]
outside a compact set. Then $H$ is Palais-Smale.
\begin{proof}
We have
\begin{align}\label{eqAcnEst}
\mathcal{A}_{H}(\gamma)&=\int_{\R/\Z}f(\gamma(t))dt+\int_{\R/\Z}\omega(Z(\gamma(t)),\gamma'(t)-X_H(t))dt\\
&\geq \int_{\R/\Z}f(\gamma(t))dt-\|\nabla\mathcal{A}_H(\gamma)\|\sqrt{C\int_{\R/\Z}f(\gamma(t))}dt.\notag
\end{align}
Suppose $\mathcal{A}_H(\gamma)\leq c$ and $\|\nabla\mathcal{A}_H(\gamma)\|\leq 1$. Since $f$ is proper, estimate~\eqref{eqAcnEst} implies
that there is a compact set $K$ depending on $c$ only such that $\gamma$ intersects $K$. Given a sequence $\gamma_n$ of loops intersecting $K$ such that
\[
\|\nabla\mathcal{A}_H(\gamma)\|_{L^2}\geq \int_{\R/\Z}\|X_H(t)-\gamma_n'(t)\|\to 0
\]
it is a standard fact that the sequence converges to an integral loop of $X_H$.
\end{proof}

In particular, consider the convex end of a symplectization $\R_+\times\Sigma$ as in Example~\ref{ExLiouvTame}. Denote by $r$ the coordinate on $\R_+$ and by $\sigma$ the coordinate on $\Sigma$. If $H$ satisfies
\[
\lim_{r\to\infty}e^r(\partial_rH)(e^r,\sigma)-H(e^r,\sigma)\to\infty,
\]
and $J$ is any almost complex structure satisfying
\[
e^r\alpha(J\partial r)\leq C(e^r\partial_rH(e^r,\sigma)-H(e^r,\sigma))
\]
for some $C$, then $H$  is Palais-Smale. This holds in particular for contact type $J$, i.e., satisfying $J\partial_r=R$ where $R$ is the Reeb flow of $\alpha$ on $\Sigma$.  After a $C^2$-small perturbation $(H,J)$ will satisfy the same estimates, so it will remain Palais-Smale. In addition, it will be non-degenerate.
\end{ex}

\begin{ex}\label{ConvEndDiss}
Continuing with the convex end of a symplectization, any function which is of the form $h(e^r)$ such that   $e^rh'(e^r)-h(e^r)\geq Ce^r$ for some constant $C$ is Palais-Smale. This holds, e.g., for $h(x)=x^{\alpha}$ with $\alpha>1$. If we have $h'(e^r)\leq e^{r/2}$ then by Example \ref{wxAdmWT} $H$ is dissipative. The cutoff appears to be $\alpha=3/2$ which unfortunately excludes quadratic Hamiltonians which are central in classical mechanics. See also the discussion in Example \ref{exSasaki}. Nevertheless, as we shall see below,  Floer cohomology can be defined by approximation by slow Hamiltonians. Moreover, similarly to the proof of Theorem \ref{rmTrivNonTriv}, it can be shown that for an arbitrary convex Hamiltonian the resulting Floer cohomology coincides with the Floer cohomology defined using contact type $J$ and relying on maximum principles.

When $e^rh'(e^r)-h(e^r)\to c<\infty$ for some $c$ which is not in the period spectrum, $H$ is still Palais-Smale even though this is not covered by the previous example, and in particular, it is dissipative. A proof of this fact is given below in Example \ref{exLinLiou}.

\end{ex}

\subsection{Some not-necessarily exact examples}\label{SubsecNonex}
Let $(M,g)$ be a Riemannian manifold, let $V$ be a time dependent vector field on $M$. For $p$ in $M$ define
\[
f(p,V,g):=\inf_{\{\gamma:[0,1]\to M|\gamma(0)=\gamma(1)=p\}}\left\{\int_0^1\|\gamma'(t)-V_t\circ\gamma(t)\|^2dt\right\}.
\]
Clearly, $f$ is continuous with respect to all variables in the $C^0$ norm. We drop $g$ from the notation when there is no ambiguity.

\begin{lm}\label{lmGeomFBound}
Let $(H,J)$ be such that $g_{J_H}$  has uniformly bounded geometry. Suppose there is a compact $K\subset M$ and a $\delta>0$ such that for all $p\in M\setminus K$ we have $f(p,X_H,g_J)\geq\delta$. Then $(H,J)$ is RLD.
\end{lm}
\begin{proof}
Let $u:[a,b]\times \R/\Z$ be a partial solution with one boundary in a compact set $K_0\supset K$  and with energy $E(u)\leq E$ for some $E$. Without loss of generality $u_a\subset K_0$. Suppose $u_b\subset M\setminus B_{R_0}(K_0)$ for some $R_0$. Then considering the graph of $u$ as a $J_H$-holomorphic map, it has energy $E+(b-a)$. Theorem \ref{tmDiamEst}\ref{tmDiamEstc1} applied to the compact set $\partial(B_{R_0}(K_0)\setminus K_0)$ then implies that for some constant $C$ depending on the bound on the geometry we have $R_0\leq C(E+(b-a))$. The assumption  on $f(p,X_H,g_J)$ implies $b-a\leq E/\delta$. Taken together we obtain the estimate $R_0\leq CE(1+1/\delta)$. % Otherwise, there is an $s\in [b-2E/\delta,b]$ such that $\|\nabla\mathcal{A}_H(u_s)\|^2 <\delta$. By Lemma \ref{lmModBd} we have that $d(u_s,u_b)\leq 2E^2/\delta$. If we choose $R>4E^2/\delta$  this implies that $u_s$ intersects $M\setminus B_{R/2}(K_0)$.
%We may assume  $\delta<R/2$. By Lemma~\ref{lmStrPrpEst} we deduce that $u_s$ is contained in $M\setminus K_0\subset M\setminus K$. But this leads to the contradiction $\delta\leq f(u_s(0),X_H,g_J)\leq \|\nabla\mathcal{A}_H(u_s)\|^2 <\delta.$ So either $R$ satisfies an a priori bound or $b-a\leq 2E/\delta$. But in the latter case, an appeal to Lemma \ref{lmModBd} again implies an a priori bound on $R$. It follows that $\Gamma_{H,J}(r_1,r)\to\infty $ as $r\to\infty$.

\end{proof}

The quantity $f(p,V,g)$ can sometimes be estimated from below by the following procedure. We say that the pair $(V,g)$ is of \emph{Lyapunov type} if there exists a constant $\lambda\geq 0$ such that for any $x,y\in M$ and any $t\geq 0$ we have, denoting by $\phi_t$ the time $t$ flow of $V$, 
\begin{equation}\label{eqLyapunovCond}
d_g(\phi_t(x),\phi_t(y))\leq e^{\lambda t}d_g(x,y).
\end{equation}
We refer to $\lambda$ as a Lyapunov constant for $V$.

\begin{lm}\label{lmLyapEstCovBo}
If $V$ is time independent and $\|\nabla V\|\leq\lambda$ then $\lambda$ is a Lyapunov constant for $V$.
\end{lm}
\begin{proof}
 For $x\neq y$ close enough and sufficiently short times, the function $h(t)=d(\phi_t(x),\phi_t(y))$ is differentiable. Moreover, for each $t$ there is a unique geodesic $s\mapsto \alpha_t(s)$ realizing the distance between $d(\phi_t(x),\phi_t(y))$. Denote by $\overline{V}_{\phi_t(x)}$ the parallel transport of $V_{\phi_t(x)}$ along $\alpha_t$.  Considering that the gradient of the distance function  $d(x,y)$ for, say, $x$ fixed  is the tangent vector to the unit speed geodesic  from $x$ to $y$ it follows that 
\begin{equation}
\frac{dh}{dt}=\langle\alpha'_t(1),V_{\phi_t(y)}\rangle-\langle\alpha'_t(0),V_{\phi_t(x)}\rangle. 
\end{equation}
From this we obtain the differential inequality
\begin{equation} 
\frac{dh}{dt}\leq |V_{\phi_t(y)}-\overline{V}_{\phi_t(x)}|\leq\lambda h.
\end{equation}

The claim for $x, y$ sufficiently close and for sufficiently short times now follows by Gr\"onwall's inequality. The claim for arbitrary $x,y$ and sufficiently short times  follows by the triangle inequality. The claim for arbitrary long time follows since the flow $\phi_t$ is autonomous. 
\end{proof}

\begin{lm}\label{lmLyap}
Let $V$ be a time independent vector field of Lyapunov type with Lyapunov constant $\lambda\geq 0$. Then
\begin{equation}\label{eqBoundL2L1}
d_g(p,\phi_1(p))^2\leq \frac{e^{2\lambda}-1}{2\lambda} f(p,V,g)\footnote{When $\lambda=0$ the coefficient on the right hand side tends to $1$.}.
\end{equation}
\end{lm}
\begin{rem}
For $V=X_H$ with $H$ time independent and uniformly Lipschitz we can replace the global requirement that \eqref{eqLyapunovCond} hold everywhere with the requirement that it hold for points $x,y\in U:=H^{-1}([H(p)-\epsilon,H(p)+\epsilon]$ for some $\epsilon>0$. We then get an estimate from below on $f(p,V,g)$ by combining the present Lemma to estimate the energy of loops which map into $U$ with Lemma \ref{lmStrPrpEst} to estimate the energy of the loops at $p$ which do not remain within $U$. Moreover, this estimate depends only on the Lipschitz constant of the restriction $H|_U$ and remains valid if $H$ is arbitrarily time dependent outside of $U$. 

\end{rem}
\begin{proof}[Proof of Lemma \ref{lmLyap}]
Fix some $\epsilon>0$ which will be later taken to be arbitrarily small. Let $\gamma:[0,1]\to M$ be a loop based at $p$. Let $r\in\R$ be small enough so that for each point $q\in \gamma([0,1])$ there is a chart $\left(U_q\subset M,\psi_q:B_{2r}(0)\to U_q\right)$ with coordinate map $\psi_q$  which is bi-Lipschitz with Lipschitz constant $1+\epsilon$. %Moreover assume the cover has Lebesgue number $r$. 
By compactness of $\gamma([0,1])$, there is a constant $K$ such that for any $q$ the vector field $d\psi_q^{-1} V$ considered as a map $B_{2r}(0)\to \R^{2n}$ is Lipschitz with constant $K$.

Write
\[
g(t):=\|\gamma'(t)-V_{\gamma(t)}\|
\]
and, 
\[
f(t)=\int_0^tg(s)ds. 
\]
Let
\[
\Delta t\ll r\min\left\{\frac1{K},\frac1{\sup\|\gamma(t)\|},\frac1{\sup\|V_{\gamma(t)}\|},\frac1{\max g(t)}\right\}
\]
Without loss of generality suppose $N:=\frac1{\Delta t}$ is an integer. Suppose $\Delta t$ is made smaller still so that $f(t)$ has an approximation by a piecewise linear function $h(t)$ such that 
\begin{equation}\label{eqHtagg}
(1-\epsilon)h'(t)<g(t)\leq h'(t),
\end{equation}
and such that $h$ is linear of slope $\epsilon_i$ on the intervals $[\frac{i}{N},\frac{i+1}{N}]$. Let $t_i=\frac{i}{N}$. Let $\gamma_i(t):=\phi_{t-t_i}(\gamma(t_i))$ and let $x_i=\gamma_i(1)$. Writing $\Delta x_i:=d(x_{i},x_{i-1})$ for $i=1,\dots N$, we have by the Lyapunov condition,
\[
\Delta x_i\leq e^{\lambda (1-t_i)}d_g\left(\gamma_i(t_i), \gamma_{i-1}(t_i)\right).
\]
On the other hand we have an  estimate
\begin{equation}\label{eqestODE}
d_g\left(\gamma_i(t_i), \gamma_{i-1}(t_i)\right)\leq (1+\epsilon)\frac{\epsilon_i}{K}\left(e^{K\Delta t}-1\right).
\end{equation}
To see this note that both the path $\gamma$ and the path $\gamma_{i-1}$ map the interval $[t_{i-1},t_i]$ into the coordinate chart $U_{\gamma(t_{i-1})}$. Let 
\begin{equation}
k(t)=d_0(\gamma(t),\gamma_{i-1}(t)),\quad t\in [t_{i-1},t_i]
\end{equation}
be the Euclidean distance. Then $k(t)$ satisfies the differential inequality
\[
\frac{dk}{dt}\leq |\gamma'(t)-V_{\gamma(t)}|+|V_{\gamma(t)}-V_{\gamma_{i-1}(t)}|\leq g(t)+Kk(t),
\]
with initial condition $k(t_{i-1})=0$. 
By Gr\"onwall's inequality we get 
\begin{align}
d_g(\gamma(t),\gamma_{i-1}(t)))&\leq (1+\epsilon)k(t)\notag\\
&\leq (1+\epsilon)e^{K(t-t_{i-1})}\int_{t_{i-1}}^t e^{-Ks}g(s)ds\notag\\
& \leq  (1+\epsilon) \frac{\epsilon_i}{K} \left(e^{K(t-t_{i-1})}-1\right),\quad t\in[t_{i-1},t_i]\notag
\end{align}
implying  \eqref{eqestODE}. 
The right hand side of \eqref{eqestODE}  is $\leq (1+\epsilon)^2\epsilon_i\Delta t$ since $\Delta t\ll \frac1K$.

%We have, $d(x_0,x_N)\leq\sum\Delta x_i$. %Let $y_i=\gamma(t_i)$ and $z_i=\phi_{\Delta t}(y_{i-1})$. 
We have  $\phi_1(\gamma(0))=x_0$ and $\gamma(1)=\gamma(0)=x_N$. Thus,
\[
d(x_0,x_N)\leq\sum\Delta x_i\leq\sum_{i=1}^N (1+\epsilon)^2\epsilon_ie^{\lambda(1- t_i)}\Delta t
\]
The last expression approximates the integral
\[
(1+\epsilon)^2\int_0^1 h'(t)e^{\lambda(1-t)}dt\leq \sqrt{\int_0^1(h'(t))^2dt}\sqrt{\frac{e^{2\lambda}-1}{2\lambda}}.
\]
Combining the last two inequalities gives the estimate
\[
d_g(p,\phi_1(p))^2=d_g(x_0,x_N)^2\leq \frac{(1+\epsilon)^2}{(1-\epsilon)^2} \frac{e^{2\lambda}-1}{2\lambda}\|\gamma'-V_{\gamma}\|^2_{L^2}.
\]
Since $\epsilon$ is arbitrary this proves the claim.

\end{proof}

We say that a Floer datum $(H,J)$ is of Lyapunov type if $(X_H,g_J)$ is of Lyapunov type.
\begin{cy}\label{cyMinDisImpDis}
Suppose $J$ is geometrically bounded, $(H,J)$ is of Lyapunov type,  and that outside of a compact set the quantity $d(p,\psi_1(p))$ is bounded away from $0$. % $d_{g_J}(\psi^H_1(p),p)>\epsilon>0$ 
Then $(H,J)$ is RLD.
\end{cy}
\begin{proof}
This is an immediate consequence of Lemmas \ref{lmLyap} and \ref{lmGeomFBound}. \end{proof}
\begin{ex}\label{exLinLiou}
Using the notation of Example~\ref{ExLiouvTame}, let $M$ have an end modeled on $\Sigma\times\R_+$ and let $H_0$ be a function which is linear at infinity with slope $a$ not in the period spectrum. Then $H$ is of Lyapunov type. Indeed the flow on any level set of $H$ is of Lyapunov type by Lemma \ref{lmLyapEstCovBo} and compactness. Since the flows on different level sets are conjugate, the existence of a Lyapunov estimate follows also for $x,y$ not on the same level set.  Let $X=X_{H_0}|_{\Sigma\times\{1\}}$. Let $J$ be a translation invariant almost complex structure. Then from \eqref{eqBoundL2L1} it follows that $f(p, X_{H_0})$ is bounded away from $0$, and so $H_0$ is LD.

Let $\delta$ be the distance of $c$ to the period spectrum of $\Sigma_{\alpha}$ and let $H_1$ be any Hamiltonian such that
\[
\frac{\|X_{H_1}-X_{H_0}\|}{\|X_{H_0}\|}\ll\delta/2.
\]
For example, this inequality will hold for our choice of $J$ whenever $\|X_{H_1-H_0}\|$ is bounded.  Then $f(p,X_{H_1})$ is bounded away from $0$ at infinity. So, $H_1$ is also LD.
\end{ex}

\begin{ex}
Let $M_1$ be as in the previous example and let $M_2$ be a compact symplectic manifold. Let $a$ be a real number not in the period spectrum of $M_1$ and let $f:\R/\Z\times M_1\times M_2$ be any function which tends to $1$ at infinity with derivatives dominated by $o(e^{-r/2})$. Then reasoning as in the previous example, the function $H:=afe^r$ is LD.
\end{ex}
\begin{lm}\label{lmLinDisCrit}
Let the end of $M$ be diffeomorphic to $\Sigma\times\R_+$ with $\Sigma$ a compact hypersurface. Suppose the projection $\pi:\Sigma\times\R_+\to\Sigma$ satisfies $\|\pi_*v\|\leq\|v\|$ for any tangent vector $v\in T(\Sigma\times\R_+)$. Let $X$ be vector field on $\Sigma$ with no $1$-periodic orbits and let $H$ be such that $\pi_*X_H$ converges uniformly to $X$. Then for some $\delta>0$ we have $f(p,X_H)>\delta>0 $ and in particular $X_H$ is LD.
\end{lm}
\begin{proof}
Let $\epsilon$ such that $f(p,X)>\epsilon$. For $r$ large enough, the convergence assumption implies
\[
f(p,\pi_*X_H)>\epsilon/2.
\]
The assumption of non-increasing implies $f(p,X_H)>f(p,\pi_*X_H)$.
\end{proof}
\begin{ex}
Let $M,\Sigma,\alpha, H_0$ and $H_1$ be as in Example~\ref{exLinLiou}. Let $\sigma$ be a  closed two form on $\Sigma$. Suppose $\sigma$ extends to a closed form on $M$ which is invariant under the Liouville flow near $\Sigma$. Then $\sigma$ can be extended in a translation invariant way to a closed two form on the completion of $M$, still denoted by $\sigma$. For $t$ small enough, the form $\omega_{t\sigma}=-d\alpha +t\sigma$ defines a symplectic form on the completion of $M$.  By rescaling $\sigma$ assume this holds for $t=1$. Then $H_0$ and $H_1$ are LD for the symplectic form $\omega_\sigma$. Indeed, write $X'_{H_0}$ for the Hamiltonian vector field with respect to $\omega_{\sigma}$. Let $X$ as in Example~\ref{exLinLiou}. Then all the requirements of Lemma~\ref{lmLinDisCrit} are satisfied for the pair $X'_{H_0},X$. The claim for $H_1$ now follows by comparison to $H_0$.
\end{ex}

\begin{ex}
In Example~\ref{exLinLiou} assume $\Sigma,\alpha$ is not necessarily contact but stable Hamiltonian for the restriction $\omega_1:=\omega|_{\Sigma\times 1}$ with stabilizing form $\alpha$. Namely, $\alpha$ satisfies $\ker\omega\subset \ker d\alpha$ and $\alpha\wedge\omega^{n-1}>0$. Assume $\omega$ is of the form $\omega_{\alpha}:=\omega+d(e^r\alpha)$ on $\Sigma\times\R_{\geq0}$ and is symplectic for all $r\geq 1$. Assume further that there exists a translation invariant $\omega_{\alpha}$-compatible almost complex structure $J$ on $\Sigma\times\R_{\geq 0}$. Then the forms $\omega(\cdot,J)$ and $d\alpha(\cdot,J)$ are separately non-negative. So, the projection $\pi_*$ is norm non-increasing. So if $H_0$ is linear at infinity with slope not in the period spectrum then $f(p,X_{H_0})$ is bounded away from $0$ and $H_0$ is dissipative. The same will hold under a sufficiently small deformation of $\omega$ or a sufficiently small Hamiltonian perturbation of $X_{H_0}$.
\end{ex}

In all the examples of this section we have considered Hamiltonians which are roughly linear at infinity. It is easy to use these examples to construct super-linear Hamiltonians which are LD. It is an interesting question what Hamiltonians can be perturbed to satisfy LD. The property of being LD is clearly related the behavior of the function $f(p,X_H,g_J)$. Namely, if one can find an exhaustion for which this function is appropriately bounded away from $0$ near the boundaries, the Floer datum will be LD.
\section{Proof of Theorem~\ref{mainTmA}}\label{Sec6}
\subsection{Floer systems}
For a symplectic manifold $(M,\omega)$, denote by $\mathcal{J}(M,\omega)$ the set of $\omega$-compatible almost complex structures on $M$. Let
\[
    \mathcal{F} \subset C^{\infty}(\R/\Z\times M)\times C^{\infty}(\R/\Z,\mathcal{J}(M,\omega))
\]
denote the set of Floer data $(H,J)$ such that $H$ is proper and bounded from below, the  Hamiltonian flow of $H$ is defined for all time, and the metric $g_{J_t}:=\omega(\cdot,J_t\cdot)$ is complete for any $t\in \R/\Z$. Denoting by $\Delta^i$ the standard simplex let
\[
\mathcal{F}^{i}\subset C^{\infty}(\Delta^i,\mathcal{F}),
\]
be the subset consisting of elements which are constant in a neighborhood of the vertices.  Furthermore, we require that for any $F\in\mathcal{F}^{(1)}$, $\partial_sF\geq 0$. Denote by $\Delta^{i,o}$ the interior of the simplex. Fix once and for all diffeomorphisms
\[
\sigma:\R\to \Delta^{1,o},
\]
\[
\psi:\R\times(0,1)\to\Delta^{2,o} ,
\]
and an increasing diffeomorphism
\[
\rho:(0,1)\to(0,\infty),
\]
for which
\begin{align}
\lim_{t\to 1}\psi(s+\rho(t),t) &=f^0\circ\sigma(s),\notag\\
\lim_{t\to 1}\psi(s-\rho(t),t)&=f^2\circ\sigma(s),\notag\\
\lim_{t\to 0}\psi(s\pm\rho(t),t)&=f^1\circ\sigma(s),\notag
\end{align}
uniformly on compact subsets of $\R$. Here $f^i:\Delta^1\to\partial\Delta^2$ is the standard embedding of the face missing the $i$th vertex. We extend the maps $\psi^{\pm}:=\psi(\cdot\pm\rho(\cdot),\cdot)$ to the closure $\R\times[-1,1]$ in the obvious way. See Figure \ref{figStBr}.
\begin{figure}
\includegraphics[scale=0.35, trim=70 245 0 340,clip]{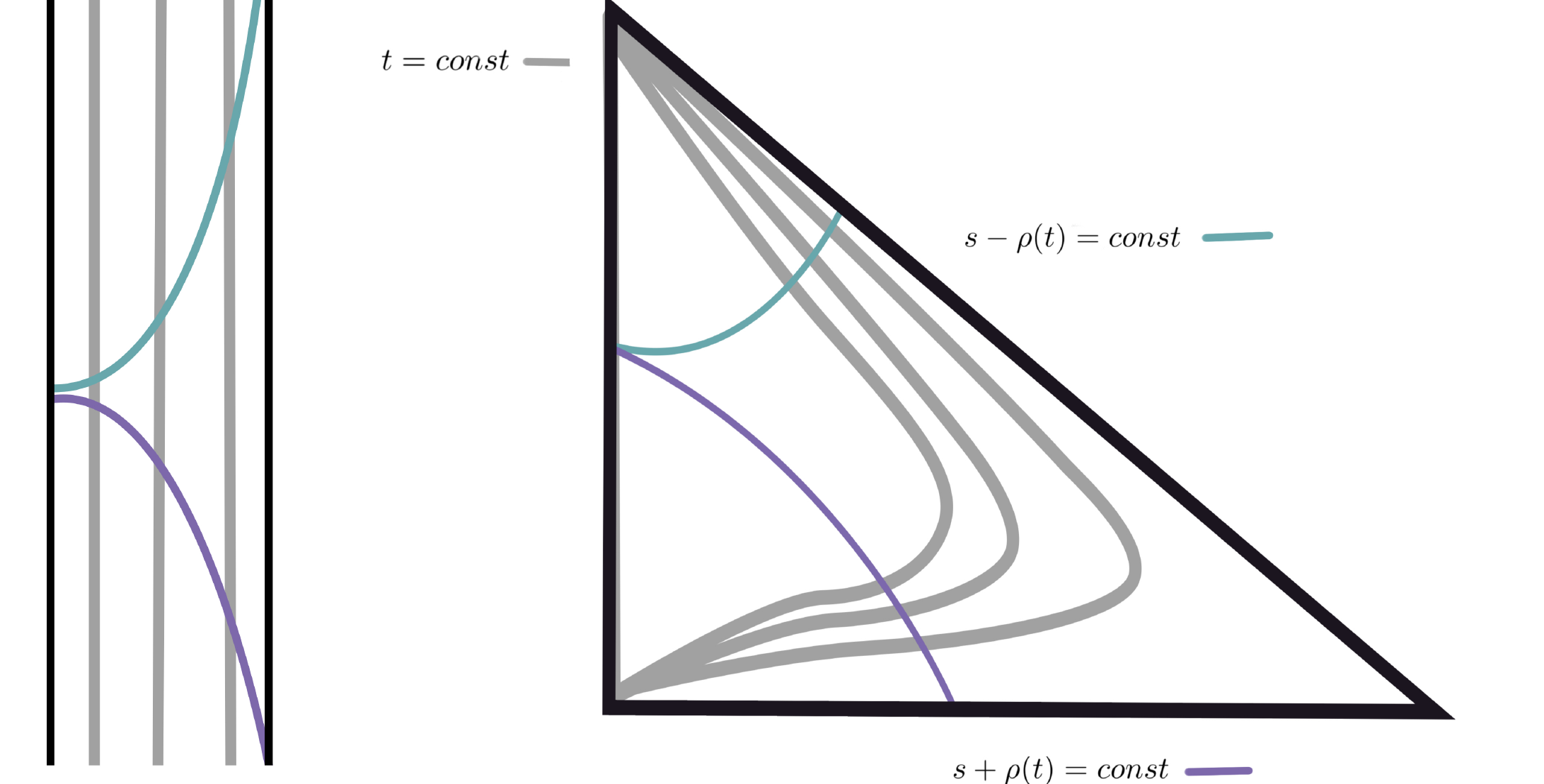}
\caption{The source and target of the map $\psi$}\label{figStBr}
\end{figure}

\begin{df}\label{dfWellBehaved}
A Floer datum $(H,J)\in\mathcal{F}^{(0)}$ is called \textbf{well behaved} if for any $E>0$ and any compact $K\subset M$ there is an $R=R(E,K)>0$ such that any solution $u$ to Floer's equation
\[
\partial_su+J(\partial_tu-X_{H})=0
\]
satisfying
\[
E(u):= \frac1{2}\int\|\partial_su\|^2\leq E,\qquad u(\R\times S^1)\cap K\neq\emptyset,
\]
is contained in the ball $B_R(K).$

A homotopy $F=(H_s,J_s)\in\mathcal{F}^{(1)}$ is called well behaved if the corresponding condition holds for the solutions to
\[
\partial_su+J_{\sigma(s)}(t)\left(\partial_tu-X_{H_{\sigma(s)}}(t,u(s,t))\right)=0.
\]
Finally, an element $\{F_p\}_{p\in\Delta^2}\in\mathcal{F}^{(2)}$ is called well behaved if the corresponding condition holds for the set of solutions  to
\[
\partial_su+J_{\psi^{\pm}(s,\tau)}(t)\left(\partial_tu-X_{H_{\psi^{\pm}(s,\tau)}}(t,u(s,t))\right)=0,\quad\tau\in[0,1],
\]
with $R(E,K)$ independent of $\tau$.  Denote by $\mathcal{F}^{(i)}_{w.b.}\subset\mathcal{F}$ the subset consisting of well behaved elements.

 \end{df}
 \begin{df}\label{dfFloerSys}
 A \textbf{Floer system} $\mathcal{D}$ on $M$ consists of the data of subsets $\mathcal{D}^{(i)}\subset\mathcal{F}^{(i)}_{w.b.}$ for $i=0,1,2$ such that the following hold:
 \begin{enumerate}
 \item\label{dfFloerSys:it0} For any element $F\in\mathcal{D}^{(i)}$ there is an open neighborhood $F\in V\subset C^1\times C^0$ such that $V\subset \mathcal{D}^{(i)}$.
 \item\label{dfFloerSys:it1} A face of an element of $\mathcal{D}^{(i)}$ is an element of $\mathcal{D}^{(i-1)}.$
 \item\label{dfFloerSys:it2} For any pair $F_i=(H_i,J_i)\in\mathcal{D}^{(0)}$, $i=0,1$ such that $H_1\geq H_0$, there is a homotopy $\{F_s\}_{s\in[0,1]}\in\mathcal{D}^{(1)}$ with endpoints $F_0$ and $F_1$.
 \item\label{dfFloerSys:it3} Given a pair $F',F''\in\mathcal{D}^{(1)}$ such that $F'_1=F''_0$ there is a $G\in\mathcal{D}^{(2)}$ whose restriction to the $\{0,1\}$ and $\{1,2\}$ faces coincides with $F'$ and $F''$ respectively.
 \item\label{dfFloerSys:it4} Given homotopies $F_{01},F_{12},F_{02}\in\mathcal{D}^{(1)}$, such that the endpoints of $F_{ij}$ are $F_i$ and $F_j$ respectively, there is a $G\in\mathcal{D}^2$ whose face $ij$  coincides with  $F_{ij}$.
 \end{enumerate}
  A Floer  system $\mathcal{D}$ is said to be \textbf{invariant} if it is invariant under the action of the symplectomorphism group given by
  \[
    \phi\cdot(H,J)=(H\circ \phi,\phi^*J).
  \]
  Elements of $\mathcal{D}^0$ will be referred to as $\mathcal{D}$-admissible. A function $H\in C^{\infty}(M)$ is said to be $\mathcal{D}$-admissible if there is an almost complex structure $J$ such that $(H,J)\in\mathcal{D}$.
  A \textbf{bi-directed} Floer system is one in which for any admissible $H_1$ and $H_2$  there are admissible $H_3$ and $H_0$ such that
        \[
        H_3\geq\max\{H_1,H_2\},
        \]
  and
  \[
  H_0\leq\min\{H_1,H_2\}.
  \]
\end{df}
%A Floer system can be seen as special kind of $(\infty,1)$-category, where all triangles are fillable and there is a morphism between monotone ordered objects by including higher simplices in the obvious way. Theorem~\ref{tmFloerFunc} below, which is just Floer's theorem\cite{Floer1989} with minor adaptations, gives a functor from such a category to chain complexes.
In Theorem~\ref{TmMain} below we show that on any geometrically bounded manifold there is a canonically defined invariant bidirected Floer system.

\begin{df}\label{dfAdmSiAb}
Define the \textit{dissipative Floer system} $\cF_d$ to consist of the following data. Let $\mathcal{F}_d^{(0)}(M)$ be the set of i-bounded Floer data $(H,J)$ which are RLD. Let $\mathcal{F}_d^{(1)}$ be the set of monotone paths $(H_s,J_s)_{s\in[0,1]}$ in $\mathcal{F}^{(0)}$ with endpoints in $\mathcal{F}_d^{(0)}$ such that the domain dependent Floer datum $(s,t)\mapsto (H(t,\cdot)_{\sigma(s)}dt,J(t,\cdot)_{\sigma(s)})$ is i-bounded as in Definition~\ref{dfFloDataWbounded}. Finally, $\mathcal{F}_d^{(2)}$ is defined as follows. Let $F_{p\in \Delta^2}=(H_p,J_p)_{p\in \Delta^2}\in\mathcal{F}^{(2)}$ with edges in $\mathcal{D}^{(1)}$. Associate to $\Delta^2$ and the map $\psi:\R\times (0,1)\to\Delta^2$ a family $C_{ \tau\in[0,1]}$ of cylinders over the unit interval degenerating to a broken cylinder in the obvious way. Let the domain dependent Floer datum on $C_\tau$ be defined by $F_{\tau}(s,t)=(H_{\psi(s,\tau)}(t,\cdot)dt,J_{\psi(s,\tau)}(t,\cdot))$. Then $F\in\mathcal{F}_d^{(2)}$ iff the family $C$ with this choice of domain dependent Floer data is uniformly i-bounded as in Definition~\ref{dfFloDataWboundedFamily}.
\end{df}
\begin{rem}
The set of all well behaved Floer data is not necessarily connected. Thus it is possible that there exist other Floer systems perhaps giving rise to inequivalent theories.  However, any Floer system  for which the well behavedness property of Definition \ref{dfWellBehaved} holds in a sufficiently domain local manner is equivalent to the dissipative system by an argument similar to the proof of Theorem \ref{tmWeakTameCont}.
\end{rem}
\begin{tm}\label{TmMain}
Let $(M,\omega)$ be a monotone or Calabi-Yau geometrically bounded symplectic manifold. Then $\mathcal{F}_d(M)$ is an invariant bi-directed Floer system on $M$.
\end{tm}

Before proving Theorem \ref{TmMain} we need the following lemma.
\begin{lm}\label{lmAdmConn}
Let $(H_i,J_i)\in\mathcal{F}_d^{(0)}(M)$ such that $H_0\leq H_1$. There exists an i-bounded monotone Floer datum on $\R\times S^1$ which coincides with $(H_0,J_0)$ on $\{s\ll0\}$ and with $(H_1,J_1)$ on $\{s\gg0\}$. Moreover, the set of such Floer data is contractible in the same sense as in Theorem~\ref{tmWeakTameCont}.
\end{lm}
\begin{proof}
To conform with definition \ref{dfFloDataWbounded}, it suffices to produce  an almost complex structure on $[0,1]\times \R/\Z\times M $ of the form $J_{H_s}$ for some $(H_s,J_s)$ such that the following are satisfied
\begin{itemize}
\item $\partial_s J_{H_s}$ vanishes identically near the boundary of $[0,1]\times \R/\Z\times M $, and thus extends to an almost complex structure on $\R\times \R/\Z\times M $ interpolating between $J_{H_0}$ and $J_{H_1}$. We continue to denote this extended almost complex structure by $J_{H_s}$. 
\item Denoting by $\pi: \R\times \R/\Z\times M $ the projection to $\R\times \R/\Z$, we have that the restriction of $J_{H_s}$ to each of $\pi^{-1}((1/3,\infty)\times \R/\Z)$ and to $\pi^{-1}((-\infty,2/3)\times \R/\Z)$ is intermittently bounded relative to $\pi$. 
\item $\partial_sH_s\geq 0.$
\end{itemize}
Other than the last condition, the construction would be the same as in the proof of Theorem~\ref{tmWeakTameCont}. We show that the monotonicity requirement  does not affect the proof of Theorem~\ref{tmWeakTameCont}. For simplicity, assume $H_i$ is time independent. As in the proof of Theorem~\ref{tmWeakTameCont} fix two disjoint open sets $V_1,V_2\subset M$ such that there is taming data for $J_{H_i}$ supported in $[0,1]\times \R/\Z\times V_i$ for $i=0,1$. We may assume that each of the $V_i$ is a disjoint union of pre-compact sets. Let $\chi: M\to[0,1]$ be a function which equals $0$ on $V_0$ and $1$ on $V_1$. Let $f:[0,1]\to[0,1]$ be a monotone function which is identically $0$ near $0$ and identically $1$ on $[1/3,1]$. Let $g:M\times [0,1]\to[0,1]$ be defined by
\[
g(x,s)= f(1-s)f(s)\chi(x)+1-f(1-s).
\]
Then $g$ is monotone increasing in $s$, identically $0$ for all $x$ when $s$ is near $0$ and identically $1$ for all $x$ when $s$ is near $1$. Take $H_s=g(x,s)H_1+(1-g(x,s))H_0$. Then $H_s$ is also monotone increasing in $s$. Moreover, $H$ is fixed and equal to $H_0$ on $[0,2/3]\times V_0$ and to $H_1$ on $[1/3,1]\times V_1$. Let $J_s$ be any homotopy which is fixed and equal to $J_0$ on $[0,2/3]$  and to $J_1$ on $[1/3,1]$. Then $J_{H_s}$ is i-bounded since it coincides with $J_{H_0}$ on $[0,2/3]\times V_0$ and with $J_{H_1}$ on  $[1/3,1]\times V_1$. Contractibility of the set of all such homotopies is similar.
\end{proof}

\begin{proof}[Proof of Theorem \ref{TmMain}]
Let $\mathcal{F}_d(M)$ be as in Definition~\ref{dfAdmSiAb}. We verify that $\mathcal{F}_d(M)$ has all the required properties. Namely, that it is a Floer system and that it satisfies the properties guaranteed in Theorem \ref{TmMain}.
\paragraph{\textbf{Well behavedness}}
This follows from the definition and Theorem~\ref{tmFloSolDiamEst}.
\begin{comment}
Tameness for elements of $\mathcal{D}^{(0)}$ and $\mathcal{D}^{(1)}$ is a consequence of Theorem~\ref{tmFloSolDiamEst}.
We prove tameness for elements of $\mathcal{D}^{(2)}$. Let $F\in\mathcal{D}^{(2)}$. First note that for each value of $\tau$, the family $F_{\psi(\cdot,\tau)}$ is u.a.b by Theorem~\ref{tmPullbackTame}. We need to show uniformity of the diameter estimates as the we vary $\tau$. Away from $\tau=1$ the uniformity is clear. We establish a uniform estimate in a neighborhood of $\tau=1$. In such a neighborhood we may assume without loss of generality that $\partial_sF_{\psi(\cdot,\tau)}$ is supported in
\[
[\rho(\tau),\rho(\tau)+1]\cup[-\rho(\tau)-1,-\rho(\tau)].
\]
Abbreviate $R=\rho(\tau)$. We now establish a diameter estimate which is independent of $R$. Let $u$ be a Floer solution for $F_{\psi(\cdot,\tau)}$. As in Lemma~\ref{lmDisDiamEst} we obtain an a-priori estimate on $u|_{(-\infty,-R-\epsilon]\times S^1}$ which is independent of $R$. Given this estimate we obtain an a-priori estimate on $u|_{[-R-1-\epsilon,-R+\epsilon]}$ as in the proof of Theorem~\ref{TmAlter}. Iterating this mode of argument for the next segments we obtain an a priori estimate as required.
\end{comment}
\paragraph{\textbf{Condition \ref{dfFloerSys:it0}}}We need to show that if $(H,J)$ is dissipative, so is a nearby $(H',J')$. The most involved case is when $i=2$ which we treat. Near each vertex, we have fixed Floer data so by definition we can pick an open neighborhood which maintains RLD-ness for all three of these.  The property of being u.i.b. depends on the metric only up to quasi-isometry which is preserved for any uniform open neighborhood.
\paragraph{\textbf{Condition \ref{dfFloerSys:it1}}} This follows by definition.
\paragraph{\textbf{Condition \ref{dfFloerSys:it2}}} This is just Lemma~\ref{lmAdmConn}.
\paragraph{\textbf{Condition \ref{dfFloerSys:it3}}}\label{tmMain:it3} Pick an $R_0>0$ for which $\sigma^{-1}(\supp\partial_sF')\subset [-R_0,R_0]$, and similarly for $F''$. For any $R>0$ define the homotopy  $I_R=F'\#_R F''$ by
    \[
    I_{R,s}:=\begin{cases} F'_{\sigma(0)},\qquad\qquad &s\leq -R-R_0,\\
                                    F'_{\sigma(s+R)},\qquad\qquad &s\in[-R-R_0,-R],\\
                                    F'_{\sigma(1)}=F''_{\sigma(0)},\qquad\qquad &s\in[-R,R],\\
                                    F''_{\sigma(s-R)},\qquad\qquad &s\in[R,R+R_0],\\
                                    F''_{\sigma(1)},\qquad\qquad &s\geq R+R_0.\end{cases}
    \]
    Define $G$ by $G_{\psi(s,\tau)}:=I_{\rho(\tau),s}$. It is immediate that $G\in\mathcal{F}_d^{(2)}.$
\paragraph{\textbf{Condition \ref{dfFloerSys:it4}}}First merge $F_{01}$ with $F_{12}$ as in the previous part. Then homotope to $F_{02}$ relying on contractibility in Lemma~\ref{lmAdmConn}.
\paragraph{\textbf{Invariance}} Evident from the definition.
\paragraph{\textbf{Bi-directedness}}This follows by Theorem \ref{tmDiaFinCofin}.%Lemma~\ref{lmWLDLDOS} and Theorems~\ref{tmDregCofinal} and~\ref{lmExDisDom}.
%\paragraph{\textbf{Cofinality}} This can be achieved by corollary~\ref{cwWeakDisCofinal}.

\end{proof}
\subsection{Transversality and control of bubbling}\label{subsecBubbCont}
\begin{df}\label{Flreg}
Denote by $\mathcal{J}_{reg}$ the set of almost complex structures for which all moduli spaces
\[
\mathcal{M}^*(A;J),
\]
of non-multiply covered $J$-holomorphic spheres representing any class $A\in H^2(M;\Z)$ are smooth manifolds of expected dimension. For $J\in\mathcal{J}_{reg},$ let $\mathcal{H}_{reg}(J)$ denote that set of all non-degenerate Hamiltonians satisfying the following conditions
\begin{enumerate}
    \item The linearization $D_u$ of Floer's equation at a Floer trajectory $u$ is surjective for all $(H,J)$-Floer trajectories.
    \item
         No Floer trajectory with index difference $\leq 2$ intersects a $J$-holomorphic sphere of Chern number $0.$
    \item No periodic orbit of $H$ intersects a $J$-holomorphic sphere of Chern number $\leq1$.
\end{enumerate}
\end{df}
Write
\[
\mathcal{F}^{(0)}_{reg}:=\cup_{J\in\mathcal{J}_{reg}}\mathcal{H}_{reg}(J)\times\{J\}.
\]
Recall, $M$ is said to be \textit{semi-positive} if for any class $A\in\pi_2(M)$ we have
 \[
    3-n\leq c_1(A)<0\qquad\Rightarrow \qquad \omega(A)\leq 0.
 \]
 Observe that if $M$ is monotone or Calabi-Yau, then it is semi-positive.
\begin{tm}\label{NondegDiss}
Suppose $M$ is semi-positive. Let $(H,J)\in\mathcal{D}^{(0)}$ and let $V\subset\mathcal{F}^{(0)}_{w.b.}$ be an open neighborhood of $(H,J)$ in $C^{\infty}\cap\mathcal{F}_{w.b.}$.  Shrinking $V$, write $V=V_1\times V_2\subset\mathcal{H}\times\mathcal{J}$. The set $\mathcal{J}_{reg}$ is of second category in $V_2$ and for each $J\in\mathcal{J}_{reg}\cap V_2$, the set $\mathcal{H}_{reg}(J)$ is of second category in $V_1$.

\end{tm}
\begin{rem}
Theorem \ref{NondegDiss} is formulated for time independent almost complex structures following \cite{HoferSalamon}. The same claim holds for time dependent Hamiltonians after appropriately modifying the regularity requirement. However, if we wish to construct homotopies of such we need to restrict to the case where $M$ is monotone or Calabi-Yau. See remark \ref{rmSPVSSM} below. One reason why one would wish to work with time dependent $J$ is that once $H$ has non degenerate periodic orbits, for generic Floer data of the form $(H,J)$ with $J$ time dependent the moduli space of smooth Floer trajectories is a smooth manifold of the expected dimension. See  Theorem 5.1 in \cite{FloerHoferSalamon}. It then follows easily that generic such $(H,J)$ are regular.
\end{rem}
\begin{proof}
Since all the moduli spaces for all the Floer data in $V$ intersecting a compact set $K$ and possessing energy $E$ are contained, for some $R<\infty$, in $B_R(K)$, this follows from the compact case. For the compact case see, e.g., \cite{HoferSalamon}.
\end{proof}

Suppose for $i=0,1,$ we have well behaved elements $F_i\in\mathcal{F}^{(0)}_{reg}$ and let
\[
    F_{01}:=\{F_s=(H_s,J_s)\}_{s\in\Delta^1},
\]
be a well behaved homotopy between them.
\begin{df}\label{horeg}
Call such a homotopy regular if the following hold.
\begin{enumerate}
    \item For any $A\in H^2(M;\Z)$ write
    \[
    \mathcal{M}^*(A;\{J_{s}\}):=\{(s,u)|u\in\mathcal{M}^*(A;J_{s,t})\}.
    \]
   Then $\mathcal{M}^*(A;\{J_{s}\})$ is smooth and of the expected dimension.
   \item For any $\tilde{\gamma}_1$ and $\tilde{\gamma}_2$ the moduli spaces
    \[
    \mathcal{M}(\tilde{\gamma}_1,\tilde{\gamma}_2,F=\{H_{s},J_{s}\})
    \]
of nontrivial continuation trajectories are smooth and of the expected dimension.
    \item There is no continuation trajectory $u$ of index $0$ or $1$ for which there is a point $(s,t)$ such that $u(s,t)$ is in the image of a $J_{s}$-holomorphic sphere of Chern number $0$.
\end{enumerate}
Similarly, let $F\in\mathcal{F}^{(2)}$ with edges corresponding to regular homotopies. For such an $F$, $\Delta^{(2)}$ parameterizes a family of time dependent Floer data $(H,J)$. Write $(H_{s,\lambda},J_{s,\lambda}):=F_{\psi(s,\lambda)}$.

We say that $F$ is regular if
\begin{enumerate}
\item The moduli space  $\mathcal{M}^*(A;\{J_{s,\lambda}\})$ is smooth of the expected dimension.
\item The corresponding family $\{u_\lambda\}$ of Floer solutions is smooth of the expected dimension.
\item There is no $\lambda$ for which there is a point $(s,t)$ and a continuation trajectory $u_\lambda$ of index $-1$ or $0$ such that $u_{\lambda}(s,t)$ intersects a $J_{s,\lambda}$-holomorphic sphere of Chern number $0$.
\end{enumerate}
Denote by $\mathcal{F}^{(i)}_{reg}$, $i=1,2,$ respectively, the regular 1- and 2-simplices of Floer data.
\end{df}
For $i=1,2$ let $F^i\in\mathcal{F}^{(i)}_{w.b.}$ and let $V(F^i)\subset\mathcal{F}^{(i)}_{w.b.}$ be an open neighborhood. Let $V_2(F^i)\subset V(F^i)$ be the set of elements whose $H$ component coincides with that of $F^i$.

To achieve transversality in the definition of continuation maps, we wish to avoid perturbing $H_{01}$ since it is required to satisfy a monotonicity condition. Thus we will perturb $J_{01}$ in an $s$ dependent manner.

\begin{tm}\label{tmHoreg}
 Suppose $M$ is monotone or Calabi-Yau. Then $\mathcal{F}^{(i)}_{reg}\cap V_2(F^i)$ is of second category in $V_2(F^i)$ for $i=1,2$.
\end{tm}
\begin{rem}\label{rmSPVSSM}
The strengthening of the assumption relative to Theorem \ref{NondegDiss} is required in the case $i=2$. Indeed, in this case, the assumption of semi-positivity does not rule out the possibility that for an isolated $(s,\lambda)$ there is a $J_{s,\lambda}$-holomorphic sphere with negative Chern number. Once such a sphere is present, its multiple covers interact with Floer trajectories in a non transverse way. Invariance of Floer cohomology under homotopies of $J$ can still be established for the semi-positive case by construction chain homotopies for truncated Floer homologies since regularity for that case is easily seen to be an open condition. We do not pursue this here.
\end{rem}
\begin{proof}
We need to verify that we can achieve regularity even though we avoid perturbing $H$. For the generic smoothness of the moduli spaces see \S16 in \cite{Ritter13}. The non-intersection property is a variation of the corresponding claim in \cite{HoferSalamon}. Namely, for $i=1$  to show the is to show that the universal moduli spaces
\[
\mathcal{N}:=\{(s,z ,F=(H_{s},J_{s}),u_1,u_2|F\in V_2, u_1(z)=u_2(s,t)\}
\]
is a smooth separable Banach space of the expected codimension. Here $u_1\in\mathcal{M}^*(J_{s,t})$ and $u_2$ is an $F$ continuation trajectory. For this it suffices that the evaluation map
\[
\R\times S^1\times S^2 \times\{u\in\mathcal{M}^*(J_{s})|J_{s}\in V_2\}\to \R\times S^1\times M
\]
defined as
\[
(s,t,z,u)\mapsto (s,t,u(z)),
\]
is a submersion. For this, apply Lemma 3.4.3. from \cite{MS2}. For $i=2$ the argument is similar.
\end{proof}

\subsection{Floer's Theorem}\label{SecDefHam}
 \textit{In this section we assume throughout that $M$ is monotone or Calabi-Yau. Moreover, we assume $M$ is connected.}

Denote by $\mathcal{L}M$ the free loop space $C^{\infty}(\R/\Z,M)$.  Let $I_{\omega},I_c:\pi_1(\mathcal{L}M)\to\R$ be given by integrating $\omega$ and the Chern class respectively. Denote by   $\widetilde{\mathcal{L}M}$ the \textbf{Floer-Novikov covering} of $\mathcal{L}M$. That is, the Abelian covering space of $\mathcal{L}M$ for which $i_*\pi_1(\widetilde{\mathcal{L}M})=\ker I_\omega\cap \ker I_c$ where $i_*:\pi_1(\widetilde{\mathcal{L}M})\hookrightarrow\pi_1(\mathcal{L}M)$ is the natural inclusion. Explicitly, the space $\widetilde{\mathcal{L}M}$ is constructed as follows. For each component $\mathcal{L}M_a$ of $\mathcal{L}M$ choose a base loop $\gamma_a$. Then $\widetilde{\mathcal{L}M_a}$  consists of equivalence classes of pairs $(\gamma,A)$ such that $\gamma\in\mathcal{L}M_a$, $A$ is a homotopy class of paths in $\mathcal{L}M_a$ starting at $\gamma_a$ and ending at $\gamma$, and the equivalence relation is $(\gamma,A_1)\sim(\gamma,A_2)$ if and only if $\omega(A_1)=\omega(A_2)$ and $c_1(A_1)=c_1(A_2).$

For a smooth function $H\in C^{\infty}(\R/\Z\times M)$ and for any $t\in \R/\Z$ denote by $X_{H_t}$ its Hamiltonian vector field. This is the unique vector field satisfying $dH_t(\cdot)=\omega(X_{H_t},\cdot).$
Define a functional $\mathcal{A}_H:\widetilde{\mathcal{L}M}\to \R$ by
\[
\mathcal{A}_{H}([\gamma,A]):=-\omega(A)-\int_0^{2\pi}H(\gamma(t))dt.
\]
Note that this functional depends on the choice of base loops $\gamma_a$ for the connected component $a\in\pi_0(\mathcal{L}M).$

Denote by $\mathcal{P}(H)\subset\mathcal{L}M$ the set of $1$-periodic orbits of $X_H$. Denoting by
\[
\pi:\widetilde{\mathcal{L}M}\to\mathcal{L}M,
\]
the covering map, let
\[
\widetilde{\mathcal{P}(H)}=\pi^{-1}(\mathcal{P}(H)).
\]
This is the same as the critical point set of $\mathcal{A}_H$.

We define an index
\[
i_{RS}:\widetilde{\mathcal{P}(H)}\to\mathbb{Z},
\]
as follows. For each homotopy class $a\in \pi_0(\mathcal{L}M)$ fix  a trivialization of $\gamma_a^*TM$. Then if $\tilde{\gamma}=(\gamma,A), $ trivialize $\gamma^*TM$ along $A$ by extending the existing trivialization from $\gamma_a$. With respect to this trivialization, the linearization $t\mapsto D\psi_{t,\gamma(t)}$ of the flow along $\gamma$ is a path of symplectic matrices to which is associated its Robbin-Salamon index \cite{RobbinSalamon93}. We take $i_{RS}(\tilde{\gamma})$ to be the Robbin-Salamon index in this trivialization.  Note that $i_{RS}$ is independent of choices up to an integer shift $n_a$ for each $a\in \pi_0(\mathcal{L}M).$

%Recall the discussion of the Floer-Novikov covering $\widetilde{\mathcal{L}M}$ of $\mathcal{L}M$ from \S\ref{SecFloerGrom}.

For each homotopy class $a\in\pi_1(M)$  let $\Gamma_a\subset \R\times 2\Z$ be the image of $\pi_1(\mathcal{L}M_a)$ under $I_\omega\times I_c$.  We identify elements of $\Gamma_a$ with equivalence classes in $\pi_1(\mathcal{L}M_a)$  modulo $\ker I_\omega\cap\ker I_c$. For any ring $R$, define the Novikov ring $\Lambda_{R,\Gamma_a}$ by the set of formal sums
\[
\sum_{A\in\Gamma_a}\lambda_AT^{I_\omega(A)}e^{2I_c(A)},\quad \lambda_A\in R\]
which satisfy for each constant $c$ that
\[
\#\{A\in\Gamma_a|\lambda_A\neq 0,\quad\omega(A)<c\}<\infty.
\]
We have an action of $\Gamma_a$ on $\widetilde{\mathcal{L}M}_a$ by
\[
A\cdot[x,B]:=[x,A\#B].
\]
This is a covering action, so it restricts to an action on $\widetilde{\mathcal{P}H}$.

Fix a Floer system $\mathcal{D}$. Write $\mathcal{D}_{reg}:=\mathcal{F}_{reg}\cap\mathcal{D}$ and let $F=(H,J)\in\mathcal{D}^{(0)}_{reg}$. We define the Floer chain complex $CF^*(H,J;R)$ as the set of formal sums
\[
\sum_{\tilde{x}\in\widetilde{\mathcal{P}(H)}}\lambda_{\tilde{x}}\langle\tilde{x}\rangle, \quad \lambda_{\tilde{x}}\in R
\]
satisfying for each constant $c$ that
\begin{equation}\label{eqFinCond}
\#\{\tilde{x}\in\widetilde{\mathcal{P}(H)}|\lambda_{\tilde{x}}\neq 0,\quad\mathcal{A}_H(\tilde{x})>c\}<\infty.
\end{equation}
$CF^*(H,J;R)$ is a graded vector space over $R$ with grading given by
\begin{equation}\label{eqCZGrading}
i(\tilde{x}):=i_{RS}(\tilde{x})+n.
\end{equation}
 Here $n=\frac12\dim M$. $CF^*(H,J;R)$ can be considered as a non-Archimedean Banach space over $R$ with its trivial valuation. The norm on  $CF^*(H,J;R)$ for a linear combination of generators is given by
\begin{equation}\label{NonArchNormDef}
\left\|\sum_ia_i\tilde{\gamma}_i\right\|:=\max_{\{i|a_i\neq 0\}}e^{\mathcal{A}_H(\tilde{\gamma}_i)}.
\end{equation}
For each homotopy class $a$, the vector space $CF^{*,a}(H,J;R)$ generated by $\widetilde{\mathcal{P}_a(H)}$ is a graded Banach module over the Novikov ring $\Lambda_{R,\Gamma_a}$  via the action of $\Gamma_a$ on $\widetilde{\mathcal{P}_a(H)}$. The set $\mathcal{P}_a(H)$ non-canonically defines a basis of $CF^{*,a}(H,J)$ over $\Lambda_{R,\Gamma_a}$ by choosing a lift to $\widetilde{\mathcal{P}_a(H)}$.

Let $\Gamma\subset\R\times\Z$ be a subgroup. Denote by $\Lambda_{R,\Gamma}$ the ring
\[
\Lambda_{R,\Gamma}:=\left\{\sum_i a_iT^{\lambda_i}e^{2n_i}|(\lambda_i,n_i)\in\Gamma, a_i\in R,\lim_{i\to\infty}\lambda_i=\infty\right\}.
\]
Let $\Gamma_{\omega}\subset\R\times\Z$ be the subgroup generated by $\cup_{a\in\pi_1(M)}\Gamma_a$. Write $\Lambda_{R,\omega}:=\Lambda_{R,\Gamma_{\omega}}$ and
$\Lambda_R:=\Lambda_{R,\R\times\Z}$. Assume $R$ is a field. $\Lambda_R$ is referred to as the \textbf{universal Novikov field over $R$}. Strictly speaking $\Lambda_{R,\R\times\Z}$ is only a graded field. That is, only homogeneous elements with respect to the grading induced by projection $\R\times\Z\to\Z$ are invertible. Henceforth let $\K$ be either $\Lambda_R$ or $\Lambda_{R,\omega}$. Note that $\K$ carries a non-Archimedean norm induced from $\|T^\lambda\|:=e^{-\lambda}$. That is, $\val(T^{\lambda})=-\lambda$ \footnote{Caution: in many texts the convention is $\|\cdot\|=e^{-\val(\cdot)}$.}.

Let
\[
CF^*(H,J;\K):=\oplus_{a\in\pi_1(M)} CF^{*,a}(H,J;\Lambda_{R,\Gamma_a})\hat{\otimes}_{\Lambda_{R,\Gamma_a}}\K,
\]
where the hat denotes completion with respect to the induced valuation. % For some purposes, however, we would like to work with $CF^*(H,J;R)$ without any Novikov module structure. Namely, the action functional is well defined over $CF^*(H,J;\Lambda_{R,\Gamma_a})$, but only over $R$ does it define a filtration by sub-modules.
 \begin{rem}\label{remAltApp}
The approach we follow here to Floer theory over the Novikov ring is the one originally introduced by  \cite{HoferSalamon}. In the literature (compare \cite{Ritter10,RitterSmith,Varolgunes2018}) there is a slightly different construction of the Floer chain complexes over the Novikov ring where one tensors the space generated by $\mathcal{P}(H)$ with $\Lambda_R$ instead of passing to a covering space. In that version, the chain complexes do not have an action filtration nor a grading, but they do have a Novikov filtration over $\Lambda_R$.  We do not pursue the latter approach here.
\end{rem}
We define a linear operator $d$ on $CF^*(H,J;R)$ by counting Floer trajectories in the usual way. Namely, for any two elements
\[
\tilde{x}_1,\tilde{x}_2\in\widetilde{\mathcal{P}(H)}
\]
of index difference $1$ denote by $\mathcal{M}(\tilde{x}_1,\tilde{x}_2;J)$ the moduli space of Floer trajectories which at $-\infty$ are asymptotic to $\tilde{x}_1$ and at $+\infty$ to $\tilde{x}_2$ divided by the action of $\R$. By the inclusion $\mathcal{D}^{0}\subset\cF_{w.b.}^{(0)}$ and Gromov-Floer compactness, $\mathcal{M}(\tilde{x}_1,\tilde{x}_2;J)$ is compact. Since $F\in\mathcal{F}^{(0)}_{reg}$ we get that when the virtual dimension is $0$,
\[
\#\mathcal{M}(\tilde{x}_1,\tilde{x}_2;J)<\infty.
\]
We thus define
\[
d\tilde{x}_1\qquad=\sum_{\tilde{x}_2|i_{RS}(\tilde{x}_2)=i_{RS}(\tilde{x}_1)+1}\#\mathcal{M}(\tilde{x}_1,\tilde{x}_2;J)\langle \tilde{x}_2\rangle.
\]

\begin{tm}\label{tmDifWD}
The Floer boundary map $d$ is well defined and satisfies $d^2=0$.
\end{tm}
\begin{proof}
We need to show that for any $\tilde{x}_1\in\widetilde{\mathcal{P}(H)}$ we have that $d\tilde{x}_2$ satisfies the finiteness condition~\eqref{eqFinCond}. For any $c,$ let
\[
A_c:=\left\{\tilde{x}_2\in\widetilde{\mathcal{P}(H)}|\mathcal{M}(\tilde{x}_1,\tilde{x}_2;J)\neq\emptyset,\quad\mathcal{A}_H(\tilde{x}_2)>c\right\}.
\]
For any $\tilde{x}\in A_c$ there is a Floer trajectory of energy at most $\mathcal{A}_H(\tilde{x}_1)-c$ connecting $\tilde{x}_1$ and $\tilde{x}_2$. Well behavedness of $F$ thus implies that there is a compact set $K\subset M$ such that any $\tilde{x}\in A_c$ is contained in $K$. The claim now follows by Gromov-Floer compactness. Thus $d$ is well defined. That $d^2=0$ follows from the compact case \cite{Salamon1999,MS2} since all Floer trajectories under a given energy level are contained in an a priori compact set.
\end{proof}

By its definition $d$ commutes with the action of $\Lambda_{R,\Gamma_a}$ and thus induces a well defined operator on $CF^*(H,J;\K)$.

\begin{tm}[\textbf{Floer's Theorem}]\label{tmFloerFunc}
Let $M$ be a monotone or Calabi-Yau symplectic manifold. Let $\mathcal{D}$ be a Floer system. Then there exists a dense subsystem $\mathcal{D}_{reg}$ such that
\begin{enumerate}
\item
for any $F=(H,J)\in \mathcal{D}^{(0)}_{reg}$ the graded filtered complex
\[
        (CF^*(H,J;\K),d),
\]
is well defined.
\item For any pair of elements $F_1\leq F_2\in  \mathcal{D}^{(0)}_{reg}$ we have that
        \[
            \mathcal{D}_{reg}^{(1)}(F_1,F_2)\neq\emptyset.
         \]
         Associated with any homotopy $F^{12}\in\mathcal{D}^{(1)}_{reg}(F_1,F_2)$  is a chain map
    \[
        f_{F^{12}}:CF^*(F^1)\to CF^*(F^2),
    \]
    defined by counting the corresponding rigid Floer solutions. If  $F_1-F_2=c$ and $F^{12}$ is of the form $F^{12}_s=(H+f(s),J)$, then $f_{F^{12}}$ is the identity.
    \item
    For any triple $F_0\leq F_1\leq F_2\in\mathcal{D}_{reg}^{(0)}$, and elements $F_{ij}\in\mathcal{D}_{reg}^{(1)}(F_i,F_j)$, the set $\mathcal{D}_{reg}^{(2)}(F_{01},F_{12},F_{02})$ is non-empty.
    Any element
    \[
    F\in\mathcal{D}_{reg}^{(2)}(F_{01},F_{12},F_{02})
    \]
    defines a chain homotopy between the map
    $f_{F_{02}}$ and the composition $f_{F_{12}}\circ f_{F_{01}}$
    by counting the corresponding rigid Floer solutions.

\end{enumerate}
\end{tm}

\begin{proof}
We take as above $\mathcal{D}_{reg}:=\mathcal{D}\cap\mathcal{F}_{reg}$. By the theorems in the previous subsection, $\mathcal{D}_{reg}$ is dense in $\mathcal{D}$.
\begin{enumerate}
\item This is just Theorem~\ref{tmDifWD}.
\item
Given $\tilde{x}_i\in\widetilde{\mathcal{P}(H_i)}$ for $i=1,2$, let $\mathcal{M}(\tilde{x}_1,\tilde{x}_2;\{H_s,J_s\})$ be the moduli space of Floer solutions for the Floer data $\{H_s,J_s\}$ with $\tilde{x}_1$ and $\tilde{x}_2$ as asymptotes. For any element
\[
u\in\mathcal{M}(\tilde{x}_1,\tilde{x}_2;\{H_s,J_s\})
\]
 we have the a priori estimate
\[
E(u)\leq \mathcal{A}_H(\tilde{x}_1)-\mathcal{A}_H(\tilde{x}_2),
\]
by Lemma~\ref{lmTopGeoEnEst} with $F=H_sdt$. By the assumption that $\cD$ consists of Floer data that are well behaved as in Definition \ref{dfWellBehaved} it follows that
\[
\mathcal{M}(\tilde{x}_1,\tilde{x}_2;\{H_s,J_s\})
\]
is compact. We define the continuation map by counting the $0$-dimensional moduli space. Since action decreases along continuation maps, the finiteness condition is the same as the case of the differential. The fact that $f_{I^{12}}$ is a chain map is the same as in the compact case \cite{Salamon1999,MS2}.
\item It follows again from well behavedness that all Floer solutions under a given energy level for all the elements of the family are contained in an a priori compact set.  The claim thus follows again from the compact case.

\end{enumerate}

\end{proof}

\begin{proof}[Proof of Theorem~\ref{mainTmA}]
This follows by Definition~\ref{dfAdmSiAb} from  Theorems~\ref{TmMain}, \ref{tmBoundGradDis} and~\ref{tmFloerFunc} by passing to homology.
\end{proof}

\section{Hamiltonian Floer cohomology by approximation}\label{SecHamFloer}
%We continue considering the chain complexes $CF^*(H,J;\K)$ as Banach spaces. Any action decreasing linear operator on $CF^*(H,J;\K)$ is continuous. The same holds for any operator which increases the action by a uniformly bounded amount. Thus the differential and the continuation maps are continuous.
\subsection{Reduced cohomology}
\begin{df}
We refer to a chain complex which is a topological vector space with continuous differential as a \textbf{topological complex}. For a topological complex $(C,d)$ the differential is in general not closed. To stay within the realm of complete Hausdorff topological vector spaces we define the \textbf{reduced cohomology} of  complete Hausdorff topological complexes$(C,d)$ by
\[
\overline{H}(C,d):=\frac{\ker d}{\overline{\im d}},
\]
where the over-line denotes the closure. For general $C^*$  we first take the Hausdorff completion and then take its reduced cohomology. Note that if $C^*$ is a Banach space then $\overline{H}(C,d)$ is itself a Banach space with the induced quotient norm which is defined by
\[
\|[a]\|:=\inf_{b\in[a]}\|b\|.
\]
\end{df}
For the ensuing discussion, we consider the vector space $CF^*(H,J;\K)$ as a vector space over $R$, forgetting the action of the Novikov parameter. For $a>0$ define $CF^*_{(-\infty,a)}(H,J;\K)$ to be the $R$-sub-complex of $CF^*(H,J;\K)$ generated by periodic orbits of action less than $a$. Define by $CF^*_{[a,b)}(H;\K)$ the quotient complex
\begin{equation}\label{eqFiltCF}
CF^*_{[a,b)}(H,J;\K):=CF^*_{(-\infty,b)}(H,J;\K)/CF^*_{(-\infty,a)}(H,J;\K).
\end{equation}
Denote by $HF^*_{[a,b)}(H,J;\K)$  the corresponding cohomology groups. These are vector spaces over $R$.
\begin{tm}
\begin{enumerate}
\item A continuous chain map between topological complexes induces a well defined map on the reduced cohomologies.
\item A nullhomotopic map induces the zero map on reduced cohomology.
\end{enumerate}
\end{tm}
\begin{proof}
For the first assertion, continuity implies that $\overline{\im d}$ is mapped into $\overline{\im d}$. For the second assertion, note that $f$ maps all cycles into $\im d\subset\overline{\im d}$.
\end{proof}
\begin{rem}
A short exact sequence of topological complexes with continuous maps induces a long sequence of reduced cohomologies. However, exactness of this sequence only holds under special assumptions. A reference in the case of Hilbert complexes is \cite{Luck2013}.
\end{rem}
Let $C^*$ be a topological complex whose topology is induced by a filtration by sub-complexes $\{C^*_t\}_{t\in\R}$ so that $C^*_t\subset C^*_{t'}$ whenever $t<t'$. For any element $a\in C^*$ let $\val(a):=\inf\{t|a\in C^*_t\}$. Then $\val$ naturally induces a filtration on $\overline{H}^*(C^*)$ defined by $\val([a])=\inf_{c\in[a]}\val c$. Define a filtration on the vector space $\varprojlim_tH^*(C^*/C^*_t)$ by
\[
\val(x)=\inf\left\{t_0\Big|x\in\ker\left(\varprojlim_tH^*(C^*/C^*_t)\to H^*(C^*/C^*_{t_0})\right)\right\},
\]
for $x\in\varprojlim_tH^*(C^*/C^*_t).$ Observe that the spaces $\varprojlim_tH^*(C^*/C^*_t)$ and $H^*(C^*/C^*_{t_0})$ are both complete with respect to the norm $\|a\|:=e^{\val(a)}$, and are thus Banach spaces.
\begin{tm}\label{lmReduceHomInv}
We have
\[
\overline{H}^*(C^*)=\varprojlim_tH^*(C^*/C^*_t),
\]
as Banach spaces. %For the topology on the right hand side we consider the spaces $H^*(C^*/C^*_t)$ with the discrete topology and take coarsest topology on $\varprojlim_tH^*(C/C_t)$ for which the natural maps are continuous.
\end{tm}
\begin{proof}
Any two cycles in the Hausdorff completion of $C^*$ representing the same element in $\overline{H}^*(C^*)$, represent the same element of $H^*(C^*/C^*_t)$ for each $t$. We thus get a well defined morphism
\[
f:\overline{H}^*(C^*)\to\varprojlim_tH^*(C^*/C^*_t).
\]
If $c\in C^*$ is a cycle and $f(c)=0$ then $c$ is a boundary$\mod C^*_t$ for each $t$. In particular it is in the closure of the space of boundaries of $C^*$, so $[c]=0$ in the reduced homology. Thus $f$ is injective. We now show that $f$ is surjective. Let $a\in \varprojlim_tH^*(C^*/C^*_t)$. We can compute the inverse limit by taking subset $\{t_i\}$ of the index set $\R$ which is discrete, bounded above and unbounded below. In this presentation, $a$ is a sequence
\[
(\dots,[a_i],[a_{i+1}],\dots,[a_0]),\quad [a_i]\in H^*(C^*/C^*_{t_i})
\]
where $[a_i]$ maps to $[a_{i+1}]$ under the natural map induced on homology. We consider the representatives $a_i$ as living in $C^*$ and claim that  they can be chosen so that already at the chain level, $a_i$ maps to $a_{i+1}\mod C^*_{t_{i+1}}$. Inductively, suppose this holds for all $ i_0<i\leq 0$. We have that there is a $b_{i_0}\in C^*$ such that $a_{i_0+1}-a_{i_0}=db_{i_0}\mod C^*_{t_{i_0+1}}$. Replace $a_{i_0}$ by $a_{i_0}+db_{i_0}$ to get the claim for $i_0$. The sequence $\{a_i\}$ converges as $i\to-\infty$ to an element $\hat{a}$ in the completion of $\hat{C}^*$. By construction, $d\hat{a}=0$ and $f([\hat{a}])=a$. Unwinding definitions one verifies that $f$ preserves the valuations and is thus a Banach space isometry.
\end{proof}
\begin{ex}
In this example we illustrate how reduced and unreduced cohomology may differ from one another. 
Fix a field $R$ and consider the vector space
\[
C^*=\oplus_{i=1}^{\infty}R\langle x^i,y^i\rangle[q]/q^2,
\]
where the $x^i,y^i$ are formal symbols of degree $0$ and $1$ respectively for all $i$, and $q$ is a formal symbol of degree $-1$. Define a non-Archimedean valuation on $C^*$ by taking $\val(x^i)=0=\val(q)$, and $\val(y^i)=-i$.  Define a differential by
\[
dx^i:=y^i,\quad dy^i:=0,\quad d(qx^i)=qy^i+x^{i+1}-x^i,\quad d(qy^i)=y^i-y^{i+1}.
\]
Suppose $\gamma\in C^*$ is a \textit{finite} sum of generators satisfying $d\gamma=0$. Then $\gamma$ is a linear combination of the $y^i$ and elements of the form $qy^i-x^i+x^{i+1}$. That is, a linear combination of $dx^i$ and $d(qx^i)$. Thus the homology of $C^*$ vanishes. Consider now the complex $\widehat{C}^*$ obtained by completing $C^*$ with respect to the valuation, and let
\[
\gamma:=x^1+\sum_{i=1}^{\infty}(-1)^iqy^i.
\]
Then $\gamma$ is well defined in $\widehat{C}^*$, and $d\gamma=0$. We show $\gamma$ is not a coboundary and is not even approximated by a sequence of coboundaries. For this consider its image in $C^*/C^*_t$. It is straightforward to verify that the class of $\gamma$ is equivalent mod $C^*_t$ to the class of $x^i$  for any $i>-t$ and that $\val([x^i])=0$. In this case one verifies that the image of the differential is closed. It follows that the reduced homology of the completion $\widehat{C}^*$ (which in this case equals the ordinary homology) is non-zero. 

Consider now the sub-complex $\widehat{R\langle x^i,y^i\rangle}$ of $\widehat{C}^*$. Let
\[
\gamma=\sum y^i.
\]
Then $\gamma$ is again a convergent sum. Moreover, for any $t$ we have that $\gamma$ is a boundary mod $C^*_t$. Indeed, for any $N>t$ we have
\[
\gamma=d\sum_{i=1}^Nx^i\mod C^*_t.
\]
But the sum on the right hand side does not converge as $i\to\infty$. One verifies in this case that the reduced cohomology $\overline{H}^*(\widehat{R\langle x^i,y^i\rangle})$ vanishes while the unreduced one does not.
\end{ex}
\subsection{Floer cohomology of lower semi-continuous exhaustion functions}\label{SubsecLscFC}
\begin{tm}\label{tmFloerInvBdd}
Let $(H_0, J_0)$ and $(H_1, J_1)$ be dissipative. Suppose
\begin{equation}
H_1-c\leq H_0\leq H_1.
\end{equation}
Then the canonical continuation map $HF^*(H_0,J_0)\to HF^*(H_1,J_1)$  is an isomorphism which decreases norms by a factor of at most $e^{-c}$. In particular, when $H_0=H_1$ the continuation map is an isometry.
\end{tm}

\begin{proof}
Recall that for a monotone homotopy, the induced continuation map is valuation decreasing. Consider the composition of continuation maps
\[
HF^*(H_0,J_0)\to HF^*(H_1,J_1)\to HF^* (H_0+c,J_0).
\]
It coincides with the continuation map
\[
HF^*(H_0,J_0)\to HF^* (H_0+c,J_0).
\]
The latter stems from a naive identification of the underlying complexes with the norm scaled by $e^{-c}$. This shows that the map $HF^*(H_0,J_0)\to HF^*(H_1,J_1)$ is right invertible and decreases norm by at most $e^{-c}$. Left invertibility is shown similarly.
\end{proof}
Henceforth we drop $J$ from the notation and talk about $HF^*(H)$. Abusing notation we will also drop $J$ from the chain level notation. Accordingly, we refer to a Hamiltonian $H$ as dissipative if there exists a compatible almost complex structure $J$ such that $(H,J)$ is dissipative.

As a consequence of Theorem~\ref{tmFloerInvBdd} we may extend the definition of Floer cohomology to some Hamiltonians which are degenerate or even non-smooth.

\begin{lm}\label{lmdfFloCohUn}
Let $H_i$ and $F_i$ pointwise monotone increasing sequences of non-degenerate dissipative Hamiltonians both converging uniformly in $C^0$ to the same continuous function $H$. Then there is an isomorphism
\[
\varinjlim_iHF^*(H_i)\to\varinjlim_iHF^*(F_i),
\]
which is natural in that it commutes with all continuation maps involving dissipative and non-degenerate Hamiltonians. We may thus define
\begin{equation}\label{eqdfFloCohUn}
HF^*(H):=\varinjlim_n HF^*(H_n).
\end{equation}
Define a semi-norm on $HF^*(H)$  by
\begin{equation}
\|a\|:=\inf_i\|a_i\|
\end{equation}
where $a_i\in HF^*(H_i)$ maps to $a$ under the natural map. Then $\|\cdot\|$ is a non-Archimedean semi-norm on $HF^*(H)$ which is independent of the choice of $H_i$.  Moreover, when $H$ is smooth, dissipative and non-degenerate, the two definitions of $HF^*(H)$ as a semi-normed space coincide.
\end{lm}

\begin{proof}
Call a sequence $H_n$ as in the hypothesis \textit{admissible} if for each $n$ there is a constant $c_n>0$ for which $H-H_n\geq c_n$.  Given any two admissible sequences we can squeeze a subsequence of one between a subsequence of the other. The first part of the statement then follows by the universal property of the direct limit.
Given a not necessarily admissible monotone sequence $H_n$ converging to $H$, the sequence $H_{nk}:=H_n-\frac1{k},$ is admissible, monotone, and converges uniformly to $H_n.$ We thus have a natural isomorphism
  \[
  \varinjlim_n HF^*(H_n)=\varinjlim_n\varinjlim_k HF^*(H_{nk}).
  \]
But the sequence $H_{nn}$ is admissible and cofinal in the doubly indexed sequence $H_{nk}$ and so we have a natural isomorphism
\[
\varinjlim_n\varinjlim_k HF^*(H_{nk})=\varinjlim_n HF^*(H_{nn}).
\]
We have a similar relation for $F_{nk}:=F_n-1/k$ and $F_{nn}$. The sequences $H_{nn}$ and $F_{nn}$ are admissible, monotone and converge uniformly to $H$. Combined with the isomorphisms we just deduced, we obtain the isomorphism
\[
\varinjlim_n HF^*(H_n)=\varinjlim_n HF^*(H_{nn})=\varinjlim_n HF^*(F_{nn})=\varinjlim HF^*(F_n),
\]
where all the isomorphisms are natural.

To see that $\|\cdot\|$  defines a non-Archimedean semi-norm note that by Theorem~\ref{tmFloerInvBdd}, the sequence $\|a_i\|$ is bounded below. Since it is monotone decreasing, it is convergent. So,
 \[
 \|a+b\|=\lim_i\|a_i+b_i\|\leq \lim_i\max\left\{\|a_i\|,\|b_i\|\right\}= \max\{\|a\|,\|b\|\}.
 \]
 The homogeneity of $\|\cdot\|$ is obvious. In light of Theorem \ref{tmFloerInvBdd}, the argument for the independence of $\val$ on the choice of sequence is similar to the claim concerning the natural isomorphism. Finally, the last part of the claim take $F_n$ to be the constant sequence $F_n=H$.
\end{proof}

The definition of action truncated Floer homology groups also extends.
\begin{lm}\label{lmAcTruncCont}
Let $H$ and $H_n$ be  smooth non-degenerate and dissipative, and suppose the sequence $H_n$ is monotone and converges uniformly to $H$. Then the natural map
\begin{equation}\label{eqDfFlAcCo}
\varinjlim_n HF^*_{[a,b)}(H_n)\to{HF}^*_{[a,b)}(H)
\end{equation}
is an isomorphism. If we drop the assumption that $H$ is dissipative and define ${HF}^*_{[a,b)}(H)$ by \eqref{eqDfFlAcCo}, the right hand side is independent of the choice of $H_n$.
\end{lm}
\textit{Caution}: The claim does not necessarily hold if we consider other segments such as $(a,b]$.

\begin{proof}
As in the proof of Lemma \ref{lmdfFloCohUn}, the first part is proven by squeezing a sequence of the form  $H_n-c_{nk}$ into a sequence $H-c_n$. The second part is proven by a similar squeezing. The argument is spelled out in the proof of Lemma \ref{lmdfFloCohUn}.
\end{proof}
The next theorem is key for what follows. It shows that truncated Floer homology is continuous with respect to convergence on compact sets.
\begin{tm}\label{lmConstApp}
Let $\{H_n\}$ be a monotone increasing sequence of dissipative Hamiltonians converging pointwise to a dissipative Hamiltonian $H$.  Then for any real $a<b$ we have that the natural map
\[
f:\varinjlim_{i}HF^*_{[a,b)}(H_n)\to HF^*_{[a,b)}(H),
\]
is an isomorphism.
\end{tm}
\begin{proof}
Fix an almost complex structure $J$ for which $(H,J)$ and $(H_i,J)$ are dissipative. Without loss of generality we may assume that all the involved Floer data are regular and non-degenerate. As in Lemma \ref{lmAcTruncCont} we reduce to the case where $H-H_n\geq c_n>0$ for some $c_n>0$. By Dini's Theorem, the $H_i$ converge to $H$ uniformly on compact sets. By squeezing in an appropriate sequence we may assume that there is an exhaustion of $M$ by compact sets $K_n$ such that  $H_n=H-c_n$ on $K_n$.  For such a sequence we have that for fixed real number $E>0$ and compact set $K$, the numbers $R(E,K)$ of Theorem~\ref{tmFloSolDiamEst}, defined for each of the $H_n$, stabilize as $n\to\infty$. So, given an $i$ and a cocycle $\gamma\in CF^*_{[a,b)}(H_i)$  there is a compact set $K$ and an $i_0>i$  such that any continuation trajectory $f_{i,i'}(\gamma)$ or $f_i(\gamma)$  is contained mod $CF^*_{(-\infty,a)}$ in $K$. Here $f_{i}:CF^*_{[a,b)}(H_{i},J)\to CF^*_{[a,b)}(H,J)$ and $f_{i,i'}:CF^*_{[a,b)}(H_{i},J)\to CF^*_{[a,b)}(H_{i'},J)$  are the natural continuation maps. Indeed, since we are considering only trajectories of energy less than $b-a,$  Theorem~\ref{tmFloSolDiamEst} provides an estimate on the diameter as required. The same claim holds for composite trajectories of the form $d\circ f_{i,i'}$ and $d\circ f_i$ etc.

Since $H_i|_{K_i}=H|_{K_i}-c_i$, we may identify those periodic orbits of $H_i$ which are inside $K_i$ with the periodic orbits of $H$ in the same region. For each periodic orbit $\gamma$ of $H$ the corresponding periodic orbit of $H_i$ is mapped mod $a$ by the continuation map to $\gamma$.  Indeed, for $i$  large enough, any continuation trajectory emanating from $\gamma$ and having energy at most $b-a$ is contained in $K_i$ and so satisfies a translation invariant equation. To be rigid it must be trivial.  Moreover, taking $i$ still larger, the same claim is true for the periodic orbits appearing in the expansion of $d\gamma\mod i$. This shows that $f$ is surjective. Injectivity follows in the same way. Namely, suppose there is an $i$ and a $\delta\in CF^*_{[a,b]}(H)$ such that $f_i(\gamma)=d\delta\bmod a$ in $CF^*(H)$. For $i$ large enough the same relation will hold in $CF^*(H_i)$.
%
%By Dini's Theorem, any monotone sequence converging to $H$ converges uniformly on compact sets. Thus, the generalization to arbitrary monotone sequences is similar to the proof of Lemma \ref{lmAcTruncCont}.
\end{proof}

We use the notation $\overline{HF}^*(H)$ for the reduced cohomology of $CF^*(H,J)$.
\begin{cy}\label{CySeqDefHF}
We have for any dissipative Hamiltonian $H$ and any sequence $H_i$ of dissipative Hamiltonians converging to $H$ uniformly on compact sets, the natural map
\begin{equation}\label{EqIsoRes}
\varprojlim_a\varinjlim_{b,i}HF^*_{[a,b)}(H_i)\to\overline{HF}^*(H)
\end{equation}
is an isomorphism. Moreover, we have for any $\alpha\in \overline{HF}^*(H)$ that
\begin{equation}\label{eqInducedValInf}
\val(\alpha)=\inf_a\left\{\alpha\in\ker\left(\overline{HF}^*(H)\to\varinjlim_{b,i}HF^*_{[a,b)}(H_i)\right)\right\}
\end{equation}
upgrading the isomorphism~\eqref{EqIsoRes} to an isometry of Banach spaces.
\end{cy}
\begin{proof}
We have by Theorem~\ref{lmConstApp} and by exactness of the direct limit
\begin{equation}\label{eqSeqDefHF}
\varinjlim_{b,i}HF^*_{[a,b)}(H_i)=\varinjlim_{b}HF^*_{[a,b)}(H)=HF^*_{[a,\infty)}(H).
\end{equation}
So, by Theorem~\ref{lmReduceHomInv} we obtain the isomorphism of Banach spaces. %For the second part of the claim, the proof of Lemma~\ref{lmConstApp} shows that $\alpha$ is a boundary$\mod a$ if and only if its image in $\varinjlim_{b,i}HF^*_{[a,b)}(H_i)$ vanishes.
\end{proof}

The last corollary will allow to extend the definition of reduced Floer homology to arbitrary lower semi-continuous exhaustion functions. That is, a lower semi-continuous function $H:\R/\Z\times M\to\R$ which is proper and bounded from below. But first we need to formulate an approximation lemma for such functions.

\begin{lm}\label{tmLoSCAppCo}
Let $H:\R/\Z\times M\to\R\cup\{\infty\}$ be a lower semi-continuous function which is proper and bounded below. Let $F:\R/\Z\times M\to\R$ be a smooth proper exhaustion function  such that $F< H-\epsilon$ pointwise for some $\epsilon>0$. Then there is a pointwise monotone sequence of smooth exhaustion functions  $H_n:\R/\Z\times M\to\R$ such that $H_n$ converges pointwise to $H$ everywhere and such that for each $n$ there is a compact set $K_n$ so that $H_n|_{M\setminus K_n}=F|_{M\setminus K_n}$.
\end{lm}
\begin{proof}
It is  a standard fact that $H$ is the supremum of a monotone increasing sequence $H'_n$ of smooth functions. For each $n$ take $H_n$ to coincide with $H'_n$ on $H^{-1}((-\infty,n))$, to equal $n$ on the set
\[
\{(t,x):H_t(x)\geq n\geq F_{t}(x)\},
\]
and to equal $F$ on $F^{-1}(n,\infty)$. Then $H_n$ is well defined and continuous. After a slight perturbation it is smooth and satisfies all the requirements.
\end{proof}
By Theorem \ref{tmBoundGradDis} we can take $F$ in the previous Lemma to be a function with sufficiently small Lipschitz constant outside of a compact set so as to be dissipative. Then all the functions in the sequence $H_n$ are also dissipative as they coincide with $F$ outside of a compact set. Thus \textit{any lower semi-continuous exhaustion function is  the pointwise limit of a monotone sequence of dissipative Hamiltonians}.
\begin{lm}\label{lmHFMonIndep}
Let $H$ be a lower semi-continuous exhaustion function. Let $H_n\leq F_n\leq H$ be a pair of monotone sequences of dissipative Hamiltonians each converging pointwise to $H$. Then for any $a<b\in\R$ the natural continuation map
\begin{equation}\label{eqLimSmSC}
\varinjlim_{i}HF^*_{[a,b)}(H_i)\to\varinjlim_{i}HF^*_{[a,b)}(F_i)
\end{equation}
is an isomorphism.
\end{lm}

\begin{proof}
For each $n$ choose monotone sequences $H_{kn},F_{kn}$ such that the following properties hold. 
First, for fixed $n$, they converge on compact sets to $H_n, F_n$ respectively as $k\to\infty$.  Secondly, there is an exhaustion of $M$ by pre-compact sets $U_k$ such that $H_{kn}$ and $F_{kn}$ coincide with with $H_1-\frac1{k}$ on the complement of $U_k$. Such sequences exist by Lemma \ref{tmLoSCAppCo}. %Subtracting the constant function $1/k$ from the $k$th element in each sequence if necessary, 
By Lemma \ref{lmConstApp} we have natural isomorphisms
\[
\lim_{k\to\infty}HF^*_{[a,b)}(H_{kn})=HF^*_{[a,b)}(H_{n}),
\]
and a similar isomorphism relating the Floer cohomologies of $F_{kn}$ and $F_n$. The map appearing in equation \eqref{eqLimSmSC} corresponds under this isomorphism to the natural map
\[
\lim_{k,n\to\infty}HF^*_{[a,b)}(H_{kn})\to \lim_{k,n\to\infty}HF^*_{[a,b)}(F_{kn}).
\]
For each $k$ we have that $H-H_{kn}$ and $H-F_{kn}$ are bounded away from $0$. Moreover, for each compact set $K$ we can make $H_{kn},F_{kn}$ arbitrarily close to $H$ on $K$ by adjusting $k,n$.  It follows that for each $k,n$ we can find numbers $k',n',k'',n''$ such that $F_{k''n''}> H_{k'n'}> F_{kn}$. Thus we  can squeeze a cofinal subsequence of $H_{kn}$ into a cofinal subsequence of $F_{kn}$. The claim follows by the same argument as in Lemma \ref{lmdfFloCohUn}.
\end{proof}

\begin{lm}\label{lmRoofLowSeq}
Let $H$ be a lower semi-continuous exhaustion function. Let $F_n,G_n$ be a pair of monotone sequences of dissipative Hamiltonians each converging pointwise to $H$. Then there exists a monotone sequence of dissipative Hamiltonian $H_n$ such that $H_n\leq\min\{F_n,G_n\}$ and $H_n$ converges pointwise to $H$.
\end{lm}
\begin{proof}
The function $H'_n=\min\{F_n,G_n\}$ is continuous. As in  Lemma \ref{tmLoSCAppCo}, let $H''_n$ be a continuous function which coincides with $H'_n$ on $H^{'-1}_n(-\infty,n)$ and with some fixed smooth dissipative function $F\leq H_1$ everywhere else. Then $H''_n$ is continuous, the sequence $H''_n$ is monotone, and it still converges to $H$. Finally,  replace $H''_n$ by a smooth $H_n$ satisfying $H''_n-1/n\leq H_n\leq H''_n$ and which equals $F$ outside some compact set. Then all the requirements are satisfied.
\end{proof}

\begin{proof}[Proof of Theorem~\ref{tmUniExtFl}]
The surjectivity statement follows from Lemma \ref{tmLoSCAppCo} as stated in the paragraph right after the proof of Lemma \ref{tmLoSCAppCo}.

For a pointwise monotone sequence $\{H_i\}$ of regular dissipative Hamiltonians, define
\begin{equation}\label{eqDfFlRedSec}
\overline{HF}^*(\{H_i\}):=\varprojlim_a\varinjlim_{b,i}HF^*_{[a,b)}(H_i).
\end{equation}
with the norm given by the right hand side of ~\eqref{eqInducedValInf}.

Let $\sup(\{H^1_i\})=\sup(\{H^2_i\})=H$ for some $H\in\cH_{s.c.}$. By Lemma \ref{lmRoofLowSeq} we can find a third sequence $\{H_i\}$ of regular dissipative Hamiltonians such that $H_i\leq\min\{H^1_i,H^2_i\}$. By Lemma \ref{lmHFMonIndep} and Equation \eqref{eqDfFlRedSec} it follows that we have natural isomorphisms
\[
\overline{HF}^*(H^1_i)=\overline{HF}^*(H_i)=\overline{HF}^*(H^2_i).
\]
For an arbitrary $H\in\cH_{s.c.}$ we define $\overline{HF}^*(H)$ as the pushout over all approximating sequences $\{H_i\}$ of $\overline{HF}^*(\{H_i\})$ under the natural isomorphisms just described. By naturality we get an induced functorial continuation map for $H_1\leq H_2$. This defines the functor $\overline{HF}^*$ on $(\cH_{s.c.},\leq)$. To see that the restriction to $(\cH_{d,reg},\leq)$  agrees with the previous definition, note that any element $H\in\cH_{d,reg}$ can be considered as a constant sequence $\{H_i\}$ with $H_i=H$.
\end{proof}

In fact we have proven the following stronger Lemma which is used below.
\begin{lm}\label{lmOverHFindep}
Let $H$ be a lower semi-continuous exhaustion function. Let $H_n, F_n$ be a pair of monotone sequences of dissipative Hamiltonians each converging pointwise to $H$. Such sequences are guaranteed to exist by Theorem \ref{tmLoSCAppCo}.
Then for any segment $[a,b)$ there exists an isomorphism
\[
\varinjlim_{i}HF^*_{[a,b)}(H_i)\to\varinjlim_{i}HF^*_{[a,b)}(F_i).
\]
This isomorphism is natural in the sense that it commutes with all induced continuation maps.
We thus define
\begin{equation}\label{eqDfFiltLSC}
HF^*_{[a,b)}(H):=\varinjlim_{i}HF^*_{[a,b)}(H_i).
\end{equation}
 In other words, we have \begin{equation}\label{eqDfFiltLSC1}
\overline{HF}^*(H)=\varprojlim_a\varinjlim_bHF^*_{[a,b)}(H).
\end{equation}
\end{lm}
In the next section we will see that it is actually possible to define $\overline{HF}^*$ as the reduced cohomology of an appropriate chain complex which is well defined up to filtered quasi-isomorphism.
\subsection{The chain level construction}\label{subsecChainLev}  We apply the telescope construction appearing in \cite{abouzaidSeidel2010} to define $HF^*(H)$, for general lower semi-continuous $H$, as the cohomology of a certain complex. Namely, let $(H_i,J_i)$ be a sequence of dissipative Floer data. Let $q$ be a formal variable of degree $-1$ satisfying $q^2=0$. Write\footnote{As usual we abuse notation, omitting mention of the almost complex structures.}
\[
SC^*(\{H_i\}):=\oplus_{i=1}^{\infty}CF^*(H_i)[q],
\]
and equip it with the differential
\[
\delta(a+qb):=(-1)^{\deg a}da+(-1)^{\deg b}(qdb+\kappa(b)-b),
\]
where $\kappa$ denotes the continuation map $CF^*(H_i)\to CF^*(H_{i+1})$ for each $i$. Let $\widehat{SC}^*(\{H_i\})$ denote the completion with respect to the action filtration. It is shown in \cite{abouzaidSeidel2010} that, ignoring topology, there is a natural isomorphism
\begin{equation}\label{eqTelVarinjlim}
\varinjlim_i HF^*(H_i)=H^*(SC^*(\{H_i\}),\delta).
\end{equation}
This isomorphism arises as follows. Consider the underlying complexes $CF^*(H_i)$ with differential $\delta(a):=(-1)^{\deg a}d(a)$. This change does not affect anything at the cohomology level, and continuation maps remain chain maps. The obvious embeddings $(CF^*(H_i),\delta)\hookrightarrow SC^*(\{H_i\})$ commute up to homotopy with the continuation maps thus giving rise to the map in \eqref{eqTelVarinjlim}. For more details see \cite{abouzaidSeidel2010}. %Note that isomorphism \eqref{eqTelVarinjlim} does \textit{not} apply to reduced cohomology. We will use it by applying it to truncated Floer homology groups.

\begin{df}\label{dfFiltQuasIs}
Let $(C_i^*,d)$ for $i=1,2$ be complexes filtered by a valuation. We say that a valuation decreasing chain map $f:C_1^*\to C_2^*$ is a \textbf{filtered quasi-isomorphism} if it induces a an isomorphism on filtered homologies $H^*_{[a,b)}$ for $a>-\infty$.% We say that $f$ is a \textbf{reduced quasi-isomorphism} if $f$ induces an isomorphism on the reduced cohomology.
We say that $(C_1^*,d)$ is filtered quasi-isomorphic to $(C_2^*,d)$ if there is a zig-zag of filtered quasi-isomorphisms starting at one and ending at the other.
\end{df}
\begin{tm}\label{tmOverHFChainLv}
Let $H$ be a lower semi-continuous exhaustion function and let $(H^1_i,J^1_i)$ and $(H^2_i,J^2_i)$ be monotone increasing sequences of dissipative Floer data such that for $j=1,2$, $H^j_i$ converges to $H$ pointwise. Then $\widehat{SC}^*(\{H_i^1\})$ is filtered quasi-isomorphic to $\widehat{SC}^*(\{H_i^2\})$. If $H$ is itself dissipative, they are both filtered quasi-isomorphic to $\widehat{CF}^*(H)$.
\end{tm}
The proof of Theorem~\ref{tmOverHFChainLv} is carried out after establishing the following few lemmas which are of interest in their own right.
\begin{lm}\label{lmchTruncIsom}
We have for any interval $-\infty<a<b\leq\infty$,
\[
H^*_{[a,b)}\left({\widehat{SC}^*(\{H_i\})},\delta\right)=\varinjlim_iHF^*_{[a,b)}(H_i).
\]
Further, we have an isometry of Banach spaces
\[
\overline{H}^*\left(\widehat{SC}^*(\{H_i\}),\delta\right)=\varprojlim_a\varinjlim_iHF^*_{[a,\infty)}(H_i).
\]
where on the right hand side we define the norm by Equation~\eqref{eqInducedValInf} and on the left hand side we take the norm induced from the $CF^*(H_i)$.
\end{lm}
\begin{proof}
The first half of the claim is what is shown in \cite{abouzaidSeidel2010} since no topology is involved. For the second half, the isomorphism of topological vector spaces follows from the first half and Lemma~\ref{lmReduceHomInv}. The fact that this is an isometry also follows from the first half by unwinding definitions.

\end{proof}
\begin{lm}\label{lmChaLevCont}
Let $F^0=\{(H^0_i,J^0_i)\},F^1=\{(H^1_i,J^1_i)\}$ be two monotone sequences of dissipative Floer data such that $H^0_i\leq H^1_i$. Let $\mathfrak{H}_{i,s}$ be a monotone dissipative interpolating family. Then there is a filtration decreasing continuation map
\[
\phi_{\mathfrak{H}}:\widehat{SC}^*(\{H^0_i\})\to\widehat{SC}^*(\{H^1_i\}),
\]
inducing, for each interval $[a,b)$, the canonical continuation map
\[
\varinjlim_iHF^*_{[a,b)}(H^0_i)\to \varinjlim_iHF^*_{[a,b)}(
H^1_i).
\]
If $\mathfrak{H}^1,\mathfrak{H}^2$ are two homotopies interpolating between $F^0$ and $F^1$, there exists a filtration decreasing chain homotopy operator
\[
\mathfrak{K}:\widehat{SC}^*(\{H^0_i\})\to\widehat{SC}^{*+1}(\{H^1_i\})
\]
such that
\[
\phi_{\mathfrak{H}^1}-\phi_{\mathfrak{H}^2}=\delta\circ\mathfrak{K}+\mathfrak{K}\circ\delta. \]
\end{lm}
\begin{proof}
Let $j_i:CF^*(H_i^0)\to CF^{*+1}(H_{i+1}^1)$ be the chain homotopy operator satisfying
\[
f_{\mathfrak{H}_{i+1}}\circ \kappa-\kappa\circ f_{\mathfrak{H}_i}=d\circ j_i+j_i\circ d.
\]
Define
\[
\phi:\widehat{SC}^*(F^0)\to\widehat{SC}^*(F^1),
\]
by $\phi(a+qb):=f(a)+qf(b)+j(b)$. One verifies that this is indeed a chain map. Inspecting isomorphism \eqref{eqTelVarinjlim} one finds that the homology level square
\[
\xymatrix{
\varinjlim HF^*_{[a,b)}(H^0_i)\ar[d]\ar[r]&\varinjlim HF^*_{[a,b)}(H^1_i)\ar[d] \\
H^*_{[a,b)}(SC^*(\{H^0_i\})) \ar[r] &H^*_{[a,b)}(SC^*(\{H^1_i\}))
}
\]
commutes. This proves the first half of the claim. To define the chain homotopy operator $\mathfrak{K}$ let $l_i: CF^*(H^1_i)\to CF^{*+1}(H^2_i)$ be the chain homotopy operator associated to a family of homotopies interpolating between $\mathfrak{H}^1$ and $\mathfrak{H}^2$. Let $j^n_i$ be the homotopies between $f_{i+1}^n\circ\kappa$ and $\kappa\circ f_i^n$ for $n=1,2$. Let
\[
m_i:CF^*(H^1_i)\to CF^{*+2}(H^2_{i+1}),
\]
be a degree $2$ operator satisfying
\[
d\circ m+m\circ d=j^1+\kappa\circ l-(l\circ\kappa+j^2)
\]
We show that such an $m$ exists before proceeding. To see this note that each term on the right hand side is a chain homotopy operator from $\kappa\circ f_1$ to $f_2\circ\kappa$ coming from appropriate 1-dimensional families of interpolating homotopies. By standard Floer theoretic machinery, a generic two dimensional family interpolating these 1-dimensional homotopies gives rise to an operator $m$ as required. By energy considerations, $m$ is action decreasing.

Having established the existence of $m$ we define the chain homotopy
\[
\mathfrak{K}(a+qb):=(-)^{deg(a+1)}l(a)+(-1)^{deg(b+1)}(ql(b)+m(b)).
\]
A straightforward but somewhat tedious calculation shows that $\mathfrak{K}$ is indeed a chain homotopy operator as required.
\end{proof}
\begin{comment}
\begin{lm}\label{lmOverHFChainLv}
\begin{enumerate}
    \item If $\{H^1_i\}$ and $\{H^2_i\}$ are sequences of dissipative Hamiltonians converging to a proper continuous function which is bounded below, then $WF^*(\{H^1_i\})$ and $WF^*(\{H^2_i\})$ are reduced quasi-isomorphic.
\item Let $H$ be dissipative. Let $H_i=H$ for all $i$. Then $\overline{HF}^*(H)$ and $WF^*(\{H_i\})$ are reduced
    quasi-isomorphic.
\end{enumerate}
\end{lm}
\end{comment}
\begin{proof}[Proof of Theorem~\ref{tmOverHFChainLv}]
Use Lemma \ref{lmRoofLowSeq} to find a monotone sequence of dissipative Hamiltonians $\{H^0_i\}$ dominated by $\{H^j_i\}$ for $j=1,2$ and still converging pointwise to $H$. By Lemmas \ref{lmHFMonIndep} and ~\ref{lmChaLevCont}, the induced continuation map quasi-isomorphism for each finite truncation.
\end{proof}
\begin{comment}
\begin{proof}[Proof of Theorem~\ref{tmUniExtFl}]
We pick any monotone sequence $H_i$ of non-degenerate dissipative Hamiltonians converging to $H$ pointwise. For $I\in\mathcal{I}$, we define $\overline{HF}_I^*(H)$ be taking the truncated homologies of $SC^*(\{H_i\})$. By Theorem~\ref{tmOverHFChainLv} this is well defined and coincides with the previous definition for dissipative Hamiltonians. The continuity property follows from Lemma~\ref{lmchTruncIsom} and implies uniqueness at the homology level. The functoriality follows from Lemma~\ref{lmChaLevCont}.
\end{proof}
\end{comment}
\section{The product structure}\label{subsecFloerProd}

\subsection{Floer data for the pair of pants product}

For time dependent Hamiltonians $H_1,H_2$, let $H_1*H_2:\R/\Z\times M\to\R$ be the time dependent function
\[
(H_1*H_2)_t=\begin{cases} 2H_1(t),&\quad t\in[0,1/2)\\ 2H_2(2t-1), &\quad t\in[1/2,1).
\end{cases}
\]
Note that $H_1*H_2$ depends discontinuously on $t$ with jump discontinuities at $t=1/2,0\sim 1$. The operation is introduced for notational convenience. A triple $(H_0,H_1,H_2)$ is called a (strict) product triple if $H_2\geq H_1*H_0$ ($H_2> H_1*H_0$).

Denote by $\Sigma$ the pair of pants $S^2\setminus\{0,1,\infty\}$. For our convenience we pick cylindrical ends which extend globally as follows. Consider the holomorphic map $\psi:\Sigma\to\R\times \R/\Z$ given by
\[
z\mapsto \frac1{2\pi}\Log z(z-1).
\]
This defines cylindrical coordinates
\[
s=\frac1{2\pi}\log |z(z-1)|,\qquad t=\frac1{2\pi}\arg (z(z-1))
\]
in punctured neighborhoods of $0,1$,  coordinatized as inputs. For the cylindrical end at $\infty$ we take
\[
s=\frac1{2\pi}\log |z(z-1)|,\qquad t=\frac1{4\pi}\arg (z(z-1)).
\]
Thus $\infty$ is coordinatized as an output. Henceforth we write
\[
\alpha_{\Sigma}:=\frac1{2\pi}d\arg z(z-1)
\]
 and $h_{\Sigma}:\Sigma\to\R$ the function
 \[
 z\mapsto s=\frac1{2\pi}\log |z(z-1)|.
 \]
 Then $dh_{\Sigma}\wedge\alpha_{\Sigma}\geq 0$ and $h_\Sigma$ has a single critical point at $z=1/2$. Note also that at the output we have $\alpha_{\Sigma}=2dt$ while at each input we have $\alpha_{\Sigma}=dt$. We consider the coordinate $t$ to be well defined on the complement of $s=h_{\Sigma}=1/2$.

\begin{df}
A Floer datum $(H,J)$ is called super-dissipative if for any $f:\R/\Z\to\R$ we have that $(fH,J)$ is dissipative.
\end{df}
\begin{lm}
If $H$ is a Lipschitz exhaustion function and $h:\R\to\R$ is a monotone function satisfying $\lim_{t\to\infty} h'(t)\to 0$ then $h\circ H$ is super dissipative.
\end{lm}
\begin{proof}
The Lipschitz constant of $ f\cdot h\circ H$ is arbitrarily small outside of a sufficiently large compact set. The claim follows by Theorem \ref{tmBoundGradDis}.
\end{proof}
For a super dissipative $(H,J)$ denote by $\mathcal{F}(H,J)$ the set of all pairs in $C^{\infty}(\R/\Z\times M)\times\cJ$ which coincide outside of some compact set with $(fH,J)$ for some $f:\R/\Z\to\R_+$. We drop $J$ from the notation when there is no ambiguity. A $1$-form $\mathfrak{H}\in\Omega^1(\Sigma,C^{\infty}(M))$  is called \textit{$H$-admissible} if there is a function $G:\Sigma\times M \to\R$ such that $\mathfrak{H}=G\otimes\alpha_{\Sigma}$ and such that for each $x\in M$ we have
$d\mathfrak{H}(x)\geq 0$, and for each $z\in\Sigma$ we have
$G(z,\cdot)\in\cF(H)$. A Floer datum $(\mathfrak{H},J')$ is called $H$-admissible if, in addition, $J'$ is quasi-isometric to $J$. For an $H$-admissible product triple we denote by $\cP(H_0,H_1,H_2)$ the set of $H$-admissible data on the pair of pants which for $i=0,1,2$ equals $H_idt$ at the $i$th end. We refer to the set $\cP(H_0,H_1,H_2)$ as \emph{product data} for the triple $(H_0,H_1,H_2)$. Included in the set $\cP(H_0,H_1,H_2)$ are broken Floer data which are concatenations of monotone continuation data in $\cF(H)$ with $H$-admissible pairs of pants.

\begin{lm}\label{lmHAdmProdDat}
\begin{enumerate}
\item $H$-admissible 1-forms satisfy the hypotheses of Lemma \ref{lmTopGeoEnEst}.
\item\label{lmHAdmProdDatb} $H$-admissible product data are dissipative.
\item If $(H^0,H^1,H^2)\in\cF(H)^3$ is a strict  product triple then
    \[
    \cP(H^0,H^1,H^2)\neq\emptyset.
    \]
\item $\cP(H^0,H^1,H^2)$ is connected, and the path connecting any two elements is dissipative.
\end{enumerate}
\end{lm}
\begin{proof}
\begin{enumerate}
\item
Equations \eqref{EqPoisson} and \eqref{EqMonHom} hold by construction.
\item
Loopwise dissipativity follows by the assumption of super dissipativity of $H$. As for $i$-boundedness observe that the metric $g_{J_{\mathfrak{H}}}$ is uniformly equivalent to the metric $g_{J_H}$. In fact, if $H$ is Lipschitz, $g_{J_{\mathfrak{H}}}$ is uniformly equivalent to the product metric on $\Sigma\times M$.
\item
Let $F\in\cF(H)$ be a time independent Hamiltonian satisfying for all $x\in M$
\begin{align}
2\max \left\{H^0_0(x),H^1_0(x)\right\}&< F(x)\notag\\&<\min\left\{H^2_0(x),H^2_{1/2}(x)\right\}.\notag
\end{align}
Such a Hamiltonian exists by assumption.
Let
\[
G'_{t}(x):=\begin{cases} \max\{2H^0_{2t}(x),F(x)\},&\quad t\in[0,1/2]\notag\\
\max\{2H^1_{2t}(x),F(x)\},&\quad t\in[1/2,1]\notag.
\end{cases}
\]
Then $G'$ satisfies (A) $(H^0,H^1,G')$ is a product triple, (B) for each $x\in M$ there is a neighborhood $U\subset M$ and $I\subset\R/\Z$ of $\{0,1/2\}$ such that $G'_{t}(y)=F(x)$ for $(t,y)\in I\times U$ and (C) $G'<H_2$. The function $G'$ has non-smooth points away from $t\in\{0,1/2\}$ but it can be smoothed to a function $G$ such that the properties (A), (B) and (C) still hold. %Indeed, the properties are all preserved under convex combination, so such a smoothing can be carried out on each compact set of an exhaustion and then pieced together by a partition of unity.
Let $G^s_t$ be a family such that for $s\gg1/2$ we have $G^s=H_2$ and for $s\leq 1/2$ we have $G^s=G$ and such that $\partial_sG^s\geq 0$. This can again be pieced together by an appropriate Urysohn function.

We define $H'_z$ to equal $G_{h_\Sigma(z),t}$ for $z$ such that $h_{\Sigma}(z)\geq1/2$. By property (B), $H'_z$ extends smoothly beyond the branchpoint at $z=1/2$. On each input end we can smoothly interpolate between the input $H_i$  and $\frac1{2}G_s$ in a monotone way by employing an Urysohn function. The result $\mathfrak{H}:=H'\otimes\alpha_\Sigma$ is an $H$-admissible product datum.
\item If $H^1\alpha$ and $H^2\alpha$ are two (unbroken) $H$-admissible product data with the same inputs and output, so is any convex combination. If $H^1\alpha$ is a broken product datum, it can connected by a  path to an unbroken one by gluing. Thus $\cP(H_0,H_1,H_2)$ is connected. Dissipativity is immediate as in  part \ref{lmHAdmProdDatb}.
    \end{enumerate}
\end{proof}

\subsection{Construction of the pair of pants product}
\begin{lm}
Fix a super dissipative $H$. Let $(H_0,H_1,H_2)\in\cF(H)^3$ consist of non-degenerate Hamiltonians. Then for generic choice of element in $\cP(H_0,H_1,H_2)$ and for a generic path in $\cP(H_0,H_1,H_2)$ the associated $0$- and $1$-dimensional moduli spaces of Floer solutions are compact, smooth and of the expected dimension.
\end{lm}
\begin{proof}
By construction, an admissible product datum satisfies the hypotheses of Lemma \ref{lmTopGeoEnEst}. Therefore,
given generators $\tilde{\gamma_i}\in CF^*(H_i)$ the pairs of pants of a fixed dissipative Floer datum with $i$th end asymptotic to $\tilde{\gamma}_i$ have energy estimated according to Lemma~\ref{lmTopGeoEnEst} by $E^{top}(U)$, which equals the action difference $\mathcal{A}_{H_3}(\tilde{\gamma}_3)-\mathcal{A}_{H_1}(\tilde{\gamma}_1)-\mathcal{A}_{H_2}(\tilde{\gamma}_2).$
The first part of Lemma \ref{lmHAdmProdDat} and Theorem ~\ref{tmFloSolDiamEst} thus imply that they
are all contained in a-priori compact set $K$ depending on the differences $b_i-a_i$. The story is now the same as the closed case which is dealt with in the aspherical case in \cite{Schwarz95}. Sphere bubbling is treated in the exact same way as for the differential, continuation maps and chain homotopy operators which was done in detail in \S\ref{subsecBubbCont}. The upshot is that in the monotone or Calabi-Yau case, whether one is working with a single Floer datum or with a family of parameterized ones, spheres occur in codimension 4. All the arguments involve at most $1$-dimensional families of Floer solutions. Generically, there is no sphere bubbling for such families.
\end{proof}
\begin{tm}\label{tmDisProd}
Fix a super dissipative $H$. For any Floer datum $(H,J)$ denote by $CF^{*,0}(H,J)$ the subcomplex generated by contractible periodic orbits. A generic choice of admissible product datum for a generic product triple $(H_0,H_1,H_2)\in\cF(H)^3$ determines a bilinear map
\[
*:CF^{*,0}(H_1,J_1)\otimes CF^{*,0}(H_2,J_2)\to CF^{*,0}(H_3,J_3)
\]
satisfying
\begin{equation}\label{eqValProdEst}
\val(x_1*x_2)\leq\val(x_1)+\val(x_2).
\end{equation}
Moreover, we have $d(x_1*x_2)=dx_1*x_2+(-1)^{\deg (x_1)}x_1*dx_2$.
The induced map on homology satisfies the following
\begin{enumerate}
    \item It is independent of the choice of admissible product datum.
    \item If, in addition, $H_3\geq H_2*H_1$, it is super-commutative.
       \item It commutes with all continuation maps in $\cF(H)$.
    \end{enumerate}
 \end{tm}
 \textit{For the remainder of this section, unless indicated otherwise, all the Floer cohomology groups are those arising from contractible orbits. We do not indicate this further in the notation.}
\begin{rem}\label{remProdContr}
The reason we restrict our discussion of the pair of pants product to contractible periodic orbits is in the formulation of the Floer complex we chose in \S\ref{SecDefHam}. In that formulation the Floer complex is generated over $R$ by appropriate equivalence classes $[\gamma,A]$ where $\gamma$ is a periodic orbit and $A$ a path from a base loop in the component of $\gamma$.  Concatenating two such data $(\gamma_0,A_0),(\gamma_1,A_1)$ with a pair of pants having output on some periodic orbit $\gamma_2$ does not give rise to an appropriate path $A_2$ unless all the involved periodic orbits $\gamma_i$ are contractible. To obtain a well defined product involving non-contractible orbits additional choices need to be made. This should not be hard, but we do not pursue the details. An alternative approach which avoids this issue altogether is indicated in Remark \ref{remAltApp}. Note also that if the symplectic form is exact or even merely aspherical and atoroidal, there is no issue. Moreover, if one of the inputs is contractible, the pair of pants product is well defined without additional choices. Thus as  a result of the present subsection we do get the module structure of the full symplectic cohomology over the contractible part without any additional work.
\end{rem}
\begin{proof}[Proof of Theorem \ref{tmDisProd}]
Fix a regular pair of pants datum $P\in\cP(H_0,H_1,H_2)$. For $\tilde{\gamma_i}\in CF^*(H_i)$ such that
\[
i_{RS}(\tilde{\gamma_2})=i_{RS}(\tilde{\gamma_0})+i_{RS}(\tilde{\gamma_1}). \]
The product $*$ is defined by counting the Floer solutions associated with $p$. The Leibnitz rule is obtained by analyzing the boundary of the $1$-dimensional moduli spaces. For details see \cite[\S 2.3.5]{Abouzaid2013}. Note that while \cite{Abouzaid2013} concerns the cotangent bundle, once we fix a regular product datum, the analysis of the moduli spaces is exactly the same.

Behavior of the valuation under $*$ follows by Lemma \ref{lmTopGeoEnEst}. Namely, by monotonicity of the product data, any solution must have non-negative energy.

Given two choices of admissible product data, Lemma \ref{lmHAdmProdDat} allows to construct a dissipative homotopy. We can perturb while maintaining dissipativity to get a sufficiently generic homotopy inducing a chain homotopy between the appropriate complexes. Commutation with continuation maps follows in the same way from Lemma \ref{lmHAdmProdDat} and a standard gluing argument.

The claim about supercommutativity follows by pulling back the product datum $P$ by a biholomorphism of $S^2$ which fixes $\infty$ and commutes $0$ and $1$. For details see \cite[Lemma 2.3.24]{Abouzaid2013}.
\end{proof}

\begin{lm}\label{lmAcTruncProd}
The pair of pants product induces a map
\[
*:HF^*_{[a_1,b_1)}(H_1)\otimes HF_{[a_2,b_2)}^*(H_2)\to HF_{[\max\{a_1+b_2,a_2+b_1\},b_1+b_2)}^*(H_3),
\]
for all $\cF(H)$ admissible triples. Moreover, the product $*$ fits into a commutative diagram
\begin{equation}\label{eqDiagFiltProd}
\xymatrix{
HF^*_{[a_1,b_1)}(H_1)\otimes HF^*_{[a_2,b_2)}(H_2)\ar[d]\ar[r]& HF^*_{[\max\{a_1+b_2,a_2+b_1\},b_1+b_2)}(H_3)\ar[d] \\
HF^*_{[a'_1,b'_1)}(H_1)\hat\otimes HF^*_{[a'_2,b'_2)}(H_2) \ar[r] &HF^*_{[\max\{a'_1+b'_2,a'_2+b'_1\},b'_1+b'_2)}(H_3)
}
\end{equation}
whenever $a'_i>a_i$ and $b_i'>b_i$.
\end{lm}
\begin{proof}
Recall the identity for any $R$-modules $A,B,C,D$,
\[
A/B\otimes C/D=A\otimes C/(B\otimes C+A\otimes D).
\]
The first part of the claim follows by definition of $CF^*_{[a,b)}$ and the estimate \eqref{eqValProdEst}. The second part is clear if one works with representatives.
\end{proof}

\begin{lm}

Suppose $H_1,H_2,H_3$ are a triple of lower semi-continuous exhaustion functions satisfying
\begin{equation}\label{eqProdTripSemCo}
H_3\geq H_1*H_2
\end{equation}
and $H_i\geq H$. Then there is a monotone sequence of admissible product triples $H_{ik}\in\cF(H),i=1,2,3$ such that $H_{ik}$ converges pointwise as $k\to\infty$ to $H_i$.
\end{lm}
\begin{proof}
Pick constants $a_1,a_2,a_3<1$ such that $a_1+a_2<a_3$. According to Lemma \ref{tmLoSCAppCo}, we can find sequences $H_{ik}, i=1,2,3$ which are monotone in $k$ converging pointwise to $H_i$ and coinciding with $a_iH$ outside of a compact set. From \eqref{eqProdTripSemCo} it follows that for each $k$ there exists an index $i_k$ such that $H_{3,i_k}>H_{1,k}*H_{2,k}$. The sequence $H_{1,k},H_{2,k},H_{3,i_k}$ is as required.
\end{proof}

\begin{lm}\label{lmfiltPaOpLSC}
The pair of pants product for Hamiltonians in $\cF(H)$ induces a map
\[
*_H: HF^*_{[a_1,b_1)}(H_1)\otimes HF_{[a_2,b_2)}^*(H_2)\to HF_{[\max\{a_1+b_2,a_2+b_1\},b_1+b_2)}^*(H_3),
\]
for all triples $(H_1,H_2,H_3)\in\cH_{s.c.}^3$ that satisfy $H_i\geq H$. Moreover the product $*_H$ fits into a commutative diagram as in \eqref{eqDiagFiltProd}.
\end{lm}
\begin{proof}
Pick a monotone sequence $(H_{0k},H_{1k},H_{2k})$ of $H$-admissible triples converging to $(H_0,H_1,H_2)$. As in \eqref{eqDfFiltLSC},
\[
HF^*_{[a,b)}(H_i):=\varinjlim_kHF^*_{[a,b)}(H_{ik}).
\]
Since tensor product commutes with direct limits we get an induced product as in the statement of the lemma. Moreover, since the pair of pants product commutes with all continuation maps in $\cF(H)$, the product $*_H$ is independent of the choice of approximating sequence.
\end{proof}

\begin{lm}\label{dfProdConsApp}
Fix a superdissipative $H$. The pair of pants product on $\cF(H)$ induces a canonical product
\[
*_H:\overline{HF}^*(H_0)\hat{\otimes}\overline{HF}^*(H_1)\to\overline{HF}^*(H_2),
\]
for all product triples $(H_0,H_1,H_2)\in\cH_{s.c.}^3$. The operation $*_H$ commutes with all continuation maps and is supercommutative.
\end{lm}
\begin{proof}
For $i=0,1$, let $\gamma_i\in\overline{HF}^*(H_i)$. By Lemma \ref{lmchTruncIsom}, $\overline{HF}^*(H_i)$ is the reduced cohomology of an appropriate chain complex $\widehat{CF}^*(H_i)$ well defined up to filtered quasi-isomorphism. Pick such chain complexes for $H_0$ and $H_1$. Let $\tilde{\gamma}_0,\tilde{\gamma}_1$ be representatives of $\gamma_0,\gamma_1$ respectively. We construct an element $\gamma_2=[\tilde{\gamma}_1]*_H[\tilde{\gamma_2}]\in\overline{HF}^*(H_2)$ as follows. By \eqref{eqDfFiltLSC1}, to give an element in $\overline{HF}^*(H_2)$ it suffices to fix some $c$ and give for each $a<c<b$  an element $\gamma^{ab}_2\in HF^*_{[a,b)}(H_2)$ so that the $\gamma^{ab}_2$ agree under the natural maps
\begin{equation}\label{eqNatMap}
HF^*_{[a,b)}(H_2)\to HF^*_{[a',b')}(H_2),
\end{equation}
defined whenever $a<a'$ and $b<b'$. For some $\epsilon>0$ write $b_i=\val(\tilde{\gamma}_i)+\epsilon$ for $i=0,1$. Fix some $b>b_0+b_1$ and for any $a<b$ let $a_0=a-b_1$ and $a_1=a-b_0$.
Then by applying the operation $*_H$ of Lemma \ref{lmfiltPaOpLSC} to the classes of $\tilde{\gamma}_i$ in $HF^*_{[a_i,b_i)}(H_i)$
 we obtain an element $\gamma_2^{ab}\in HF^*_{[a,b)}$. Moreover, $\gamma_2^{ab}$ agrees with $\gamma_2^{a'b'}$ under the natural maps \eqref{eqNatMap}. We thus obtain an element $\gamma_2\in \overline{HF}^*(H_2)$ which is well defined after fixing representatives $\tilde{\gamma}_i$.
We need to verify that $\gamma_2$ is independent of the choice of representatives. For this it suffices to show that if $\tilde{\gamma}_i$ is in the closure of the image of the boundary for either $i=0$ or $i=1$, then $\gamma_2=0$. This amounts to showing that for each $a$ there exists a $b$ such that $\gamma_2^{ab}=0$.
For definiteness assume $[\tilde{\gamma}_0]=0\in\overline{HF}^*(H_0)$.
By \eqref{eqDfFiltLSC1} we need to show that for each $a$ there is a $b$ such that
\[
\gamma_2^{ab}=0\in HF_{[a,b)}(H_2).
\]
For this it suffices to show that there is a $b$  so that we can find numbers $a_i,b_i$ for $i=0,1$ such that
\[
\max\{a_0+b_1,a_1+b_0\}\leq a<b_0+b_1\leq b, \quad b_i>\val(\tilde{\gamma}_i),
\]
and such that $[\tilde{\gamma}_0]=0\in HF^*_{[a_0,b_0)}(H_0)$. We choose $b_1=\val(\tilde{\gamma}_1)+\epsilon$ and $a_0=a-b_1$. Since $\tilde{\gamma}_0$ vanishes in reduced cohomology, there exists a $b_0=b_0(a_0)>\val(\tilde{\gamma}_1)$ such that it vanishes in $HF^*_{[a_0,b_0)}(H_0)$. Pick $a_1=a-b_0$. Then all the requirements are satisfied. It is clear that changing the underlying complexes for $\overline{HF}^*(H_i)$ up to filtered quasi-isomorphism does not affect the definition of $*$.
\end{proof}
\subsubsection{Independence of the choice of $H$ at infinity}

\begin{lm}\label{lmCondSupAdmProd}
Let $F_0\leq F_1$ be superdissipative. For $i=0,1$ let $(H^i_0,H^i_1,H^i_2)\in\cF(F_i)$, and suppose $H_j^0\leq H_j^1$ for $j=0,1,2$. Then the operation $*$ commutes with the continuation maps $H_j^0\to H_j^1$. In particular, the definition of the pair of products is independent at the homology level on the choice of $H$ in the approximating scheme.
\end{lm}
\begin{proof}
Fix dissipative homotopies $H^0_j\to H^1_j$. Gluing either the first two homotopies to the input of the pair of pants in $\cF(F_0)$ or the last one to the pair of pants in $\cF(F_1)$ gives rise to two pairs of pants from $H^0_0,H^0_1$ to $H^1_2$ with $1$-forms of the form $\mathfrak{H}_i=G_i\alpha_\Sigma$ such that $d\mathfrak{H}_i\wedge dh_\Sigma\geq 0$. We need a family $(G_s,J_s)$ such that $(\mathfrak{H}_s:=G_s\alpha_\Sigma,J_s)$ form a dissipative interpolating family still satisfying
\[
d\mathfrak{H}_s\wedge dh_\Sigma\geq 0.
\]
The proof of existence of such a $(G_s,H_s)$ is exactly as in Lemma \ref{lmAdmConn}.
\end{proof}

\subsubsection{Associativity}
\begin{lm}\label{lmAssociativity}

Let $H^i$ for $i=1,2, 3$ and $H^{1,2},H^{2,3}$ be elements of $\cH_{s.c.}$. Suppose $(H^1,H^2,H^{1,2})$ and $(H^2,H^3,H^{2,3})$ are product triples. Let $H^4$ be such that
\[
H^4\geq 2\max_{t\in\R/\Z}\{H^{1,2}_t,H^{2,3}_t\}.
\]
Then the maps
\[
\overline{HF}^*(H^1)\otimes\overline{HF}^*(H^2)\otimes\overline{HF}^*(H^3)\to\overline{HF}^*(H^4)
\]
coming from the two compositions in Figure \ref{figAssociativity} coincide.
\end{lm}

\begin{figure}
\includegraphics[scale=0.6, trim=0 350 0 100,clip]{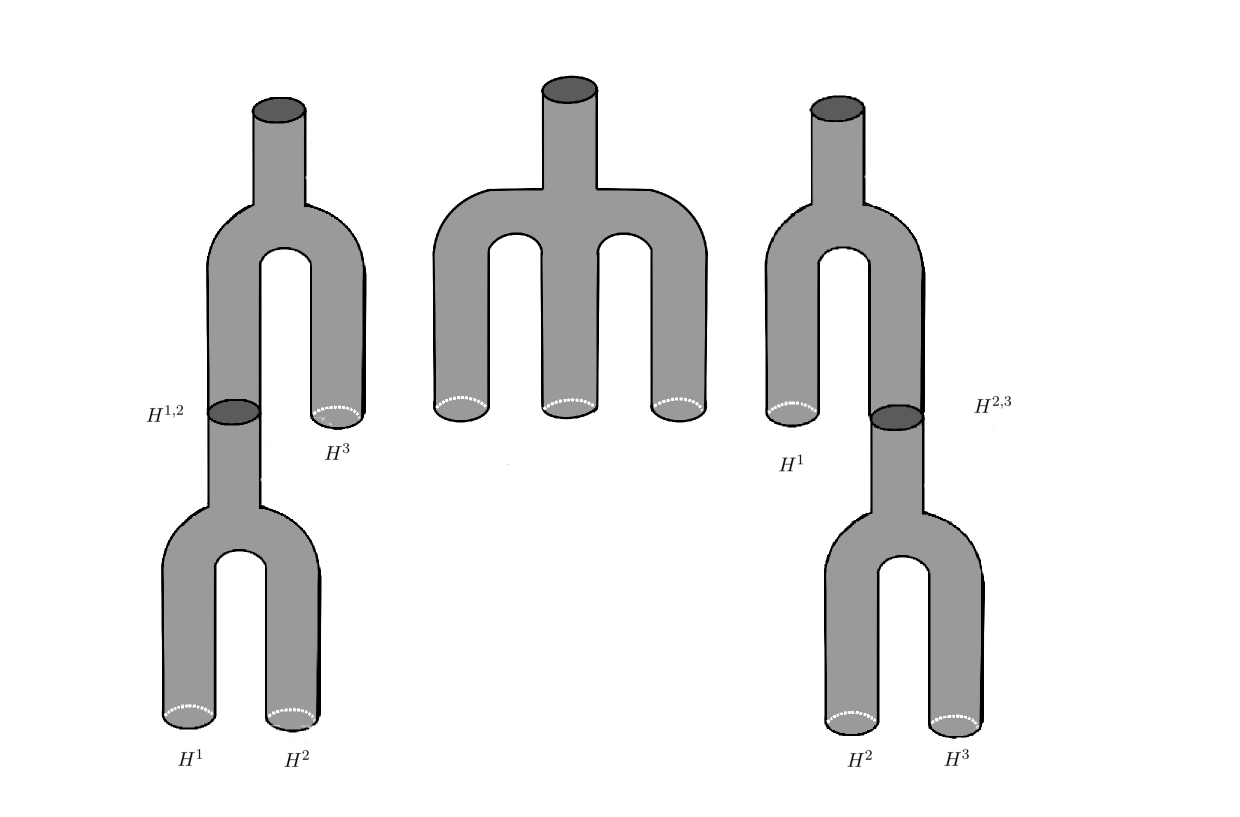}
\caption{Homotopy for associativity}
\label{figAssociativity}
\end{figure}
\begin{proof}
It suffices to prove the claim under the assumption that  $H^i,H^{i,j}\in\cF(H)$ for some super dissipative $H$ since we can replace all the involved Hamiltonians by approximating sequences in $\cF(H)$. Moreover, we may assume all inequalities are strict. The assumption implies there is a time independent $H$ such that $H^4> H> 2\max\{H^{1,2},H^{2,3}\}$. Thus if we prove associativity for the case where all the functions are positive multiples of a single function, the general claim will follow by naturality of the pair of pants product with respect to continuation maps. For this case the proof is standard in the literature (e.g., \cite{abouzaidSeidel2010,Abouzaid2013}) but we spell out the details.

Consider $H^1=H^2=H^3=H$, $H^{1,2}=H^{2,3}=2H$ and $H^4=4H$. Let $f^{1,2}, f^{2,1}:\Sigma\to\R$ be functions such that $d(f^{i,j}\alpha_\Sigma)\geq 0,$ $df^{i,j}$ is compactly supported, and $f^{i,j}$   equals $i$ at the input $z_0$, $j$ at the input $z_1$ and $4$ at the output.  We consider the $1$-forms $\mathfrak{H}^{1,1,2}:=H\alpha_\Sigma$ and $\mathfrak{H}^{i,j,4}:=f^{i,j}H\alpha$. The two possible compositions correspond to the gluing of $\mathfrak{H}^{1,1,2}$ to either $\mathfrak{H}^{2,1,4}$ at the first input, or to $\mathfrak{H}^{1,2,4}$ at the second input.
We must show that there exists a homotopy between the two compositions. We write these glued $1$-forms as $H\alpha_0,H\alpha_1$ where $\alpha_i$ are $1$-forms on $S^2\setminus \{0,1,z',\infty\}$ where $z'$ is a point $z_0$ near $0$ for the first composition and $z_1$ near $1$ for the second composition. Note that $d\alpha_i\geq 0$. We pick a smooth path $z_{\sigma\in[0,1]}$  in the Riemann sphere. We show that we can lift this to a path $\alpha_\sigma$ of $1$-forms satisfying $d\alpha_\sigma\geq0$ and connecting $\alpha_0$ to $\alpha_1$. Smooth four times punctured spheres are diffeomorphic by a diffeomorphism which preserves the cylindrical ends. Thus the claim reduces to finding such a homotopy on a fixed surface. But the condition $d\alpha\geq 0$ is convex. Thus we can find a path $\alpha_\sigma$ as required. The family $\mathfrak{H}_\sigma:=H\alpha_\sigma$ gives the required homotopy.
\end{proof}

\subsection{The PSS map}
\begin{tm}
Let $M$ be geometrically bounded. Then the small quantum product on $H^*(M;\K)$ is well defined.
\end{tm}
\begin{proof}
We take as our model of $H^*(M)$ the homology of the Morse complex $CM^*$ arising from considering the positive gradient flow of some proper exhaustion function $f:M\to\R$ with non-degenerate critical points, together with a geometrically bounded $J$ such that the pair $(f,g_J)$ is Morse-Smale. For this to compute cohomology (and, indeed, for the Morse differential to be well defined) we take $CM^*:=\K^{crit(f)}$. Indeed, $CM^*$ thus defined is the dual of $CM_*$ which is consists of finite formal sums of critical points with the differential defined by counting negative gradient lines. Since $g$ is proper and bounded below, the subcomplexes $CM_a^*\subset CM^*$ generated by critical points in the sublevel set $M_a:=f^{-1}(-\infty,a)$ compute the singular homology of $M_a$ by standard Morse theory \cite{Milnor}. By taking colimits it follows that $CM_*$ computes singular homology of $M$, and therefore, $CM^*$ computes singular cohomology. The small quantum product is defined by counting the $J$-holomorphic spheres with three marked points intersecting the unstable manifolds of some input critical points $p,q$ and the stable manifold of an output critical point $r$. Since $f$ is proper and bounded below, the stable manifold of any critical point is pre-compact. Thus by Theorem \ref{tmDiamEst} all the spheres are contained in a priori compact sets. The fact that the operation thus defined is indeed an associative product is now standard \cite[\S12.2]{MS2}\cite{PSS}.% To show independence of choices we consider accelerated gradient flow  for a time dependent Morse function together with a uniformly i-bounded path of almost complex structures and define .

\end{proof}
\begin{rem}
An alternative way to think of the construction of the small quantum product is to observe that cohomology of a non-compact manifold is Poincare dual to Borel-Moore homology. That is, a homology where one allows locally finite sums of singular chains. Given a triple $\gamma_1,\gamma_2,\gamma_3$ of Borel-Moore homology classes, the coefficient of $\gamma_3$ in the quantum product $\gamma_1*_{QH}\gamma_2$ is the three pointed Gromov-Witten associated with the triple $\gamma_1,\gamma_2,\gamma_3^*$ where $\gamma^*_3\in H_*(M)$ be the Poincare dual of $\gamma_3$. $\gamma_3^*$ is a cycle in ordinary homology and thus a finite chain. Therefore, by geometric boundedness, the number of $J$-holomorphic spheres intersecting $\gamma_3^*$ and representing a given homology class is finite.
\end{rem}
\begin{tm}\label{tmPSS}
For any $H\in\cH_{s.c.}$ there is a natural map
\[
f^{PSS}_H:H^*(M;\K)\to \overline{HF}^*(H).
\]
Denote by $*$ the product in Floer cohomology and by $*_{QH}$ the small quantum product. Then $f^{PSS}$ satisfies for any product triple $H_0,H_1,H_2$ and for any pair of classes $a,b\in H^*(M;\K)$,
\[
f^{PSS}_{H_0}(a)*f^{PSS}_{H_1}(b)=f^{PSS}_{H_2}(a*_{QH}b).
\]
In addition, for any $x\in \overline{HF}(H_1)$,
\[
f^{PSS}_{H_0}(1)*x=f_{H_1,H_2}(x),
\]
where
\[
f_{H_1,H_2}:\overline{HF}^*(H_1)\to\overline{HF}^*(H_2)
\]
is the natural continuation map.
\end{tm}
\begin{proof}
In the compact case for smooth non-degenerate Hamiltonians this is \cite{PSS}. In the non-compact case, we achieve $C^0$ estimates for smooth dissipative non-degenerate Hamiltonians by considering appropriate dissipative data on the plane. Namely, we pick a geometrically bounded $J$ and monotone homotopy $H_z$ going from $0$ for $z$ in a neighborhood of the origin and to $H$ for $z$ near $\infty$, and such that the associated Gromov metric $g_{J_H}$ is i-bounded. This is done as in Lemma \ref{lmAdmConn}. Alternatively, we can restrict the direct definition to the superdissipative case, and just take $H_z=f(|z|)H$ for some monotone increasing function. The definition for arbitrary $H$ is by an approximation scheme as in the definition of the pair of pants product. Write $dt:=d\arg z$. Then the Floer datum $(J,H_zdt)$ on the complex plane is dissipative. Moreover, it is monotone, so by Lemma \ref{lmTopGeoEnEst} and Theorem \ref{tmFloSolDiamEst} the solutions emanating from any critical point of $f$ to any critical point of $\cA_H$ are confined to an a priori compact set. This reduces the claims to the compact case.
\end{proof}
%\subsection{Symplectic cohomology}
\subsection{Proof of Theorems~\ref{tmFloerProduct} and \ref{tmLocFlHoProp}}\label{SecSHuniv}
Let $\cH\subset\cH_{s.c.}$ be a subset consisting of time-independent Hamiltonians such that for any $H_1,H_2\in\cH$ we have that $2\max\{H_1,H_2\}\in\cH$. We call $\cH$ a \textit{monoidal indexing set}. For each monoidal  indexing set $\cH$ we define a group
\[
SH^*(M;\cH):=\varinjlim_{H\in\cH}\overline{HF}^*(M).
\]
We denote by $\cH_{sm}$ the monoidal indexing set consisting of all smooth functions which are proper and bounded from below and define
\[
SH^*_{univ}(M):=SH^*(M;\cH_{sm}).
\]
We now prove Theorem~\ref{tmFloerProduct} from the introduction which states that $SH^*(M;\cH)$ is a unital algebra over $QH^*(M;\K)$.
\begin{proof}[Proof of Theorem~\ref{tmFloerProduct}]
\begin{enumerate}[wide, labelwidth=!, labelindent=10pt]
\item
Given $\gamma_0,\gamma_1\in SH^*(M;\cH)$ we can find $H_0,H_1\in\cH$ such that $\gamma_i$ lifts to an element still denoted by $\gamma_i\in \overline{HF}^*(H_i)$ for $i=0,1$. Since $\cH$ is a monoidal indexing set we can find an $H_2\in\cH$ such that $(H_0,H_1,H_2)$ form a product triple. Pick a super dissipative Hamiltonian $H\leq H_i,i=0,1,2$ and let $\gamma_2:=\gamma_0*_H\gamma_1\in\overline{HF}^*(H_2)$ using the induced product $*_H$ from Lemma \ref{dfProdConsApp}. By Lemma \ref{lmCondSupAdmProd}, $\gamma_2$ is independent of the choice of $H$. We define $\gamma_0*\gamma_1\in SH^*(M;\cH)$ to be the image of $\gamma_2$ under the natural map $\overline{HF}^*(H_2)\to SH^*(M;\cH)$. Since $*_H$ commutes with all continuation maps, $\gamma_1*\gamma_2$ is independent of the choice of product triple $(H_0,H_1,H_2)\in\cH^3$. Associativity and supercommutativity hold up to continuation maps in $\cH$ by Lemmas \ref{lmAssociativity} and \ref{dfProdConsApp}. Indeed, if $H_1,H_2,H_3\in\cH$ so is $4\max\{H_1,H_2,H_3\}$.
\item
This is an immediate consequence of Theorem \ref{tmPSS}.
\item
This is an immediate consequence of the naturality of the pair of pants product with respect to the continuation maps.
\end{enumerate}
\end{proof}

We next turn to the proof of Theorem \ref{tmLocFlHoProp} concerning local symplectic cohomology, but first we recall some definitions. Let $K\subset M$ be a compact set. Let
\[
H_K(x):=\begin{cases} 0, & x\in K,\\ \infty, & x\in M\setminus K.\end{cases}
\]
The \textit{local symplectic cohomology} at $K$ is defined by
\[
SH^*(M|K;\K):=\overline{HF}^*(H_K;\K).
\]

\begin{proof}[Proof of Theorem \ref{tmLocFlHoProp}]
\begin{enumerate}[wide, labelwidth=!, labelindent=0pt]
\item We have
    \[
    K_1\subset K_2\iff H_{K_2}\leq H_{K_1}.
    \]
    Thus there is a continuation map
    \[
    SH^*(M|K_2)=\overline{HF}^*(H_{K_2})\to SH^*(M|K_1)=\overline{HF}^*(H_{K_1})
    \]
\item
    This is the symplectic invariance of the construction of $\overline{HF}^*$.
\item We have that $\cH_K:=\{H_K\}$ forms a monoidal indexing set. So $SH^*(\cH_K)=\overline{HF}^*(H_K)$, and the claim follows from Theorem \ref{tmFloerProduct}.
\item This is an immediate consequence of Theorem~\ref{tmFloerProduct} and the functoriality of the continuation maps.
\item We have $H-\sup_KH\leq H_K$. On the other hand the map corresponding to $H\to H+c$ is a conformal isomorphism decreasing valuation by $c$.
\end{enumerate}

\end{proof}

\subsection{Symplectic cohomology as a topological vector space}\label{SecSHTopCof}
There is more than one natural way to put a topology on $SH^*(\cH)$ depending on the purpose one has in mind. In the rest of the paper we shall consider the \textit{final topology} on $SH^*(M;\cH)$. That is, the strongest topology for which all the continuation maps $\overline{HF}^*(H)\to SH^*(M;\cH)$ for $H\in\cH$ are continuous. Note that this topology is not necessarily Hausdorff. We also do not address the continuity of the pair of pants product. For applications later in this paper we will only need the following lemma.
\begin{lm}\label{lmSHHyoSHKCont}
Let $\cH$ be a monoidal indexing set consisting of continuous Hamiltonians. Let $K\subset M$ be a compact set. Then the natural map
\[
SH^*(\cH)\to SH^*(M|K)
\]
is continuous.
\end{lm}
\begin{proof}
By Theorem \ref{tmLocFlHoProp}\ref{tmLocFlHoPropPe} for any continuous Hamiltonian $H$, the continuation map $HF^*(H)\to SH^*(M|K)=HF^*(H_K)$ is continuous. Continuity of the induced map follows by definition of the final topology.
\end{proof}
\begin{rem}
Since the spaces $SH^*(M|K)$ are all Banach spaces, Lemma \ref{lmSHHyoSHKCont} will still hold if one considers topologies on $SH^*(\cH)$ which are weaker than the final topology. We do not pursue this further however.
\end{rem}
We illustrate the various notions of symplectic cohomology by considering the case $M=\R\times\R/\Z$. We will compare symplectic cohomologies for three different monoidal indexing systems. \textit{Since $M$ is a Liouville domain we no longer restrict the discussion of the pair of pants product to contractible orbits.}
\begin{ex}
Consider the monoidal indexing set $\cL$ consisting of Hamiltonians which outside of a compact set are of the form
\[
H(s,t)=a|s|+b.
\]
According to a theorem by Viterbo \cite{Viterbo2018} equating symplectic cohomology of a cotangent with loopspace homology of the underlying manifold, we have for any coefficient field $R$ that
\begin{equation}
SH^*(M;\cL)=R[x,x^{-1},\partial_x]/\partial_x^2.
\end{equation}
Thus, \textit{$SH^*(M;\cL)$ is the exterior algebra of polyvector fields on $R\setminus\{0\}$.} By the Kunneth formula \cite{Oancea06}, or by Viterbo's theorem again, the same holds for $M=T^*\T^n$.
\end{ex}

\begin{ex}
Let now $K\subset M=[-a,a]\times\R/\Z$. To keep track of actions choose the primitive $sdt$ of the standard symplectic form. With this choice one shows that $SH^*(M|K)$ is obtained from $SH^*(M;\cL)$ by completing with respect to the valuation $\val(x^i):=|i|a$ and $\val(\partial_x):=0$. When the underlying ring $R$ is trivially valued, this completion is of no effect. However, when working over the universal Novikov ring we obtain for example that $SH^0(M|K)$ consists of all infinite Laurent series
\[
\sum_{i=-\infty}^{\infty}b_ix^i,\quad b_i\in\Lambda_{0,nov},
\]
satisfying
\[
\lim\val(b_i)+|i|a\to-\infty.
\]
This is the same as the set of analytic functions on the rigid analytic torus $\Lambda^*$ which converge on the subtorus $\{z\in\Lambda^*|\val(z)\in[-a,a]\}$. A reference for rigid analytic geometry is \cite{Bosch}. Closer to home, \cite{Tu2014,Abouzaid2014} provide a closely related point of view from the vantage point of Lagrangian Floer homology of the $\R/\Z$ fibers.
\end{ex}

\begin{ex}\label{exSUniv}
Finally we study $SH^*_{univ}(\R\times S^1)$. We claim that \textit{over the universal Novikov ring $SH^0_{univ}(\R\times S^1)$ consists of formal Laurent series $\sum b_ix^i$ that are rapidly decreasing in the sense that there exists a superlinear convex function $g$ such that
$\val(b_n)=-g(n).$}

To see this observe that $SH^*_{univ}(M)$ is computable by  direct limit of $HF^*(H)$ over functions $H$ which outside of a compact set are of the form $H(s,t)=h(|s|)$ for $h:\R_+\to\R_+$ a convex function such that $h'(t)$ is unbounded as $t\to\infty$\footnote{Such a function is not necessarily dissipative. However, we may still talk about its Floer cohomology by approximating by linear Hamiltonians.}. Note that for any such $H$ there is a natural map
\[
SH^*(M;\cL)\to HF^*(H),
\]
by the universal property of the direct limit and the fact that for any $H'\in\cL$ we have $H'\leq H$ outside of a compact set. Each monomial $x^i$ maps to a class associated with a unique periodic orbit $\gamma^i$ of $H$. It is not hard to show that $HF^*(H)$ is in fact the completion of the algebra $SH^*(M;\cL)=R[x,x^{-1},\partial_x]$ with respect to the norm $\val(x_i)=\cA_H(\gamma_i)$. This is computed as follows. Let $s_i\in\R$ be the unique real number such that $\partial_sH(s,t)=i$. Then
\[
\val(x^i)=s_ih'(s_i)-h(|s_i|)=is_i-h(|s_i|).
\]
Note that writing $f=h'$ we have $s_i=f^{-1}(i)$ and that the right hand side the last equation is exactly $g(i):=\int^i_0f^{-1}(t)dt$. Since $f^{-1}$ is monotone and unbounded this means $g_i$ is convex and superlinear. From this it is not hard to deduce the claim.
\end{ex}
%A fourth invariant we didn't discuss, but is relevant here, is
%\[
%SH^*_c(M):=\varprojlim_KH^*(M|K)
%\]
%where the inverse system arises by considering the restriction map from $K_1$ to $K_2$ whenever there is an inclusion $K_2\subset K_1$. By the logic of the previous examples this is the same as entire functions on $\Lambda^*$, or, sums $\sum b_ix^i$ for which
%\[
%\lim\frac{\val(b_n)}{n}=-\infty.
%\] 
\section{Computations and applications}\label{Sec7}
\subsection{Liouville domains}\label{subSecLiouville}
Let $(\Sigma,\alpha)$ be a contact manifold with contact form $\alpha$. Let $(U,\omega=d\lambda)$ be a compact exact symplectic manifold with $\Sigma$ as boundary such that $\alpha=\lambda|_{\partial U=\Sigma}$ and such that the Liouville field $Z$, which is defined by $\iota_Z\omega =\lambda$,  points outward at the boundary. Let $\hat{U}$ be the completion of $U$ by attaching the cone $\Sigma\times\R_{\geq0}$ with the symplectic form $\omega_{\alpha}=e^r(d\alpha+dr\wedge\alpha )$. The vector field $Z$ extends to $\hat{U}$ and is given by $Z=\frac{\partial}{\partial r}$.  

Denoting by $\phi_t$ the time $t$ flow of $Z$, the skeleton of $U$ relative to $\lambda$ is defined by 
\[
\Skel(U,Z)=\cap_{t>0}\phi_{t}(U). 
\]
The map $\Sigma\times\R\to \hat{U}$ defined by $(x,r)\mapsto\phi_r(x)$ is a symplectic embedding of the symplectization of $\Sigma$ whose image is  $\hat{U}\setminus\Skel(U,\lambda)$.  A reference for these basic definitions and claims is \cite{CielEli}. In particular, the function $\phi_r(x)\mapsto e^r$ is defined and smooth on $\hat{U}\setminus\Skel(U,Z)$. Moreover, it extends to a continuous function of $\hat{U}$, still denoted by $e^r$ by defining $e^r(p)=0$ for $p\in \Skel(U,Z)$. Denote by $\mathcal{L}$ the set of Hamiltonians that outside of a compact set containing the skeleton are of the form $ae^r+b$ for $a>0$ and $b\in\R$. We refer to these Hamiltonians as being \emph{linear at infinity}. Similarly, Hamiltonians which outside of a compact set are  of the form $h(e^r)$ for $h$ convex are referred to as \emph{convex at infinity}. Let $J$ be of contact type. That is $J$ is an $\omega$-compatible translation invariant almost complex structure $J$ satisfying $JR=\partial r$ for $R$ the Reeb flow. As in \cite{Viterbo99}, define $SH^*_{Viterbo}(U)=\varinjlim_{H\in\mathcal{L}}HF^*(H,J)$.
\begin{lm}
We have that $SH^*_{Viterbo}(U)=SH^*(\hat{U};\mathcal{L}).$
\end{lm}
\begin{proof}
 By Example~\ref{wxAdmWT}, when paired with a contact type $J$, the elements of $\mathcal{L}$ are i-bounded. Any $H\in\mathcal{L}$ with slope at infinity not in the period spectrum of $\Sigma$ is dissipative by Example~\ref{ConvEndDiss}. It follows that $\overline{HF}^*(H)=HF^*(H)$. So the directed systems computing each side coincide.
 \end{proof}
 In particular, we have a natural map
\[
f:SH^*_{Viterbo}(U;R)=SH^*(\hat{U};\mathcal{L},R)\to SH^*_{univ}(\hat{U};R).
\]
\begin{tm}\label{VitUnivEmb}
The map $f$ is an isomorphism.
\end{tm}
\begin{rem}\label{rmTrivNonTriv}
Note that the proof below of Theorem \ref{VitUnivEmb} relies crucially on the fact that for Hamiltonians which are convex at infinity the action spectrum is bounded from, below rendering the topology of $HF^*(H)$ discrete. This fails when working over a non-trivially valued field. To see what sets trivially valued fields apart, consider the following. Given two Hamiltonians $H_1\leq H_2$ such that $H_i=h_i(e^r)$, the continuation map
\[
f_{12}:HF^*(H_1;R)\to HF^*(H_2;R)
\]
can be shown two be an isomorphism of vector spaces, and thus, since the topology is discrete, of topological vector spaces. However the inverse of $f_{12}$ will generally not be bounded. Thus when working over $\Lambda_R$, the map $f_{12}$ will no longer be homeomorphism.
\end{rem}

\begin{proof}[Proof of Theorem \ref{VitUnivEmb}]
We consider the set $\mathcal{C}$ of smooth Hamiltonians $H$ for which there is a compact  $K=\{e^r<\epsilon\}$ for some $\epsilon>0$ so that $H$ is $C^2$ small and negative on $K$ and of the form  $H=h(e^r)$ outside of $K$. The action of any $1$-periodic orbit of such a Hamiltonian is positive. 
The set of Hamiltonians $\mathcal{C}$ is cofinal in the set of all smooth Hamiltonians with respect to the order relation $\preceq$ defined as in the introduction by 
\[
H_1\leq H_2\iff \forall x\in M, H_2(x)-H_1(x)\geq C>-\infty.
\]
Pick a sequence $F_i\in\mathcal{C}$ given outside of a compact set by  $F_i=h_i(e^r)$ so that the sequence $F_i$ converges to $H$ on compact subsets of $M$ and such that near infinity $h_i$ is linear of slope not in the period spectrum of $\alpha$. The action of a periodic orbit is given by the right hand side of \eqref{eqAcnFor} which in this case specializes for a non-trivial periodic orbit $\gamma$ of $F_i$ occurring at some level set  $e^r=t$  to 
\[
\mathcal{A}_{F_i}(\gamma)=th'_i(t)-h_i(t),
\]
which is positive for $h_i$ convex. Positivity also holds for the trivial periodic orbits, since they occur inside  $U$ where $F_i<0$. We thus have that $HF^*_{[k,\infty)}(F_i)=HF^*(F_i)$ for all $k\leq 0$. A similar statement holds for $H$. From this we deduce, first, that $\overline{HF}^*(H)=HF^*(H)$, and, second, that  $HF^*(H)=\varinjlim_iHF^*(F_i).$

The set $F_i$ is cofinal in $\mathcal{L}$ with respect  the order relation $\preceq$. Therefore, we obtain an isomorphism of $R$-modules
\[
HF^*(H)=\varinjlim_i{HF}^*(F_i)\to SH^*(\mathcal{L}_{reg};R).
\]
Moreover, given two convex functions $H_1\leq H_2$, the continuation maps from $H_1$ to $H_2$ will commute with the above isomorphisms since they are all defined via continuation maps between functions which are linear near infinity. The claim follows.
\end{proof}
We similarly prove
\begin{tm}
We have
\begin{gather}
SH^*(\hat{U}|\Skel(U,Z);R)=SH^*(\hat{U}|U;R)\notag\\
=SH^*_{Viterbo}(U;R)=SH^*_{univ}(\hat{U};R)\notag.
\end{gather}
\end{tm}
\begin{proof}
Consider a monotone sequence $H_i$ belonging to the set of convex Hamiltonians $\mathcal{C}$ defined in the proof of Theorem \ref{VitUnivEmb} so that 
\[
\lim_{x\to\infty}H_i(x)=\begin{cases}
					0,&\quad x\in U,\\
					\infty,&\quad x\in\hat{U}\setminus U.
				\end{cases}
\]
Then, by positivity of the action spectrum,
\[
SH^*(\hat{U}|U;R)=\varprojlim_k\varinjlim_iHF^*_{[k,\infty)}(H_i)=\varinjlim_iHF^*(H_i).
\]
The right hand side equals $SH^*_{Viterbo}(U;R)$ by the same argument as Theorem \ref{VitUnivEmb}. In a similar way $SH^*(\hat{U}|\Skel(U,Z);R)=SH^*_{Viterbo}(U;R)$ by considering a sequence $H_i\in\mathcal{C}$ such that
\[
\lim_{x\to\infty}H_i(x)=\begin{cases}
					0,&\quad x\in \Skel(U,Z),\\
					\infty,&\quad x\in\hat{U}\setminus \Skel(U,Z).
				\end{cases}
\]
Finally, The equality $SH^*_{Viterbo}(U;R)=SH^*_{univ}(\hat{U};R)$ is Theorem \ref{VitUnivEmb}. 
\end{proof}

\begin{proof}[Proof of Theorem \ref{tmSkellVitFunc}]
We consider the radial coordinate $t=e^r$ on $U$ which we may assume surjects onto $(0,1)$, with $\Skel(U)$ corresponding to $t=0$. We use the notation $U(t_0):=\{p\in U|t(p)\leq t_0\}$. We will consider a family of dissipative $S$-shaped Hamiltonians $H_{c,\epsilon}$ which are defined as follows. $H$ is equal to $0$ on $U(\epsilon)$, to $ct-c\epsilon$ on $U(1/2)\setminus U(\epsilon)$ and has small gradient and Hessian outside $U(1/2)$. Here it is understood that we perturb slightly to get a smooth Hamiltonian which is transversely non-degenerate on $U(1/2)$. By part \ref{mainTmA:itC} of Theorem \ref{mainTmA} the Hamiltonians $H_{c,\epsilon}$ are dissipative. We construct a monotone increasing sequence $c_i$ going to $\infty$ and a monotone decreasing sequence $\epsilon_i$ going to $0$ and such that the distance of $c_i$ to the period spectrum of $\partial U$ is more than $2\epsilon_i$. We take $\epsilon_i$ even smaller so that the energy required according to Theorem~\ref{tmFloSolDiamEst} for a Floer trajectory to meet both sides of $U(1)\setminus U(1/2)$ is more than $\epsilon_i c_i$. Observe now that by our assumption, the action functional on $M$ restricted to loops in $U$ which are contractible in $M$ coincides up to a constant with the action functional defined using the Liouville form. Moreover, the periodic orbits outside of $U(1/2)$ are constants with large value of $H$. Thus, the set of periodic orbits of $H_{c_i,\epsilon_i}$ having non-negative action are the constants inside $U(\epsilon_i)$ as well as the periodic orbits appearing as the slope goes from $0$ to $c_i$.  Their actions are all at most $c_i\delta_i$. Thus, the Floer trajectories connecting orbits of non-negative action all remain inside $U(1)$. So, $SH^{*,0}_{[0,\infty)}(M|\Skel(U);R)=SH^{*,0}_{Viterbo}(U;R)$.

It remains to show that the negative action periodic orbits form an acyclic complex. Consider an increasing $1$-parameter family of Hamiltonians $H_t=H_{c(t),\epsilon(t)}$ with $c(t)\to\infty$ as $t\to\infty$, and fix an action window $[a,0)$.  We cannot show that for an arbitrary $a$ there is a fixed $t$ such that $HF^*_{[a,0)}(H_t)=0$ since as we increase the slope, new negative periodic orbits keep appearing with action not far from $0$. However, we claim that for each $t_0$ there is a $t_1>t_0$ such that, denoting by $f_{t_0,t_1}$ the continuation map from $t_0$ to $t_1$, we have
\[
f_{t_0,t_1}(HF^*_{[a,0)}(H_t))=\{0\}.
\]

Indeed, let $\gamma$ be a periodic orbit of $\partial U$ with period $T\leq c(0)$. It will appear as a periodic orbit of $H_{c(t),\epsilon(t)}$ of action $\frac{T-c(t)}{2}-\epsilon(t)$. Consider the cohomology class $\alpha(t)=f_{0,t}(\alpha)$. By functoriality of the continuation maps, its action is a monotone decreasing function of $t$. Moreover, since for any $t$ the complex $CF^*_{[a,0)}(H_t)$ is finitely generated, there is a discrete set of points $\{t_i\in[0,\infty)\}$ such that on the interval $(t_i,t_{i+1})$, the cocycle $\alpha$ is represented by the action minimizing cycle $\sum\gamma^i_j$. Let $T^i_j$ be the period of $\gamma^i_j$ as a Reeb orbit of $\partial U$ and let $T^i$ be the maximal of these. Along the interval $(t_i,t_{i+1})$ the action of $\alpha(t)$ will be given by $\frac{T^i-c(t)}{2}-\epsilon(t)$. Since the action of $\alpha(t)$ is non-increasing, we must have that $T^i$ is nonincreasing. Since $c(t)\to\infty$ it follows that the negative periodic orbits in $U(1)$ eventually fall out of any action window under the continuation maps. The periodic orbits outside of $U(1)$ are constants with action going to negative infinity. This means all the periodic orbits which lie outside $U(1)$ are in the closure of the boundary operator in $SC^*(\{H_i\})$. Upon tensoring $SC^*(\{H_i\})$ with $\Lambda_R$ the same remains true.
\end{proof}

\begin{proof}[Proof of Theorem \ref{tmLiouSubNondisp}]
Let $K$ be a compactly supported displacing Hamiltonian for $\Skel(U)$. Then $K$ displaces an open neighborhood of $\Skel(U)$ which we may take to be $U$ itself. Let $F$ be a Hamiltonian which vanishes on a neighborhood of $U\cup \supp K$ and has small enough gradient and Hessian to be dissipative and have only critical points as periodic orbits.  Let  $H_i$ be a sequence of Hamiltonians which vanishes on $\Skel(U)$, increases on $U(\epsilon_i)$ for some $\epsilon_i\to 0$, and becomes a constant $C_i$ outside the $u(\epsilon_i)$ with $C_i\to\infty$. The sequence $H_i+F$ is monotone and converges to $H_{\Skel(U)}$ and so computes $SH^{*}(M|\Skel(U);\mathbb{K})$.

Recall the notation $H_1\#H_2:=H_1+H_2\circ\psi_{H_1}$ where $\psi_{H}$ is the time 1 Hamiltonian flow of $H$. We have $\psi_{H_1\#H_2}=\psi_{H_2}\circ\psi_{H_1}$. Observe that the sequence $(H_i+F)\# K$ computes $SH^{*}(M|\Skel(U);\mathbb{K})$ as a topological vector space. To see this, note there is a constant $C$ such that $|H_i\# K-H_i|<C$ so we can factor
\[
SC^*(\{H_i+F-C\})\to SC^*(\{(H_i+F)\# K\}\to SC^*(\{H_i+F+C\}),
\]
and vice versa.

Note that $F+K=F\#K$ is also a displacing Hamiltonian which after slightly perturbing we can take to be non-degenerate. Moreover, all the fixed points of $(H_i+F)\# K$ coincide with those of $F\#K$. By a standard argument\cite{Usher2010}, adding $H_i$ has the effect of shifting the action spectrum by $-C_i$. The action spectrum of $F\#K$ is bounded from above since all the positive action orbits lie in the compact set defined by $F=0$. Thus the whole action spectrum of $(H_i+F)\# K$ moves to negative infinity. So, $SH^*(M|\Skel(U))=0$. By Theorem \ref{tmSkellVitFunc}, this implies $SH^*_{Viterbo}(U)=0$.
\end{proof}

\begin{proof}[Proof of Theorem \ref{tmEpsLiouEmb}]
We consider the family of Hamiltonians $H_{c,\epsilon}$ as in the proof of Theorem~\ref{tmSkellVitFunc}. The periodic orbits of $U$ that are contractible in $U$ embed in an action preserving manner in $\mathcal{L}M$. We take $\delta>0$ such that any Floer trajectory of energy at most $\delta$ which meets $U(1/2)$ is contained in $U(1)$. The classes $[x,A]$ with $x$ a contractible periodic orbit in $U(1/2)$ and $A$ a path in $\mathcal{L}U(1)\subset\mathcal{L}(M)$ thus form a direct summand of $SC_{[0,\delta)}^*(\{H_{c_i,\epsilon_o}\})$. Moreover, the proof of Theorem~\ref{tmSkellVitFunc} shows that the differential applied to contractible periodic orbits in $U(1/2)$ coincides $\mod\delta$ with the differential computing $SH^*(\hat{U}|U)$. This proves the claim.
\end{proof}

\begin{proof}[Proof of Theorem~\ref{tmPosDis}]
Consider Hamiltonians as in the proof of Theorem~\ref{tmLiouSubNondisp} and denote by
\[
f:SC^*(\{H_i+F\})\to SC^*(\{(H_i+F)\#K\})
\]
a continuation map induced by an appropriate homotopy, by $g$ the continuation map in the other direction, and let $\mathfrak{H}$ be the chain homotopy operator between the identity and $g\circ f$. Fix an action value $a$. By taking $i_0$ large enough we have as in the proof of Theorem~\ref{tmLiouSubNondisp} $f_i$ vanishes $\mod a$ for all $i\geq i_0$. Therefore, starting our sequences at $i_0$, we have $\id=\mathfrak{H}\circ d+d\circ \mathfrak{H}\mod c$ for all $c\geq a$.  But $\mathfrak{H}$ can increase the valuation by at most the Hofer norm of $K$. It follows that if $\val\alpha=c>a$ and $d\alpha=0\mod c$ the largest possible window $[c,d)$ for which $\alpha\neq0\in HF^*_{[c,d)}(H_U)$ has $d-c< d(U)\leq\|H\|_{Hofer}$. So taking $a<0$ and $\delta$ as in Theorem~\ref{tmEpsLiouEmb} we get $d(e)>\delta$.
\end{proof}

\subsection{Mapping Tori}\label{subSecMappingTorus}

Let $(M,\omega)$ be a compact symplectic manifold and let $\psi:M\to M$ be a symplectomorphism. Denote by $M_\psi$ the mapping torus
 \[
M_{\psi}:=[0,1]\times M/(0,p)\simeq (1,\psi(p)).
\]
Let $\tilde{\omega} $ be the $2$-form on $M_{\psi_i}$ obtained by pulling back $\omega$ via projection to $M$, and let
\[
\tilde{M}_{\psi}:=\R\times M_{\psi},
\]
with the symplectic structure $\tilde{\omega}+ds\wedge dt$. Denote by $HF^*(M,\psi)$ the fixed point Floer homology of $\psi$ as introduced in \cite{DS94}. The closed $1$-form $dt$ induces a grading of the Floer homologies by integrating over periodic orbits.

Denote by $S:\tilde{M}_\psi\to\R$ the natural coordinate $(s,t,p)\mapsto s$. Let $f:\R\to\R$ be a proper convex function which is linear at infinity of slope at greater than $k$ for some integer $k$.  Let $J$ be an almost complex structure for which the map $\pi:\tilde{M}_\psi\to \R\times S^1$ defined by $(s,t,p)\mapsto (s,t)$ is $J$-holomorphic. Let $H=f\circ S$. The following Theorem is due to M. Abouzaid.
\begin{tm}
\begin{equation}\label{EqFixFloH}
\overline{HF}^{*,k}(H;\Lambda_R)=HF^{*,k}(H;\Lambda_R)\simeq HF^*(M,\psi^k;\Lambda_R),
\end{equation}
where the $\simeq$ denotes a conformal isomorphism.
\end{tm}
\begin{proof}
The Hamiltonian vector field of $S$ is $\partial_t$. So, the periodic orbits of $H=f\circ S$ are contained in fibers of $H$  for which $f'$ is an integer. The periodic orbits corresponding to an integer $k$ are the fixed points of $\psi^k$.  The periodic orbits corresponding to different values of $k$ have different homotopy classes. Thus the Floer differential only connects orbits within a fiber. The $(H,J)$-Floer trajectories in $\tilde{M}_\psi$ project under $\pi$ to maps satisfying the inhomogeneous Cauchy Riemann equation
\[
\partial_sS=\partial_tT-1,\qquad \partial_sT=\partial_tS.
\]
Thus the function $s+it\mapsto S+i(T-t)$ is holomorphic. By the maximum principle, $S$ must be constant. In particular, Floer trajectories connecting orbits within a fiber of $H$ must stay within that fiber. Also, we have $T=t$.

Recall the definition of the differential in fixed point Floer homology. Namely, for fixed points $x,y$ of $\psi^k$  it counts $J$-holomorphic strips asymptotic to $x,y$ satisfying the boundary condition $u(s,1)=\psi(u(s,0))$. Given such a $u$ we obtain a Floer cylinder $\tilde{u}$ in the mapping torus by $\tilde{u}(s,t):=(t,u(s,t))$. This sets up a bijection between the
Floer trajectories connecting periodic orbits in a fiber and fixed point holomorphic strips. The rightmost equality in \eqref{EqFixFloH} follows. For the other equality note that since $\psi^k$ has a finite number of fixed points, $CF^{*,k}(H;\Lambda_R)$ is finite dimensional and thus the differential has closed image.
\end{proof}
\begin{comment}
\begin{tm}[Cf.  \cite{Fabert10}]\label{tmMappFloFixed}
We have ${SH}^{*,k}_{univ}(\tilde{M}_\psi)\simeq HF^{*,k}(M,\psi^k).$
\end{tm}
\end{comment}
\begin{proof}[Proof of Theorem~\ref{tmMappFloFixed}]
Let $f$ be any proper convex function and let $H=f\circ S$. Consider a monotone sequence of convex functions $f_n$ which are linear of slope larger than $k$ near infinity and which converge to $f$. Write $H_n:=f_n\circ S$.

Since Floer trajectories remain in fibers of $S$, we have by the isomorphism~\eqref{EqFixFloH} that
\[
HF^{*,k}(H)=HF^{*,k}(H_n)=HF^*(M,\psi^k).
\]
By the same reasoning, given convex function $f_1\leq f_2$ and denoting $H_i=f_i\circ S,i=1,2$, we get that the natural map $HF^{*,k}(H_1)
\to HF^{*,k}(H_2)$ is just the identity under the above identification.

 It follows that
\[
SH^{*,k}_{univ}(\tilde{M}_\psi)=HF^{*,k}(M,\psi^k),
\]
Observe that $SH^{*}_{univ}$ amounts to completing the direct sum
\[
\oplus_{k\in\Z}SH^{*,k}_{univ}(\tilde{M}_\psi)
\]
by allowing certain infinite sums. The claim follows. %Moreover, for each such $H$ there is a continuous map $\overline{HF}^*(H)\to\hat{\oplus}_{k\in\Z }HF^*(M,\psi^k).$ Indeed, the norm on $\overline{HF}^{*,k}$ is the product of the norm on $HF^*(M,\psi^k)$ by $e^{c_k}$ where $c=ks_k-f(s_k)$ and $s_k$ satisfies $f'(s_k)=k$. By convexity of $f$, $c_k$ increases with $k$, so the map is indeed well defined and continuous. Moreover, it is a dense embedding. Thus there is an induced dense embedding $\widehat{SH}^*(\tilde{M}_{\psi})\to\hat{\oplus}_{k\in\Z }HF^*(M,\psi^k).$ It remains to verify surjectivity. For thus it suffices to show that for each Cauchy sequence in $\oplus_{k\in\Z }HF^*(M,\psi^k)$ there is an $H$ such that its pre-image in $\overline{HF}^*(H)$ is Cauchy. Equivalently, given a sequence in $\oplus_{k\in\Z }HF^*(M,\psi^k)$ with valuation going to $0$ we can find an $H$ such that the pre-image in $\overline{HF}^*(H)$ has valuation going to $0$. This is immediate from the description above.
\end{proof}

\begin{comment}
\begin{cy}\label{CyFixFlDist}
Let $M$ be a compact symplectic manifold and let $\psi_i:M\to M$ be a symplectomorphism for $i=0,1$. Let
\[
 \phi:\tilde{M}_{\psi_1}\to\tilde{M}_{\psi_2}.
 \]
 Then $\phi$ induces an isomorphism $HF^*(M,\psi_1)=HF^*(M,\psi_2)$.
\end{cy}
\end{comment}
\subsection{The Kunneth formula for split Hamiltonians}\label{SubsecKunnethHam}
For $i=1,2$, let $M_i$  be symplectic manifolds and let $(H_i,J_i)$ be dissipative Floer data on $M_i$.  Unless $(H_i,J_i)$ are strictly bounded, the data $(H_1\circ\pi_1+H_2\circ\pi_2,\pi_1^*J_1+\pi_2^*J_2)$ will not be i-bounded. In that case we replace $CF^*(H)$ via the telescope construction by a sequence of Hamiltonians which are strictly bounded and continue to denote this by $CF^*(H)$.  We have,
\begin{equation}\label{eqKunnethChain}
\widehat{CF}^*(H_1\circ\pi_1+H_2\circ\pi_2,\pi_1^*J_1+\pi_2^*J_2)=CF^*(H_1,J_1)\hat\otimes CF^*(H_2,J_2),
\end{equation}
where the hat denotes here and later the complete tensor product. This is defined by taking the Banach norm $\|\cdot\|$ on the tensor product $X\otimes Y$ to be defined  by
\[
\|z\|:=\inf\left\{\max_i\{\|x_i\|\|y_i\|\}:z=\sum x_i\otimes y_i\right\},\qquad z\in X\otimes Y.
\]
It is straightforward to verify using \eqref{NonArchNormDef} that this is indeed the norm induced by~\eqref{eqKunnethChain}. %The complete tensor product satisfies the universal property for jointly continuous bilinear maps from $X\times Y$ to non-archimedean locally convex topological vector spaces over $R$.
\begin{tm}\label{tmCohomologyKunneth}
We have a natural isometry of Banach spaces over $\Lambda_R$,
\[
\overline{HF}^*(H_1+H_2;\Lambda_R)=\overline{HF}^*(H_1;\Lambda_R)\hat\otimes \overline{HF}^*(H_2;\Lambda_R).
\]
%The map arises from a natural map $HF^*(H_1)\otimes HF^*(H_2)$ by completion.
\end{tm}

\begin{proof}
We follow the proof of the finite dimensional case from \cite{griffiths2014principles}. Isomorphism~\eqref{eqKunnethChain} induces a norm preserving map
\[
\overline{HF}^*(H_1)\hat{\otimes}\overline{HF}^*(H_2)\to \overline{H}^*(CF^*(H_1) \hat{\otimes}CF^*(H_2))=\overline{HF}^*(H_1+ H_2).
\]
We show that this map is  surjective. All the spaces considered here are countably generated. In particular every closed subspace has a closed complement \cite[Proposition 10.5]{Schneider02}. We thus decompose the chain complexes $CF^*(M_i;H_i)$ into a direct sum $C_i\oplus Z_i$ of chains and cycles and then further decompose $Z_i=K_i\oplus B_i$ where $B_i=\overline{\partial C_i}$.

Any cycle $\gamma\in CF^*(H_1)\hat{\otimes}CF^*(H_2)$ is up to the closure of the boundary an element of
\[
B_1\hat{\otimes} C_2\oplus K_1\hat{\otimes} C_2\oplus C_1\hat{\otimes} K_2\oplus K_1\hat{\otimes} K_2\oplus C_1\hat{\otimes} C_2.
\]
Now note the images of the spaces under $\partial$ are contained respectively in,
\[
B_1\hat{\otimes} B_2,K_1\hat{\otimes} B_2, B_1\hat{\otimes} K_2,0,B_1\hat{\otimes} C_2\oplus B_2\hat{\otimes} C_1,
\]
which are pairwise disjoint. So, each component of the boundary must vanish separately. Thus if $\gamma$ is a cycle it must actually be an element of $K_1\hat{\otimes} K_2$ up to the closure of the boundary. In particular, the map is indeed surjective.

\end{proof}

\subsection{The Kunneth formula for universal symplectic cohomology}\label{subsecKunnethSH}

\begin{comment}
\begin{tm}[\textbf{Kunneth formula}]\label{tmKunneth}
Let $M_1$ and $M_2$ be tame symplectic manifolds. Then
\begin{equation}\label{eqKunnethSH}
\widehat{SH}_{univ}^*(M_1\times M_2)=\widehat{SH}_{univ}^*(M_1)\hat{\otimes}_\iota\widehat{SH}_{univ}^*(M_2).
\end{equation}
Here the tensor product is the injective tensor product in the category of complete locally convex non-Archimedean topological vector spaces.
\end{tm}
\end{comment}
We shall need the following lemma. The author is grateful to Lev Buhovski for its proof.
\begin{lm}\label{lmSplitCofinal}
Let $M,N$ be smooth manifolds and let $P=M\times N$. The set of functions of the form $f\circ\pi_1+g\circ \pi_2$ is cofinal in $C^{\infty}(P)$.
\end{lm}
\begin{proof}
Take an exhaustion $ K_1 \subset K_2 \subset ... \subset M $ of $ M $ and an exhaustion $ L_1 \subset L_2 \subset ... \subset N $ of $ N $, by compact sets, and define positive locally bounded functions $ g_1 : M \rightarrow \mathbb{R} $, $ g_2 : N \rightarrow \mathbb{R} $ by
$ g_1(x) = \max_{K_i \times L_i} f $ and $ g_2(y) =  \max_{K_r \times L_r} f $, where $ i $ is the minimal positive integer such that $ x \in K_i $, and $ r $ is the minimal positive integer such that $ y \in L_r $. Then we have
$ f(x,y) \leqslant g_1(x)+g_2(y) $, for any $ (x,y) \in M \times N $. Now, since $ g_1 $ and $ g_2 $ are locally bounded, one can find smooth functions $ f_1 : M \rightarrow \mathbb{R} $ and $ f_2 : N \rightarrow \mathbb{R} $ such that
$ g_1(x) \leqslant f_1(x) $ for any $ x \in M $, and $ g_2(y) \leqslant f_2(y) $ for any $ y \in N $, and then we have
$ f(x,y) \leqslant f_1(x) f_2(y) $ for any $ (x,y) \in M \times N $.
\end{proof}

\begin{comment}
We shall also need the following technical lemma.
\begin{lm}
The injective tensor product and the reduced direct limit  commute with each other.
\end{lm}
\begin{proof}
The injective tensor product in a category $\mathcal{C}$ of topological vector spaces  is defined to be the unique up to isomorphism complete topological vector space in the same category space which satisfies the universal property of the tensor product with respect to bilinear maps which are separately continuous with respect to each variable. Let $\{A_i\}_{i\in I}$ and $\{B_j\}_{j\in J}$ be a pair of directed systems of objects in the category $\mathcal{C}$ with morphisms denoted by $a_{st}:A_s\to A_t$, and $b_{uv}:B_u\to B_v$ whenever $s\leq t\in I$ and $u\leq r\in J$. By the universal property for the reduced direct limit, for a fixed object $C\in\mathcal{C}$, a system $f_{ij}:A_i\times B_j\to C$ of separately continuous bilinear maps commuting with the $a_{st}$ and $b_{uv}$ induces a unique map $\varinjlim_{i\in I}A_i\times\varinjlim_{j\in J}B_j\to C$ which is bilinear and separately continuous. Thus, it induces a  continuous map
 \[
 \varinjlim_{i\in I}A_i\hat{\otimes}_\epsilon\varinjlim_{j\in J}B_j\to C.
 \]
 Conversely, such a map defines a system of separately continuous bilinear maps $A_i\times B_j$ by composing with the natural map
 \[
  \varinjlim_{i\in I}A_i\times\varinjlim_{j\in J}B_j\to \varinjlim_{i\in I}A_i\hat{\otimes}_\epsilon\varinjlim_{j\in J}B_j,
 \]
 and the maps $\{a_{st}\},\{b_{uv}\}$.
\end{proof}
\end{comment}

Before
\begin{proof}[proof of Theorem~\ref{tmKunneth}]
 By the discussion preceding Theorem~\ref{tmCohomologyKunneth} we have a natural map
\[
\varinjlim_{(H_1,H_2)\in \cH(M_1)\times\cH(M_2)}\overline{HF}^*(H_1;\Lambda_R)
\otimes \overline{HF}^*(H_2;\Lambda_R)\to SH_{univ}^*(M_1\times M_2;\Lambda_R).
\]
By Lemma~\ref{lmSplitCofinal} we can consider the right hand side as a direct limit over the same indexing set of split Hamiltonians. So an element of the right hand side lifts for some pair $(H_1,H_2)$ to an element of $\gamma\in HF^*(H_1\circ\pi_1+H_2\circ\pi_2)$. By Theorem~\ref{tmCohomologyKunneth} the image of the natural map
\[
HF^*(H_1)\otimes HF^*(H_2)\to HF^*(H_1\circ\pi_1+H_2\circ\pi_2)
\]
is sequentially dense. The density part of the claim follows.  Now suppose some element $x$ of the left hand side maps to $0$ in the right hand side. Then there is an $(H_1,H_2)$ such that the lift $\tilde{x}$ of $x$ to $\overline{HF}^*(H_1)\otimes \overline{HF}^*(H_2)$ maps to $0$ in $\overline{HF}^*(H_1+H_2)$. If follows from Theorem \ref{tmCohomologyKunneth} that $\tilde{x}=0$.% has norm $0$. Thus $x\in\overline{\{0\}}\subset SH^*_{univ}(M_1)\times SH^*_{univ}(M_2)$.
%Observe that for Banach spaces, which are a particular instance of Fr\'echet spaces, the injective and projective tensor products coincide. So, the right hand side of the last equation is
%\[
%\widehat{SH}_{univ}^*(M_1)\hat{\otimes}_\epsilon\widehat{SH}_{univ}^*(M_2).
%\]
%This proves the claim.
\end{proof}
\begin{cy}
Suppose $SH^*_{univ}(M_1)=\{0\}.$ Then $SH^*_{univ}(M_1\times M_2)=\overline{\{0\}}$.
\end{cy}

\subsection{Vanishing results}
\begin{tm}\label{tmVanishing}
Let $M$ be a geometrically bounded manifold such that $c_1(M)=0$. Suppose there exists a proper dissipative non-degenerate Hamiltonian on $M$ carrying no periodic orbits of index $0$. Then $SH^*_{univ}(M;\K)=0$.
\end{tm}
\begin{proof}
By definition, the natural map $H^*(M;\K)\to SH^*(M;\K)$ factors through $HF^*(M;\K)$. Since $HF^*(H;\K)=0$, we get from Theorem~\ref{tmFloerProduct}  that $SH^*_{univ}(M;\K)$ is a unital algebra in which $1=0$.
\end{proof}
\begin{lm}\label{lmCircFukTriv}
Let $M$ be a geometrically bounded manifold such that $c_1(M)=0$. The hypotheses of Theorem \ref{tmVanishing} are satisfied if $M$ carries a circle action $\psi_{\theta\in S^1}$ for which the following holds.
\begin{enumerate}
    \item It is generated by a Hamiltonian $H$ which is proper and bounded from below.
    \item There is an equivariant compatible geometrically bounded almost complex structure $J$ such that the distance $d(p,\psi_{1/2}(p))$ under $g_J$ is bounded away from $0$ outside of a compact set and such that $\|\nabla H\|_{g_J}\leq f(H)$ for some function $f:\R\to[1,\infty)$ for which the primitive of $\frac1{f}$ is unbounded from above.
    \end{enumerate}
\end{lm}
\begin{proof}
Assume that the flow of $H$ has minimal period $1$. Our assumptions imply that for any integer $k$, the function $(k+1/2)H$ is dissipative. Indeed invariance of $J$ under the flow implies the flow of $H$ is Killing. Thus, by Corollary \ref{cyKilling} and Lemma \ref{lmCompCrit}  the metric $g_{J_H}$ is geometrically bounded. The estimate on $d(p,\psi_{1/2}(p))$ implies loopwise dissipativity by Lemma \ref{lmLyap} and Corollary \ref{cyMinDisImpDis}. Indeed, in this case, the Lyapunov constant vanishes. Let $P\subset M$ be a connected component of the set of critical points of $H$. Then $P$ is compact and Morse Bott. Since the flow of $H$ is $1$-periodic orbit the Robbin-Salamon index $i_{RS}(p)$ for any $p\in P$ is related to the Morse index of $H$ by $2i_{Morse}(p)+\dim P=i_{RS}(p)+2n$. Suppose first that $i_{RS}(p)\neq 0$. The Robbin-Salamon index is additive with respect to concatenation and invariant under reparametrization. Thus, the absolute value of the Robbin-Salamon index of the critical points $p\in P$ can be made arbitrarily large by multiplying $H$ by a large enough constant. Suppose now $i_{RS}(p)=0$. Then $0\leq i_{Morse}(p)=n-\frac1{2}\dim P.$ We have $\dim P< 2n$ since the action is non-trivial. So, we can perturb $P$ and obtain fixed points with  Robbin-Salamon indices lying in $[-\dim P/2,\dim P/2]\subset(-n,n)$. Since the grading defined in eq. \ref{eqCZGrading}(for which the unit has degree $0$) is by $i_{RS}(p)+n$, we get that in either case for $k$ large enough there are no periodic orbits of index $0$.
\end{proof}
\begin{rem}
Lemma~\ref{lmCircFukTriv} has the curious implication that on a closed symplectic manifold $M$ with $c_1(M)=0$ there are no Hamiltonian circle actions. This is in fact proven in \cite{Ono88}.
\end{rem}
\begin{ex}\label{exTCy}
Let $M$ be a toric Calabi Yau manifold obtained as the symplectic reduction of $\C^N$ by a torus preserving the holomorphic volume. Then $M$ has an induced almost complex structure which preserves the action of the residual torus. With the induced Kahler metric, $M$ can be shown to have bounded geometry and the circle action given by the diagonal action
\[
    \theta\cdot[z_1,...,z_N]=[e^{i\theta}z_1,...,e^{i\theta}z_N],
\]
can be shown to satisfy the conditions of Lemma~\ref{lmCircFukTriv}. Thus $SH^*_{univ}(M)=0$. This generalizes the vanishing of the symplectic cohomology of $\C^n$ as well as the more general result of \cite{Ritter14} concerning the case where $M$ is total space of a negative line bundle over projective space and $c_1(M)=0$.
\end{ex}

\subsection{Existence of periodic orbits}\label{subsecNearby}
\begin{proof}[Proof of Theorem \ref{lmNearbyEx}]
Let $H:M\to\R$ be a proper smooth function such that $H^{-1}(-\infty,0)=K$. Suppose there is a $\delta>0$ such that the flow of $H$ on $H^{-1}(0,\delta)$ has no periodic orbits representing $\alpha$ in the first part or contractible in the second. We may assume without loss of generality that $H$ has sufficiently small Hessian everywhere so that the only periodic orbits are critical points. Let $h_n:\R\to\R$ be a monotone function constructed inductively so that $-\frac1{n}<h_n(x)<0$ for $x\in(-\infty,a]$, $h_n(x)=x+n$ on $(a+\delta/n,\infty)$ and $h_n(x)\geq h_{n-1}(x)$ everywhere. Let $H_n=h_n\circ H$. Note that by our assumption the only periodic orbits of $H_n$ are critical points, or, in the first part, periodic orbits not representing $\alpha$. We have that $H_n$ converges in a monotone way to $H_K$. So, by Lemma~\ref{lmchTruncIsom},
\[
SH^*(M|K;\K)=\widehat{SC}^{*}(\{H_n\})=\varprojlim_a\varinjlim_n HF^*_{[a,\infty)}(H_n;\K).
\]
The first part of the theorem now follows since the complex $\widehat{SC}^{*,\alpha}(\{H_n\})$ computing $SH^{*,\alpha}(M|K)$ is the zero complex. We prove the second part. We claim that for any $n$ any $-\infty<a<b$, and any $x\in HF^*_{[a,b)}(H_n;\K)$ supported on critical points lying outside of $K$, there is an $n'$ such that $x\mapsto 0$ in $HF^*_{[a,b)}(H_{n'};\K)$. Indeed, if we choose sufficiently generic time independent almost complex structures we may assume that  for any triple of integers $m,n_1,n_2$ we have that any simple Floer trajectory in the differential of $CF^*(\frac1mH_{n_i})$ or in the continuation map $CF^*(\frac1m H_{n_1})\to CF^*(\frac1m H_{n_2})$ is of the expected dimension. By a standard argument in Floer theory \cite{HoferSalamon} all the solutions are time independent. Namely,  since the Floer data and the asymptotic data are all time independent, a solution $\tilde{u}$ is either time independent as well, or part of a non-trivial $S^1$ family of solutions. In the latter case $\tilde{u}$  is an $m$-fold cover of a simple time dependent solution $u$ associated a Hamiltonians $\frac1{m}H_{n_i}$ which also appears in an $S^1$ family  and thus not of the expected dimension, contradicting the assumption.  Any time independent trajectory is gradient like for $H$. So if it emanates from a critical point outside of $K$ it remains outside of $K$. Moreover, the action difference for a continuation trajectory going from a critical point $x_1$ of $H_{n_1}$ to a critical point $x_2$ of $H_{n_2}$, both lying outside of $K$, is just
\[
-H_{n_2}(x_2)+H_{n_1}(x_1)<-(n_2-n_1).
\]
Thus, if $n_2-n_1>\val(x)+a$ then $x$ will map to $0$ in $HF^*_{[a,b)}(H_{n_2};\K)$. By similar reasoning, if $x$ is supported in $K$, it will map to itself under the obvious identification of critical points of $H_{n_i}$ with those of $H$. The claim follows.
\end{proof}
\begin{proof}[Proof of Theorem~\ref{tmNearbyEx0}\ref{NearbyEx1}]
For any compact set $K$, the map $SH^*_{univ}\to SH^*(M|K)$ is unital.  Moreover, the map $SH^*_{univ}\to SH^*(M|K)$ is continuous by Lemma \ref{lmSHHyoSHKCont}. Therefore, $\overline{\{0\}}$ maps to $0$ under this map. So the hypothesis implies $SH^*(M|K)=0$ for all $K\subset M$. The claim follows from Theorem \ref{lmNearbyEx}.
%Let $H$ be a Hamiltonian with a gap near $H^{-1}(0)$. By subtracting a small enough constant we may assume that the gap is of the of form $(0,\delta)$. Let $H_n$ be any Hamiltonian as in the proof of the Theorem~\ref{lmNearbyEx}. Then the map $\overline{HF}^*(H)\to SH^*(M|K)$ does not vanish.
%Considering the commutative triangle
%\[
%\xymatrix{
%5HF^*(H_n)\ar[d]\ar[dr]& \\
%\widehat{SH}^*_{univ}(M) \ar[r] &SH^*(M|K)
%}
%\]
%we obtain the claim.
\end{proof}

The proof of part \ref{NearbyEx2} of Theorem \ref{tmNearbyEx0} relies on the following lemma.
\begin{lm}\label{lmNearbyEx2}
Under the assumption of Theorem~\ref{tmNearbyEx0}\ref{NearbyEx2} we have that for any smooth $J$-proper Hamiltonian $H:M\to\R_+$ there is an $a\in\R_+$ such that the set of $x\in[a,\infty)$ for which $H^{-1}(x)$ has a periodic orbit representing $\alpha$ is dense in $[a,\infty)$.
\end{lm}
\begin{proof}
Suppose otherwise. Then there is a monotone increasing unbounded sequence $\{a_i\}$ such that for $x\in (a_{2i},a_{2i+1})$ the flow of $H$ has no periodic orbits representing $\alpha$. Fix a geometrically bounded almost complex structure $J$. For any $R>0$ we may assume without loss of generality that $B_R(H^{-1}(a_{2i-1}))\subset H^{-1}(a_{2i})$. Fix a constant $\epsilon$ and consider the set $\mathcal{E}$ of functions $h:\R_+\to\R_+$ such that $\|\nabla h\circ H\|<\epsilon$ outside  of the segments $(a_{2i},a_{2i+1})$. Then the set $\{h\circ H|h\in\mathcal{E}\}$ is $\preceq$-cofinal in $\mathcal{H}$. Taking $R$ large enough and epsilon small enough, the Floer data $(h\circ H,J)$ will be dissipative by the proof of Theorem \ref{tmDiaFinCofin}. Moreover, these compositions  have no periodic orbits representing $\alpha$. Thus $SH^{*,\alpha}(M)=0$ contradicting the assumption.
\end{proof}

\begin{proof}[Proof of Theorem \ref{tmNearbyEx0}\ref{NearbyEx2}]
Suppose otherwise. Then for any $K$ there is a proper Hamiltonian $H$ and real numbers $0<a<b$ such that $K\subset H^{-1}([0,a])$ and there are no periodic orbits in the interval $(a,b)$. Inductively choose an exhaustion by compact sets $K_i$, and exhaustion Hamiltonians $H_i$ with gaps $(a_i,b_i)$ such that for all $i$ we have
\[
K_i\subset H^{-1}([0,a_i])\subset H^{-1}([0,b_i])\subset K_{i+1},
\]
and
\[
a_i<b_i<a_{i+1}.
\]
Let $H$ be any proper Hamiltonian  which coincides with $H_i$ on $H_i^{-1}([a_i,b_i])$ and satisfies
\[
b_i<H(x)<a_{i+1},
\]
on the region
\[
\{H_i(x)>b_i\}\cap\{H_{i+1}<a_{i+1}\}.
\]
By taking a subsequence, we can assume further that $H$ is $J$-proper. There is no $a>0$ for which $H$ satisfies nearby existence on $[a,\infty)$ in contradiction to Lemma~\ref{lmNearbyEx2}.
\end{proof}

\bibliographystyle{amsabbrvc}
\bibliography{Ref}
\vspace{.5 cm}
\noindent
Institute of Mathematics\\
Givat Ram,\\
Jerusalem, 91904\\
Israel\\
\textit{Email address:}ygroman@gmail.com

\end{document}